\documentclass[a4paper,english]{smfart}
\usepackage[utf8]{inputenc}
\usepackage{amsfonts,amssymb}
\usepackage{amsmath}
\usepackage{graphicx}
\usepackage{array}
\usepackage{enumitem}
\usepackage{url}
\usepackage{bbm}
\usepackage{hyperref}
\usepackage{ifthen}
\usepackage{calc}
\usepackage{xargs}

\usepackage{empheq}
\numberwithin{equation}{section}
\usepackage{tikz}
\usetikzlibrary{cd,decorations.pathreplacing,calc,intersections,arrows,arrows.meta}
\tikzstyle{dot}=[shape=circle,draw,color=black,fill=black,inner sep=1pt]

\usepackage{amsthm}
\theoremstyle{plain}
\newtheorem{introtheorem}{Theorem}

\newtheorem{introcorollary}[introtheorem]{Corollary}
\newtheorem*{maintheorem}{Main theorem}
\newtheorem{theorem}{Theorem}[section]
\newtheorem{lemma}[theorem]{Lemma}
\newtheorem{corollary}[theorem]{Corollary}

\newtheorem{proposition}[theorem]{Proposition}

\newtheorem{conjecture}[theorem]{Conjecture}
\newtheorem{convention}[theorem]{Convention}

\theoremstyle{definition}
\newtheorem{definition}[theorem]{Definition}
\newtheorem*{definition*}{Definition}
\newtheorem{example}[theorem]{Example}
\newtheorem{notation}[theorem]{Notation}
\theoremstyle{remark}
\newtheorem{remark}[theorem]{Remark}
\newtheorem{keyremark}[theorem]{Key Remark}

\newcommand{\R}{\mathbb{R}}
\newcommand{\C}{\mathbb{C}}

\newcommand{\Z}{\mathbb{Z}}
\newcommand{\N}{\mathbb{N}}
\newcommand{\E}{\mathbb{E}}
\newcommand{\F}{\mathbb{F}}
\newcommand{\proba}{\mathbb{P}}
\newcommand{\truc}{\diamond}
\newcommand{\calO}{\mathcal{O}}
\newcommand{\LocalField}{F}

\newcommand{\cH}{\mathcal{H}}
\newcommand{\cK}{\mathcal{K}}

\newcommand{\SL}{\operatorname{SL}}
\newcommand{\GL}{\operatorname{GL}}
\newcommand{\CB}{\operatorname{CB}}
\newcommand{\VN}{\operatorname{VN}}

\DeclareMathOperator{\LL}{\mathrm L}
\newcommand{\defeq}{\mathrel{\mathop{:}}=}
\newcommand{\abs}[1]{\lvert #1 \rvert}
\newcommand{\dabs}[1]{\left\lvert #1 \right\rvert}
\newcommand{\norm}[1]{\lVert #1 \rVert}
\newcommand{\gen}[1]{\langle #1 \rangle}

\newcommand{\lk}{\operatorname{lk}}

\newcommand{\CAT}{\operatorname{CAT}}

\DeclareMathOperator{\Conv}{Conv}
\newcommand{\mmod}{\mathrel{\text{mod}}}
\newcommand{\iE}{\tilde{E}}
\newcommandx{\diE}[1][1={}]{\smash{\widetilde{E^*_{#1}}}}
\newcommand{\ud}{\mathop{}\!\mathrm{d}}
\newcommand{\cb}{\mathrm{cb}}

\newcommand{\typ}{\operatorname{typ}}

\renewcommand{\P}{\mathcal{P}}
\renewcommand{\L}{\mathcal{L}}
\newcommand{\Pa}{\mathcal{P}^{\mathrm{a}}}
\newcommand{\La}{\mathcal{L}^{\mathrm{a}}}

\newcommand{\posproj}[1]{v(#1)}
\newcommand{\negproj}[1]{w(#1)}

\newcommand{\llb}{(\mkern-2mu(}
\newcommand{\rrb}{)\mkern-2mu)}

\newcommand{\lseries}[1]{\llb #1 \rrb}

\newcommand{\plussim}{\mathrlap{\raisebox{.25ex}{$+$}}\raisebox{-.5ex}{$\sim$}}
\newcommand{\lsim}[1]{{}^\sim #1}
\newcommand{\lplus}[1]{ {}^+ #1}
\newcommand{\lmoins}[1]{{}^- #1}
\newcommand{\ltruc}[1]{{}^\truc #1}
\newcommand{\lun}[1]{{}^{(1)} #1}
\newcommand{\lplussim}[1]{{}^{\plussim} #1}

\author[J.~L\'ecureux]{Jean L\'ecureux}
\address{Université Paris-Saclay, CNRS, Laboratoire de mathématiques d\'Orsay, 91405, Orsay, France.}
\thanks{JL was funded by the ANR grant AGIRA ANR-16-CE40-0022}

\email{jean.lecureux@universite-paris-saclay.fr}

\author[M.~de~la~Salle]{Mikael de la Salle}
\address{Université de Lyon, CNRS, France}
\thanks{MdlS was funded by the ANR grants AGIRA ANR-16-CE40-0022 and Noncommutative analysis on groups and quantum groups ANR-19-CE40-0002-01}

\email{delasalle@math.univ-lyon1.fr}

\author[S.~Witzel]{Stefan Witzel}
\address{JLU Gießen, Mathematical Institute, Arndtstr. 2, 35392 Gießen}
\thanks{SW was funded by a Feodor Lynen Fellowship of the Humboldt Foundation and the the DFG Heisenberg grant WI 4079/6}

\email{switzel@math.uni-giessen.de}

\begin{document}

\title[Strong Property (T) in $\tilde{A}_2$-buildings]{Strong Property (T), weak amenability and $\ell^p$-cohomology in $\tilde{A}_2$-buildings}

\alttitle{Propriété (T) renforcée, moyennabilité faible et cohomologie $\ell^p$ dans les immeubles $\tilde{A}_2$}

\begin{abstract}
We prove that cocompact (and more generally: undistorted) lattices on $\tilde{A}_2$-buildings satisfy Lafforgue's strong property (T), thus exhibiting the first examples that are not related to algebraic groups over local fields. Our methods also give two further results. First, we show that the first $\ell^p$-cohomology of an $\tilde{A}_2$-building vanishes for any finite $p$. Second, we show that the non-commutative $L^p$-space for $p$ not in $[\frac 4 3,4]$ and the reduced $C^*$-algebra associated to an $\tilde{A}_2$-lattice do not have the operator space approximation property and, consequently, that the lattice is not weakly amenable. 
\end{abstract}

\begin{altabstract}
Nous montrons que les réseaux cocompacts (et plus généralement non distordus) dans les immeubles de type $\tilde{A}_2$ ont la propriété (T) renforcée de Lafforgue, ce qui fournit les premiers exemples qui ne proviennent pas de groupes algébriques sur des corps locaux. Les mêmes méthodes permettent d'obtenir d'autres résultats. D'une part nous montrons que, pour tout $p$ fini, la cohomologie $\ell^p$ d'un immeuble de type $\tilde{A}_2$ s'annule en degré $1$. D'autre part, nous montrons que l'espace $L^p$ non commutatif pour $p$ en dehors de l'intervalle $[\frac 4 3,4]$ ainsi que la $C^*$-algèbre réduite associés à un réseau dans immeuble de type $\tilde{A}_2$ n'ont pas la propriété d'approximation au sens des espaces d'opérateurs. Par conséquent, un tel réseau n'est pas faiblement moyennable.
\end{altabstract}


\altkeywords{Propriété (T) renforcée, immeubles, moyennabilité faible}

\maketitle

\section*{Introduction}

Property (T) was introduced by Kazhdan in order to study algebraic properties of lattices in simple Lie groups. It turned out to be a powerful tool in many areas of mathematics, from geometric group theory to ergodic theory or operator algebras. In \cite{Lafforgue08}, V. Lafforgue, with K-theoretic applications in mind, introduced a property called \emph{strong property (T)}. This property can be formulated as an extension of property (T) for representations on Hilbert spaces which are not necessarily unitary, but with a mild growth condition on the norm. Lafforgue also considered variants for representations on (some) Banach spaces, see Definition \ref{def:strongT} for a precise definition. Such non-unitary representations indeed appear in Lafforgue's work on the Baum-Connes conjecture, so strong property (T) is a natural obstruction for this approach to work for some higher-rank lattices, see \cite{MR2732057}. Further motivation for studying them include the study of actions on hyperbolic graphs \cite{Lafforgue08}, on Banach spaces \cite{MR2574023}, on manifolds \cite{brownFisherHurtado,brownFisherHurtado2,borwndamjanovicZhang}, or the study of the geometry of Banach spaces and expander graphs \cite{Lafforgue08,MR2574023}.

The main result of \cite{Lafforgue08} is that for a local field $\LocalField$, Archimedean or not, the group $\SL_3(\LocalField)$, and more generally every almost simple algebraic group over $\LocalField$ whose Lie algebra contains the Lie algebra of $\SL_3(\LocalField)$, as well as their cocompact lattices have strong property (T). For the Banach space variants, the results from \cite{Lafforgue08} were, in the non-Archimedean case, improved in \cite{MR2574023} to cover all Banach spaces with nontrivial Rademacher type, see \S~\ref{sec:Banach_space_preliminaries}. Such results are not known for Lie groups, despite some efforts \cite{MR3474958,LaatMimuradlS}. Lafforgue's results were extended in \cite{liao,strongTsp4} to all higher rank simple algebraic groups, and later to their non-cocompact lattices \cite{dlSActa19}. Strong property (T) played an important role in the recent proof of Zimmer's conjecture \cite{brownFisherHurtado,brownFisherHurtado2,borwndamjanovicZhang}. However, the only know examples so far are related to (lattices in) simple algebraic groups.

We prove strong property (T) for $\tilde{A}_2$-lattices, providing the first examples outside the realm of algebraic groups:

\begin{maintheorem}
Let $X$ be an $\tilde{A}_2$-building, and let $\Gamma$ be an undistorted lattice on $X$.
Then $\Gamma$ has strong property (T) with respect to every admissible Banach space. The class of admissible Banach spaces contains $L^p$ spaces and more generally subspaces of quotients of non-commutative $L^p$ spaces ($1 < p < \infty)$.
\end{maintheorem}

Recall that if $\LocalField$ is a non-Archimedean local field then $\SL_{n+1}(\LocalField)$ acts on a Bruhat--Tits $\tilde{A}_n$-building \cite{MR327923,MR756316}. Conversely every $\tilde{A}_n$-building with $n \ge 3$ arises this way (possibly with $\LocalField$ non-commutative) \cite{MR546588,MR2468338}. However, there are infinitely many isomorphism classes of $\tilde{A}_2$-buildings that are not Bruhat--Tits and that admit lattices (called \emph{$\tilde{A}_2$-lattices}) \cite[Section~10]{MR3904159}, \cite[Corollary~E]{MR3928791}. While arithmetic lattices in $\SL_n(\LocalField)$ are undistorted \cite{MR1828742}, we need to assume undistortedness in our result. Note, however, that all known non-arithmetic $\tilde{A}_2$-lattices are cocompact and therefore undistorted. 

These non-arithmetic $\tilde A_2$ lattices, while being geometrically analogous to lattices in $\SL_3(\LocalField)$, exhibit some algebraic properties which are quite different from them, and in fact, from any linear group. For example, it has been proven that any linear representation of a non-arithmetic, cocompact $\tilde{A}_2$-lattice has finite image \cite{MR3904159}. It is moreover conjectured that these groups are virtually simple.

The class of admissible Banach spaces is introduced in Section~\ref{sec:good_banach_spaces} and formally depends on the building. In particular, the complete statement about which Banach spaces we show to be admissible is Theorem~\ref{thm:Banach_spaces_with_exponential_decay}. We conjecture a Banach space is admissible if and only if it has nontrivial type (Conjecture~\ref{conj:nontrivial_type}), as it is the case for Bruhat-Tits buildings.

Strong property (T) immediately translates in terms of fixed points for affine actions \cite[Proposition 5.6]{MR2574023}:

\begin{introcorollary}
Let $\Gamma$ be an undistorted $\tilde{A}_2$-lattice.
Then for every admissible Banach space $E$ every action of $\Gamma$ by isometries on $E$ has a fixed point.
\end{introcorollary} 

In particular, we obtain the following generalization of \cite{Nowak}, which deals only with sufficiently  small $p>2$ (the precise condition depends on the building, but always implies $p\leq 2.106$) and small bounds on the uniformly bounded representation.

\begin{introcorollary}\label{cor:cohomology}
Let $\Gamma$ be an undistorted $\tilde{A}_2$-lattice.
For every uniformly bounded representation $\pi$ of $\Gamma$ on a $L^p$ space ($1<p<\infty$), or more generally a subspace of a quotient of a non-commutative $L^p$ space ($1<p<\infty$), one has $H^1(G,\pi) =0$.
\end{introcorollary}
Here a \emph{non-commutative $L^p$ space} is the $L^p$-space of a von Neumann algebra \cite{PisierXu}, see also \S~\ref{sec:OSpreliminaries}.

This is in sharp contrast with groups acting properly and cocompactly on bounded degree hyperbolic graphs, which admit proper actions on $L^p$ spaces for $p<\infty$ large enough \cite{MR2221161}, and proper actions on quotients of $L^p$ spaces for $p>1$ small enough \cite{MR3940900}. 

\medskip

Kazhdan's property (T) was already known to hold for cocompact $\tilde A_2$-lattices by the work of Cartwright and  M\l otkowski \cite{CartwrightMlotkowski94} and by Zuk \cite{Zuk} and Pansu \cite{Pansu} (this is generalized to non cocompact lattices in Theorem~\ref{thm:propertyT}). The proof of property (T) by \.Zuk and Pansu relies on a spectral estimate of links on vertices (see also the book \cite{BdlHV} for a nice exposition).  Our proof of strong property (T) is quite different and follows more closely Lafforgue's strategy for lattices in $\SL_3(\LocalField)$. Indeed, strong property (T) deals with representations with exponential growth rate at most some small number $\alpha$, that is, satisfying 
\[ \exists C, \forall g, \|\pi(g)\| \leq C e^{\alpha |g|}.\]
Therefore a proof of strong property (T) has to involve a local analysis at infinitely many different scales: local because the representation is a priori unbounded on the group, and at different scales because a representation with small exponential growth rate $\alpha$ but large constant $C$ cannot be distinguished from a representation with large exponential growth rate.

Lafforgue's strategy can be outlined as follows (see also \S\ref{sec:Lafforgue} below): first, he induces the representation (with small exponential growth rate) of $\Gamma$ to the ambient algebraic group $G=\SL_3(\LocalField)$. Then he defines a sequence of suitably chosen averaging operators on $G$, and is able to obtain the desired convergence by a clever use of harmonic analysis on a maximal compact subgroup $K$ of $G$. The first difference in our proof is that there is no ambient group $G$ to work with. We are therefore led to do a more "geometric" induction, to a space of functions on the building. Then main difficulty is then to replace the harmonic analysis on $K$ by a careful study of the (asymptotic) geometry of balls in the building.

It turns out that the tools used in the proof of the main theorem allow us to derive some other interesting properties of $\tilde A_2$-buildings and their lattices. The first is a vanishing result for $\ell^p$-cohomology of $\tilde{A}_2$-buildings. Recall that $\ell^p$ cohomology is a quasi-isometry invariant popularized by Gromov \cite{Gromov93}.

\begin{introtheorem}\label{thm:lp_cohomology}
Let $X$ be a locally finite $\tilde{A}_2$-building. Then $\ell^p H^1(X)=0$ for every $1<p<\infty$.
\end{introtheorem}

If $X$ is an $\tilde A_2$-building and $\Gamma$ is a cocompact lattice on $X$ then then $\ell^p H^1(X)$ is isomorphic to $H^1(\Gamma,\ell^p(\Gamma))$ \cite[Proposition~5]{MR3497258} and Theorem~\ref{thm:lp_cohomology} is a consequence of the main theorem. However, there are $2^{\aleph_0}$ $\tilde{A}_2$-buildings \cite{MR843395} while there are only countably many cocompact $\tilde{A}_2$-lattices. In fact, there is even a natural topology on the space of $\tilde A_2$ buildings for which one can prove that a generic building (of a fixed order) has a trivial automorphism group \cite{BarrePichot}.

The second application of our techniques concerns finite dimensional approximation properties of operator algebras, and their group counterparts such as weak amenability, see Section~\ref{sec:Banach_space_preliminaries} and the references therein for definitions and background. These approximation properties have been introduced and studied from the 80s after Haagerup's influential work \cite{Haagerup78} on the reduced $C^*$-algebra of the free group by various other mathematicians including de Cannières, Cowling, Kraus. They have turned to be very useful in the theory of operator algebras as they allow to distinguish some von Neumann algebras \cite{CowlingHaagerup} and more recently in Popa's deformation/rigidity theory, see for example \cite{OzawaPopa}. Our result in this direction generalizes \cite{Haagerup86,LafforguedlS} to $\tilde{A}_2$-buildings and answers a question by Ozawa \cite[p.~13]{OzawaSurvey}.

\begin{introtheorem}\label{introthm:AP}
Let $\Gamma$ be an $\tilde{A}_2$-lattice. Then $L^p(\VN( \Gamma))$ does not have the operator space approximation property for any $p \in (1,\infty) \setminus [\frac 4 3,4]$. For $p=\infty$, $C^*_{\mathrm{red}}(\Gamma)$ does not have the operator space approximation property, and $\VN(\Gamma)$ does not have the weak-* operator space approximation property.
\end{introtheorem}

A group $\Gamma$ is \emph{weakly amenable} if $C^*_{\mathrm{red}}(\Gamma)$ has the operator space approximation property. As a particular case of Theorem \ref{introthm:AP}, we get

\begin{introcorollary}
No $\tilde{A}_2$-lattice is weakly amenable.
\end{introcorollary}

The proofs of the three theorems above follow the same general strategy. A crucial role is played by operators $A_\lambda$ that average a function on the building over the points at vectorial distance $\lambda$; thus $\lambda$ may be thought of as a refined radius. These operators form a non-unitary representation of the algebra introduced and studied in \cite{CartwrightMlotkowski94,CartwrightMlotkowskiSteger}. A significant part of the proofs is concerned with showing that $(A_\lambda)_\lambda$ converges as $\lambda$ goes to infinity. This involves a careful analysis of the large-scale structure of the building. In the case of a lattice $\Gamma$ in $G = \SL_3(F)$ with maximal compact subgroup $K$, these operators can be understood in terms of harmonic analysis on maximal compact subgroups. Since we do not have such powerful tools at our disposal, we need to develop a more geometric argument.

In analysing the large-scale structure of the building we are led to studying two kinds of combinatorial substructures of the building. Fixing a base vertex $o$ leads to the study of \emph{Hjelmslev planes}, which encode balls around $o$ and are classical objects of study. Fixing an incident point-line pair $(p,\ell)$ in the projective plane at infinity leads to studying what we call a \emph{biaffine plane}, the subgeometry consisting of points not incident with $\ell$ and lines not incident with $p$. A careful combinatorial analysis of the interplay of these structures leads to the needed estimates on operator norms.

Another significant part in our analysis is to prove,  for various kinds of functions on the building, statements of the kind \emph{harmonic implies constant}. The simplest instance is Theorem~\ref{thm:lp_cohomology}, which is equivalent to saying that the only harmonic functions on $X$ whose gradient is $\ell^p$-integrable for some finite $p$ are the constant functions. The idea is to introduce three variants of the operators $A_\lambda$ discussed above, which depend on pairs of adjacent vertices: the operators $A_\lambda^\truc$ for $\truc \in \{+,-,\sim\}$. With the combinatorial analysis that we develop, we are able to obtain similar convergence results for these operators. The statements \emph{harmonic implies constant} are then derived from a crucial exploitation of the folding of the building. This part of the analysis serves as a replacement for the analysis of $K$-finite but not $K$-invariant matrix coefficients of representations of $\SL_3(\LocalField)$ when $X$ is a Bruhat-Tits building, and is one of the key novelties that we introduce.

The article is organized as follows. Buildings of type $\tilde{A}_2$ are recalled in Section~\ref{sec:buildings}. In Section~\ref{sec:outline} we give an outline of the proof of the main theorem for cocompact lattices, reducing it to the main technical statement, Theorem~\ref{thm:strong_T_on_building}, which determines the limit of the operators $A_\lambda$. In Section~\ref{sec:Banach_space_preliminaries} we introduce the results we need on Banach spaces and their operators. In Section~\ref{sec:hjemlslev_biaffine} we analyse the Hjelmslev planes and biaffine planes associated to an $\tilde{A}_2$-building by respectively fixing a base vertex $o$ in the building and a point--line-pair $(p,\ell)$ in its boundary. These are used in Section~\ref{sec:norm_estimates} to establish bounds on the norm of certain operators 
that, by definition, guarantee admissibility. Section~\ref{sec:averaging_basepoint} removes the dependency on $o$ from the results of the previous two sections. Sections~\ref{sec:convergence_harmonic} and~\ref{sec:harmonic_constant} contain the proof of Theorem~\ref{thm:strong_T_on_building}: the first is concerned with the convergence of the net of operators $(A_\lambda)_\lambda$, the second with determining its limit. At this point, the main theorem has been proven in the cocompact case. Section~\ref{sec:non-uniform} proves the main theorem for non-cocompact lattices. Theorem~\ref{thm:lp_cohomology} is proven in Section~\ref{sec:lp_cohomology} and Theorem~\ref{introthm:AP} in Section~\ref{sec:operator_algebras}.

\subsection*{Acknowledgements}
Part of the research for this article was carried out during the stay of the JL and MdlS at IMPAN during the Geometric and Analytic Group Theory Simons Semester in 2019, which was supported by the grant 346300 for IMPAN from the Simons Foundation and the matching 2015-2019 Polish MNiSW fund. Another part was carried out while SW visited École polytechnique on a Feodor-Lynen Fellowship of the Humboldt Foundation. SW would like to thank École polytechnique, and Bertrand Rémy in particular, for this opportunity.  We thank the referees for their numerous suggestions that improved the presentation.

\setcounter{tocdepth}{1}
\tableofcontents

\section{Buildings of type $\tilde{A}_2$}\label{sec:buildings}


Let us first recall a few things about $\tilde A_2$-buildings and their boundaries. 
For more information about buildings, the reader is advised to consult the book \cite{AbramenkoBrown}. 

\subsection{Projective planes and $\tilde A_2$-buildings}

A \emph{projective plane} consists of a set $P$ of \emph{points}, a set $L$ of \emph{lines}, and an \emph{incidence relation} $\mathop{\sim} \subseteq P \times L$ satisfying the following axioms (where $p \in P$ and $\ell \in L$ are \emph{incident} if $p \sim \ell$):
\begin{enumerate}
    \item for every pair of points $p,p'\in P$ there exists a unique line $\ell\in L$ incident with both $p$ and $p'$;
    \item for every pair of lines $\ell,\ell'\in L$ there exists a unique point $p\in P$ incident with both $\ell$ and $\ell'$;
    \item there exists at least two lines, and every line is incident to at least three points.
\end{enumerate}

Note that if $(P,L,\sim)$ is a projective plane, then its dual $(L,P,\sim)$ is again a projective plane.

A \emph{generalized triangle}, or \emph{building of type $A_2$}, is the incidence graph of a projective plane: its vertices are the elements of $P\cup L$ and its edges connect incident vertices. Following the standard graph-theoretical terminology, a coloring of a graph by a set $S$ (the colors) is a map from the vertex set to $S$ such that no pair of adjacent vertices edges have the same image in $S$. 

\begin{definition}
A \emph{building of type $\tilde A_2$} is a 2-dimensional simplicial complex $X$ that is simply connected, such that the link of every vertex is a generalized triangle, and that is equipped with a coloring of the vertices by $\Z/3\Z$. The color of a vertex $x$ is called its \emph{type} and denoted $\typ(x)$.
An \emph{apartment} in $X$ is any subcomplex that is isomorphic, as a simplicial complex, to the tiling of the Euclidean plane by regular triangles.
\end{definition}

\begin{remark}
For our purposes it will be convenient to consider the type map as part of the structure of the building. This is not customary in the literature but note that the coloring is determined by the simplicial structure of $X$ up to a permutation of $\Z/3\Z$ (see e.g. \cite[Corollary 4.8]{AbramenkoBrown}).
\end{remark}

\begin{definition}
A simplicial map $\phi \colon Y \to Z$ between simplicial complexes $Y,Z$ whose vertices are colored by $\typ \colon Y,Z \to \Z/3\Z$ is said to be \emph{type-preserving} if $\typ(\varphi(x)) = \typ(x)$ for every vertex $x$. It is \emph{type-rotating} if there is an $i \in \Z/3\Z$ such that $\typ(\varphi(x)) = \typ(x) + i$ for all vertices $x$.
\end{definition}

If a group $G$ acts by simplicial automorphisms on a building $X$ of type $\tilde A_2$, then it induces a permutation on the set $\Z/3\Z$ of types. Therefore we get:

\begin{lemma}\label{lem:type-preserving}
Let $G$ be a group acting by simplicial automorphisms on an $\tilde A_2$-building. Then $G$ has a finite index subgroup which is type-preserving.
\end{lemma}

Every automorphism $\phi$ of a building $X$ of type $\tilde{A}_2$ induces an (anti-)automorphism $\phi_\infty$ of its projective plane at infinty (see Section~\ref{sec:boundary}). If $\phi$ is type-rotating then $\phi_\infty$ is an automorphism, i.e.\ it takes points to points and lines to lines; if $\phi$ is not type-rotating then $\phi_\infty$ is an anti-automorphism, i.e.\ it takes points to lines and lines to points. We therefore refer to type-rotating maps also as \emph{type-preserving at infinity}.

\begin{definition}
We call the \emph{dual building} $\overline{X}$ of $X$ the building that is isomorphic to $X$ as a simplicial complex but has the types $1$ and $2$ interchanged.
\end{definition}

\begin{keyremark}
 Our results analyzing balls in $X$ will generally be relative to some points and lines in the projective plane at infinity. Since the role of points and lines is formally symmetric (even though the projective plane at infinity may not admit an anti-automorphism), for each statement there is a different, dual statement which is obtained by ``exchanging the roles of points and lines''. This is made explicit as follows.
Once we have proven a result on general $\tilde A_2$ buildings, it applies to both the building $X$ and its dual $\overline X$. Thus we obtain the dual statement for $X$ by applying it to $\overline X$ and applying the identity $X \to \overline X$. Since identity map $X \to \overline X$ is an isomorphism that is not type-rotating, it takes points at infinity of $X$ to lines at infinity of $\overline{X}$ and vice versa.
\end{keyremark}

%

\subsection{Convex hulls and apartments}

For the rest of the paragraph we fix an $\tilde{A}_2$-building $X$. It carries a unique metric in which edges have unit length that makes it a $\CAT(0)$-space. We denote the geodesic segment connecting two points $x,y\in X$ by $[x,y]$. An apartment is a subspace $\Sigma$ that is isometric to the Euclidean plane with respect to this metric. The \emph{combinatorial convex hull} of two vertices $x,y \in \Sigma$, denoted $\Conv(x,y)$, is the least convex subcomplex of $\Sigma$ that contains $x$ and $y$. Note that it contains the metric convex hull. Every apartment is convex in the following strong sense showing, in particular, the combinatorial convex hull is independent of the apartment.

\begin{proposition}[{\cite[Proposition~4.40]{AbramenkoBrown}}]\label{thm:apmt_convex}
Let $x, y \in X$ be vertices and let $\Sigma$ and $\Sigma'$ be apartments containing $x$ and $y$. Then $\Sigma'$ contains the combinatorial convex hull of $x$ and $y$ in $\Sigma$.
\end{proposition}

A convex subset of the Euclidean plane can be embedded into $X$ by embedding it into some apartment. In fact, any embedding arises this way:

\begin{theorem}[{\cite[Theorem 11.53]{AbramenkoBrown}}]\label{thm:isomappart}
Let $X$ be an $\tilde A_2$-building. Let $Y\subset X$ be a set that is isometric to a subset of the Euclidean plane. If $Y$ is convex in $X$ or if $Y$ has non-empty interior in $X$ then there is an apartment of $X$ containing $Y$.
\end{theorem}

The following is well-known:

\begin{lemma}\label{lem:order}
In a building of type $\tilde{A}_2$ there is a cardinality $q$ such that the number of triangles around every edge is $q+1$.
\end{lemma}

\begin{proof}
The link of each vertex is a building of type $A_2$. For the edges around a fixed vertex the claim therefore follows from \cite[Corollaries~5.117 and~5.118]{AbramenkoBrown}. It follows in general since the edge graph of $X$ is connected.
\end{proof}

In general, $q$ can be finite of infinite. When $q$ is finite (which we will assume from now on), that is, when $X$ is locally finite, the number $q$ is called its \emph{order}.

Let $\Sigma$ be an apartment. A \emph{wall} of $\Sigma$ is a (Euclidean) line formed by edges of the tessellation. A \emph{half-apartment} of $\Sigma$ is one of the two connected component of $\Sigma $ minus a wall.
If we fix an origin $o \in \Sigma$, the set of walls passing through $o$ divide $\Sigma$ into $6$ connected components that are called \emph{Weyl chambers}.
More generally, a \emph{wall} (resp. \emph{half-apartment},\emph{Weyl chamber}) of an $\tilde A_2$-building is a wall (resp. half-apartment, Weyl chamber) of any of its apartments.

The following lemma allows us to prove that some sets are contained in apartments whithout having to check the distances beforehand.

\begin{lemma}\label{lem:extensionroot}
Let $Y\subset X$ be a half-apartment, bounded by some wall $H$. Let $C$ be a simplex with an edge belonging to $Y$. Then $Y\cup C$ is contained in an apartment.
\end{lemma}

\begin{proof}
This is \cite[Exercice 5.83a)]{AbramenkoBrown}.
\end{proof}

\subsection{Bruhat--Tits buildings and exotic buildings}\label{sec:brutit}

Classically, the interest in Euclidean buildings lies in fact that they arise as the \emph{Bruhat--Tits buildings} associated to reductive groups over (non-Archimedean) local fields (that is, locally compact, totally disconnected fields). Specifically there is a locally finite building of type $\tilde{A}_2$ associated to $\SL_3$ over a local field. We briefly describe these. Let $\LocalField$ be a non-Archimedean local field with ring of integers $\calO$ and let $\pi \in \calO$ be a uniformizing element. Let
\[
r \defeq \begin{pmatrix}
0 & \pi & 0\\
0 & 0 &1\\
1& 0 & 0
\end{pmatrix}
\in \GL_3(\LocalField)\text{.}
\]
The group $G \defeq \SL_3(\LocalField)$ has three conjugacy classes of maximal compact subgroups represented by
\[
K_0 \defeq K \defeq \SL_3(\calO)\text{,}
\]
$K_1 \defeq K^r$ and $K_2 \defeq K^{r^2}$. In the building $X$ associated to $G$ the group $K_i$ is the stabilizer of a vertex of type $i$. The group $G$ acts transitively on the vertices of each type, so the vertex set may be identified with $G/K_0 \cup G/K_1 \cup G/K_2$. Similarly, the groups $K_i \cap K_j$ are edge stabilizers and $K_0 \cap K_1 \cap K_2$ is a triangle-stabilizer. The order of $X$ is the size of the field $\calO/\pi\calO$. This construction is possible even if $\LocalField$ is a finite dimensional division algebra over a local field.

However, these are not the only locally finite buildings of type $\tilde{A}_2$. It has been long known that, up to isomorphism, there exist infinitely many $\tilde{A}_2$-buildings that are not Bruhat--Tits \cite{MR843395,MR908973} and also that such buildings exist with cocompact automorphism group \cite{MR843390,MR1232965,MR1232966}. The verification that there exist, up to isomorphism, infinitely many $\tilde{A}_2$-buildings that are not Bruhat--Tits and have a cocompact automorphism group is more recent \cite[Section~10]{MR3904159}, \cite[Corollary~E]{MR3928791}.

\subsection{The boundary}\label{sec:boundary}

Let $X$ be a building of type $\tilde A_2$. One can associate to $X$ another projective plane, called its \emph{building at infinity}, which we define now.

\begin{definition}
A \emph{singular ray} of an apartment is a half-line which is contained in a wall. A \emph{singular ray} of $X$ is the singular ray of some apartment. The \emph{type} of a singular ray is the cyclic order in which the type of vertices appear. Two singular rays (of $X$) are \emph{equivalent} if they contain subrays which are parallel rays in some apartment.
\end{definition}

There are two possible types of singular rays: $(0,1,2)$ or $(0,2,1)$. Note that two equivalent rays have the same type. Every Weyl chamber is bounded by by two singular rays, one from each type.

\begin{definition}
The \emph{projective plane at infinity} of $X$ is the projective plane ${(\P,\L,\sim)}$ defined as follows:
\begin{enumerate}
    \item Elements of $\P$ are equivalent classes of singular rays of type $(0,1,2)$,
    \item elements of $\L$ are equivalent classes of singular rays of type $(0,2,1)$,
    \item two classes of rays are incident if they contain rays that form the boundary of some Weyl chamber in $X$.
\end{enumerate}
The \emph{building at infinity} of $X$ is the incidence graph of this projective plane.
\end{definition}

The two following propositions are well-known:

\begin{proposition}
The triple $(\P,\L,\sim)$ is a projective plane.
\end{proposition}

\begin{proposition}
For every singular ray $\rho$, and every origin $o\in X$, there is a unique singular ray $\rho'$ which starts from $o$ and is equivalent to $\rho$.
\end{proposition}

This allows us to identify $\P\cup \L$ with the set of singular rays starting at any fixed origin $o$. If $\xi \in \P\cup\L$, the singular ray starting from $o$ in the class of $\xi$ is also called the \emph{geodesic ray} from $o$ to $\xi$ and can be written $[o,\xi)$. This is justified by the fact that $\P\cup \L$ is a subset of the CAT(0) boundary $\partial_\infty X$ of $X$.

Note that the boundary of $\overline X$ is the projective plane $(\overline \P,\overline \L)=(\L,\P)$ with the same incidence relation. 

\subsection{Combinatorial distance and spheres}

Let $\E$ be a two-dimensional Euclidean vector space and let $\Lambda \subseteq \E$ be a linear simplicial cone with an opening angle of $60^\circ$. As a basis of $\E$ we choose the unit vectors in the bounding rays, and we denote by $|\cdot|_1$ the $\ell^1$-norm with respect to this basis. This the bounding rays are $[0,\infty) \times \{0\}$ and $\{0\} \times [0,\infty)$ and $\Lambda = \{(x,y) \mid x,y \ge 0\}$. We regard $\Lambda$ as tesselated by regular triangles with vertices in $\N \times \N$. Moreover, the type of $(s,t) \in \N \times \N$ is defined to be $s + 2t + 3\Z \in \Z/3\Z$. Thus for $i\in \{0,1,2\}$, we have the sets
\[
\Lambda_i = \{(s,t) \in \Lambda | s \equiv t+i \mmod 3\}
\]
of vertices of type $i$ in $\Lambda$.

Let $X$ be an $\tilde A_2$-building. If $x \in X$ is a vertex and $W$ is a Weyl chamber with tip $x$, there is a unique simplicial isomorphism $\rho_{x,W} \colon W \to \Lambda$ such that $\typ(\rho(y)) = \typ(y) - \typ(x)$ for every vertex $y$. By Proposition~\ref{thm:apmt_convex} these maps for arbitrary $W$ fit together to define a map $\rho_x \colon X \to \Lambda$. Note that it takes rays issuing at $x$ tending toward elements of $\P$ to $[0,\infty) \times \{0\}$ while it takes rays issuing at $x$ tending toward elements of $\L$ to $\{0\} \times [0,\infty)$.

\begin{definition}
The \emph{combinatorial distance} between two vertices $x,y\in X$ is $\sigma(x,y) = \rho_x(y) \in \Lambda$.
\end{definition}

Note that if $\sigma(x,y)=(s,t)$ then the combinatorial convex hull of two vertices $x$ and $y$ is a (possibly degenerate) parallelogram whose sides have length $s$ and $t$.

If $\lambda \in \Lambda$ and $x \in X$ are vertices, we define the \emph{$\lambda$-sphere} around $x$ to be
\begin{equation*}
S_\lambda(x) = \{y \in X \mid \sigma(x,y)=\lambda\}.
\end{equation*}

\begin{notation}\label{notation:bar}
Every set or quantity attached to $X$ which depends on the choice of types can be considered also in $\overline X$, and we will denote it by using a $\overline{\phantom x}$ (but we identify the set of vertices of $X$ and $\overline X$). For $\lambda=(i,j)\in \Lambda$ we also denote $\overline \lambda=(j,i)$. For example, $\overline \sigma(x,y)$ is the $\sigma$-distance between $x$ and $y$ in $\overline X$. Similarly we can talk about $\overline S_\lambda(x) \;(=S_{\overline \lambda}(x))$ and $\overline \P=\L,\overline \L=\P$, etc.
\end{notation}

\begin{lemma}\label{lemma:counting}
Let $\lambda_1,\lambda_2 \in \Lambda$ be vertices. Then for vertices $x_1,x_2 \in X$, the cardinality of $S_{\lambda_1}(x_1) \cap S_{\lambda_2}(x_2)$ depends only on $\sigma(x_1,x_2)$.
\end{lemma}

\begin{proof}
This is a particular case of \cite[Theorem 5.21]{Parkinson06}. The reference has a standing assumption of \emph{regularity} which refers to the fact that $X$ has an order by Lemma~\ref{lem:order}.
\end{proof}

For the rest of the article we will not have occasion to consider points of buildings other than vertices. We therefore make the following notational

\begin{convention}
The expressions $x \in X$ and $\lambda \in \Lambda$ are understood to refer to vertices throughout.
\end{convention}

\section{Outline of proof}\label{sec:outline}

Throughout the article, $X$ will denote a locally finite building of type $\tilde{A}_2$ and of order $q$ and $\Gamma$ will be group (possibly trivial) acting properly on $X$. The action of $\Gamma$ will be assumed to be cocompact in Sections~\ref{sec:convergence_harmonic} and~\ref{sec:harmonic_constant} while in Sections~\ref{sec:non-uniform} and~\ref{sec:operator_algebras} it is only assumed to be an undistorted lattice on $X$ that may not be cocompact (see definitions in Section~\ref{sec:non-uniform}). Finally, in Section~\ref{sec:lp_cohomology}, the group $\Gamma$ does not appear at all and $X$ is not assumed to have any non-trivial automorphism.

\subsection{Strong property (T)}
One of the characterizations of property (T) for a locally compact group $G$ is that the maximal $C^*$ algebra $C^*(G)$ of the group $G$ contains a projection $P$ such that $\pi(P)$ is a projection onto the space of invariant vectors for every unitary representation $\pi$. In that case, $P$ is unique and is called the \emph{Kazhdan projection}. When one no longer works with unitary representations, many characterizations of property (T) make sense, but they do not remain equivalent. When defining strong property (T) (resp. its Banach space variant) in \cite{Lafforgue08}, Lafforgue chose to adapt this characterization in terms of Kazhdan projections, but allowing more generally representations with small exponential growth on a Hilbert space (resp. some Banach space).

We will only consider finitely generated groups, and for them we will work with the following definition of strong property (T), which is more convenient.

\begin{definition}\label{def:strongT}
  Let $\Gamma$ be a finitely generated group with word length $|\cdot |$ (associated to a fixed generating set) and $\mathcal E$ a class of Banach spaces. $\Gamma$ has \emph{strong property (T)} with respect to $\mathcal E$ if there is $\alpha>0$ and a sequence $\mu_n$ of finitely supported probability measures on $\Gamma$ such that, for every representation of $\Gamma$ on a Banach space $E \in \mathcal E$ satisfying $\sup_{\gamma \in \Gamma} \frac{\|\pi(\gamma)\|_{B(E)}}{e^{\alpha |\gamma|}}<\infty$, the sequence $\pi(\mu_n) = \sum_\gamma \mu_n(\gamma) \pi(\gamma) $ converges in norm to a projection on the space $E^\pi$ of $\pi(\Gamma)$-invariant vectors.
\end{definition} 
\begin{remark}
This definition is formally a bit stronger than Lafforgue's original definition, which corresponds to allowing $\mu_n$ to be complex measures, but no example is known that satisfy the original form but not this one;  see for example \cite[Section 3]{MR4014781} for a detailed explanation of the relation with the original definition. It is in this (stronger) form that strong property (T) was applied to the study of group actions on manifolds \cite{brownFisherHurtado,brownFisherHurtado2,borwndamjanovicZhang}. We therefore adopt this definition in this paper.
\end{remark}
In \S~\ref{subsection:outline_nonuniform}, we will construct measures $\mu_n$ that are not finitely supported but merely have good exponential integrability properties. The following lemma shows that this implies strong property (T).
\begin{lemma}\label{lem:truncation}
 Let $\Gamma$ be a finitely generated group with word-length $|\cdot|$ and $\mathcal E$ a class of Banach spaces. Assume that there is $\alpha>0$ and a sequence $\mu_n$ of  probability measures on $\Gamma$ such that
\[ \sum_\gamma \mu_n(\gamma) e^{\alpha |\gamma|}<\infty,\]
and for every representation of $\Gamma$ on a Banach space $E \in \mathcal E$ satisfying $\sup_{\gamma \in \Gamma} \frac{\|\pi(\gamma)\|_{B(E)}}{e^{\alpha |\gamma|}}<\infty$, the sequence $\pi(\mu_n) = \sum_\gamma \mu_n(\gamma) \pi(\gamma) $ converges in norm to a projection on the space $E^\pi$ of $\pi(\Gamma)$-invariant vectors.

Then $\Gamma$ has strong property (T).
\end{lemma}
\begin{proof} This follows by a straightforward truncation argument. Indeed, there is a sequence $R(n)$ such that $\lim_n \sum_\gamma 1_{|\gamma|>R(n)} \mu_n(\gamma) e^{\alpha|\gamma|}=0$. Then the truncated measure $\mu'_n =\mu(\cdot \cap B_{R(n)})/\mu(B_{R(n)})$ satisfies the definition of strong property (T), because $\|\pi(\mu_n) - \pi(\mu'_n)\| \to 0$.
\end{proof}
\subsection{Lafforgue's proof for Bruhat-Tits buildings}\label{sec:Lafforgue}
We briefly recall Lafforgue's proof of the main theorem when $\Gamma$ is a cocompact lattice in a Bruhat-Tits building associated to a $\SL_3$ over a local field $\LocalField$. By an induction argument, the theorem is equivalent to strong property (T) for the locally compact group $G=\SL_3(\LocalField)$. The maximal compact subgroup $K=\SL_3(\calO)$ then plays a crucial role. The main technical result asserts that $K$-finite matrix coefficients of representations of $\SL_3(\LocalField)$ with small exponential growth have a limit with an explicit rate of convergence. We state it in a sub-optimal form, but still enough to deduce strong property (T) for $G$. Here, if $V$ is a finite dimensional representation of a compact group $K$ and $\pi$ is an arbitrary representation of $K$ on $E$, a vector $\xi \in E$ is said to be of $K$-type $V$ if the linear span of $\pi(K) \xi$ is isomorphic to a subrepresentation of $V^{\oplus d}$ for some integer $d$.
\begin{theorem}\label{thm:Lafforgue_s_strategy}\cite{Lafforgue08}
Let $\pi$ be a representation of $G=\SL_3(\LocalField)$ with small enough exponential growth on a Hilbert space $E$. For every finite dimensional representation $V$ of $K=\SL_3(\calO)$, there is a function $\varepsilon_V \in C_0(G)$ such that for any unit vectors $\xi,\eta \in E$ of $K$-type $V$, there is $c_\infty(\xi,\eta) \in \C$ such that
\[ |\langle \pi(g)\xi,\eta\rangle - c_\infty(\xi,\eta)| \leq C(V) \varphi(g).\]
\end{theorem}
When $\xi,\eta$ are $K$-invariant (that is $V$ is the trivial representation), the theorem asserts that the operators $\iint_K \pi(k gk') dk dk'$ (average with respect to the Haar measure on $K$) are Cauchy when $g \to \infty$, and therefore converge. The case of general $V$ allows to identify the limit as a projection on the space of invariant vectors, as in Definition~\ref{def:strongT}, adapted to non-discrete group.

Theorem~\ref{thm:Lafforgue_s_strategy} is obtained by combining two distinct ingredients. One way of expressing the first is as a form of Theorem~\ref{thm:Lafforgue_s_strategy} but for the pair $(K,G)$ replaced by the pair $(U,K)$, where $U = \begin{pmatrix} 1 & 0 \\ 0& \SL_2(\calO)\end{pmatrix}$.
\begin{proposition}\label{prop:lafforguesStrategy_compact}
For every finite dimensional unitary representation $V$ of $U$, there is a constant $C(V)$ such that for every unitary representation $\pi$ of $K$ on $E$, every vector $\xi,\eta$ of $U$-type $V$ and every $\delta \in \calO$,
\[ | \langle \pi(k_\delta) \xi,\eta\rangle - \langle \pi(k_0) \xi,\eta\rangle| \leq C(V) \sqrt{|\delta|}\]
where $k_\delta = \begin{pmatrix} \delta & 1 & 0 \\ -1&0&0\\0&0&1\end{pmatrix}$.
\end{proposition}
The second ingredient can be expressed as an efficient exploration of the Weyl chamber $K\backslash G/K$ by the images of $\{k_\delta \mid \delta \in \calO\}$ in $U\backslash K/U$.

\subsection{Strategy of the proof}\label{sec:strategy}

We now explain the proof of the main theorem. We first consider the case of a cocompact lattice, because it is the most interesting case and also because it is technically easier. In this case the classical Milnor--Schwarz lemma asserts the following:

\begin{lemma}\label{lem:orbit_map_quasiisometric}
Let $X$ be an $\tilde{A}_2$-building and let $\Gamma$ act properly and cocompactly on $X$. Let $\Omega$ be a set of representatives $\Gamma \backslash X$.
There exist positive real numbers $a,b$ such that, for every $x \in \Omega$ and $y \in X$, there is $\gamma \in \Gamma$ such that $\gamma y \in \Omega$ and $|\gamma| \leq a d(x,y)+b$.
\end{lemma}

The proof of the main theorem in the cocompact case will be completed in Section~\ref{sec:harmonic_constant}. The necessary modifications for non-cocompact (but undistorted) lattices are explained in Section~\ref{sec:non-uniform}.

The main difference between our setup and that of existing proofs of strong property (T) is that there is no algebraic group over a local field $G$ that contains $\Gamma$. In fact, the automorphism group of an exotic building $X$ is discrete in most known cases, see \cite{MR3719076}, \cite{MR3928791}. It turns out that the building $X$ can serve as a replacement for $G$. This is illustrated by the following induction, that imitates induction of representations from $\Gamma$ to $G$. Actually, this induction is the perfect analogue of the $K$-invariant vectors for the induced representation of $G$, where $K$ is a maximal compact subgroup of $G$, cf. Section~\ref{sec:brutit}. The fact that there is no analogue of the full induced representation is one of the main difficulties that we will have to overcome, mainly in Section~\ref{sec:harmonic_constant}.

Given a representation $\pi$ of $\Gamma$ on a Banach space $E$ we define
\[
\iE = \{f \colon X \to E \mid \forall \gamma \in \Gamma, x \in X, f(\gamma x) =\pi(\gamma) f(x) \}
\]
and equip it with the norm $\norm{f} = \left(\sum_{x \in \Omega} \frac{1}{\abs{\Gamma_x}}\norm{f(x)}^2\right)^{1/2}$, where $\Gamma_x$ is the stabilizer in $\Gamma$ of the point $x$. Here and throughout $\Omega$ is a set of orbit representatives for $\Gamma \backslash X$. Note that the norm on $\iE$ depends on $\Omega$ but the topology does not. Note also that $\iE$ is isometric to the $\ell^2$-direct sum $\bigoplus_{x \in \Omega} E^{\Gamma_x}$ via the isomorphism
\begin{align}
\bigoplus_{x \in \Omega} E^{\Gamma_x} &\to \iE\label{eq:etilde_decomposition}\\
(\xi_x)_{x \in \Omega} & \mapsto (\gamma.x \mapsto \abs{\Gamma_x}^{\frac 1 2} \pi(\gamma)(\xi))\text{.}\nonumber
\end{align}

Moving on with the proof, there is a small reduction step. Recall from Lemma \ref{lem:type-preserving} that $\Gamma$ has a finite index subgroup $\Gamma'$ that preserves types. If we prove that $\Gamma'$ has strong property (T) then it will also follow for $\Gamma$ (for the same family of spaces). Therefore we can and will assume throughout that $\Gamma=\Gamma'$ acts on $X$ preserving types. This also implies that the action on $\partial X$ preserves types. Also recall that $\Lambda_i$ denotes the subset of $\Lambda$ of type $i$, for $i\in \{0,1,2\}$, and in particular that $\Lambda_0$ is the set of vertices of $\Lambda$ of the same type as $0$.

For every $\lambda \in \Lambda$ we define the linear map $A_\lambda$ on the space of all functions $X \to E$ by
\begin{equation}\label{eq:def_Alambda}
A_\lambda f(x) = \frac{1}{|S_\lambda(x)|} \sum_{y \in S_\lambda(x)} f(y)\text{.}
\end{equation}
Note that $A_\lambda$ maps $\iE$ to itself because the action of $\Gamma$ on $ X$ is type-preserving. The operators $A_\lambda$ form a (non unitary) representation of the commutative algebra $\mathcal A$ of \emph{biradial operators} associated to an $\tilde{A}_2$-building in \cite{CartwrightMlotkowski94}. The key statement in the proof, which we will verify by Section~\ref{sec:harmonic_constant} is the following:

\begin{theorem}\label{thm:strong_T_on_building}
Assume that $\Gamma$ is cocompact and that $E$ is admissible. There is $\alpha_0$ depending on $\Gamma$ and $E$ such if the representation $\pi$ of $\Gamma$ on $E$ satisfies $\sup_{\gamma} e^{-\alpha_0 |\gamma|} \|\pi(\gamma)\|<\infty$, as $\lambda \in \Lambda_0$ tends to infinity, the sequence $(A_\lambda)_{\lambda}$ converges in the norm of $B(\iE)$ towards a projection $P$ on the subspace of functions in $\iE$ which are constant on the vertices of each type. Moreover the convergence is exponentially fast : $\|A_\lambda - P\| = O( e^{-\alpha_0 |\lambda|_1})$.
\end{theorem}

The main theorem follows readily from the theorem in the cocompact case. Specifically, the sequence of measures can be taken to be $\mu_\lambda$ where $\lambda \in \Lambda$ tend to infinity and $\mu_\lambda$ is defined as follows. For $y \in X$, denote by $m_{y}$ the uniform probability measure on the finite set of all $\gamma \in \Gamma$ such that  $y\in \gamma\Omega$. Then
\begin{equation}\label{eq:def_mulambda}
\mu_\lambda \defeq \frac{1}{Z} \sum_{x \in \Omega} \frac{1}{\abs{\Gamma_x}} \frac{1}{\abs{S_\lambda(x)}} \sum_{y \in S_\lambda(x)} m_{y}\text{,}
\end{equation}
where $Z = \sum_{x \in \Omega} \frac{1}{\abs{\Gamma_x}}$.

To see this we introduce two linear maps $u \colon E \to \iE$ and $v \colon \iE \to E$ defined by $u(\xi)(y)=\pi(m_y)(\xi)$ for $y \in X$ and $v(f) = \frac{1}{Z} \sum_{x \in \Omega} \frac{1}{\abs{\Gamma_x}}f(x)$ for $f \in \iE$.

These definitions are made so that
\begin{equation}\label{eq:mu_lambda}
v \circ A_\lambda \circ u = \pi(\mu_\lambda)\text{.}
\end{equation}

\begin{proof}[Proof of the main theorem for cocompact $\Gamma$]
Let $P$ be the projection from Theorem~\ref{thm:strong_T_on_building}.
We claim that the operator $v \circ P \circ u \colon E \to E$ is a projection onto the space $E^\Gamma$ of $\pi(\Gamma)$-invariants. To see this note first that if $\xi \in E^\Gamma$ then $u(\xi) \colon X \to E$ is constant. As a consequence $P \circ u(\xi) = u(\xi)$ and $v \circ u(\xi) = \xi$. Thus $v \circ P \circ u$ restricts to the identity on $E^\Gamma$. Second, note that every $f \colon X \to E$ that lies in the image of $P$ (i.e.\ is constant on vertices of each type) has values in $E^\Gamma$: for $\pi(\gamma) f(x) = f(\gamma x) = f(x)$ as $x$ and $\gamma x$ have the same type by our standing assumption. In particular, $v(f) \in E^\Gamma$. Thus the image of $v \circ P \circ u$ lies in $E^\Gamma$ and the claim is proved.

Now Theorem~\ref{thm:strong_T_on_building} implies that $v \circ A_\lambda \circ u$ converges in norm to $v \circ P \circ u$ as $\lambda \in \Lambda_0$ tends to infinity:
\[
\norm{v \circ (A_\lambda - P) \circ u}_{B(E)} \leq \norm{v}_{B(\iE,E)} \norm{A_\lambda -P}_{B(\iE)} \norm{u}_{B(E,\iE)} \to 0
\]
In view of \eqref{eq:mu_lambda} this completes the proof.
\end{proof}
It is perhaps worth mentioning that by Lemma~\ref{lem:orbit_map_quasiisometric}, $\mu_\lambda$ is supported in the ball of radius $a|\lambda|_1+O(1)$ around the identity. Since the convergence is exponential in Theorem~\ref{thm:strong_T_on_building}, we deduce that there is a sequence $\mu_n$ of probability measures supported in the ball of radius $n$ in $\Gamma$ such that for every admissible $E$ and every representation with small enough exponential growth, $\pi(\mu_n)$ converges exponentially fast to the projection $P$. The existence of such a sequence $\mu_n$ is in fact a formal consequence of the main theorem, see  \cite[Corollary 3.12 and \S~3.5]{MR4014781}.
\begin{definition}\label{def:Lambda0-harmonic}
A map $f \colon X\to E$ is \emph{harmonic} if $A_\lambda f=f$ for all $\lambda \in \Lambda$. It is \emph{$\Lambda_0$-harmonic} if $A_\lambda f=f$ for all $\lambda \in \Lambda_0$.
\end{definition}

The proof of Theorem \ref{thm:strong_T_on_building} is done in two steps. We first show in Section~\ref{sec:convergence_harmonic} that $(A_\lambda)_{\lambda}$ indeed converges, but only to a projection on the space of $\Lambda_0$-harmonic elements of $\iE$ (Proposition~\ref{prop:harmonic_limit}). 
The second step in Section~\ref{sec:harmonic_constant} is to prove that the $\Lambda_0$-harmonic elements are constant on vertices of each type (Proposition~\ref{prop:harmonic=constants}). Contrary to the first step, which was the perfect analogue of the case $V=1$ in Theorem~\ref{thm:Lafforgue_s_strategy}, this second step is not a translation of a known statement for Bruhat-Tits buildings. Indeed, the case $V\neq 1$ in Theorem~\ref{thm:Lafforgue_s_strategy} cannot be translated for an exotic building. 

In order to show that $(A_{\lambda})_\lambda$ is a Cauchy sequence when $\lambda \in \Lambda_0$ tend to infinity, we establish certain norm estimates in Section~\ref{sec:norm_estimates}. These in turn rely on certain building combinatorics that are developped in Section~\ref{sec:hjemlslev_biaffine}. These estimates play the role of Lafforgue's Proposition~\ref{prop:lafforguesStrategy_compact} on harmonic analysis for the pair $(K,U)$. The proof of the Cauchy criterion for $(A_{\lambda})_\lambda$ using these estimates is similar to the exploration of the Weyl chamber $K\backslash G/K$ by copies of $U\backslash K/U$. The proof that $\Lambda_0$-harmonic elements are constant on vertices of each type rely on the same tools.

\subsection{Induction and duals}

We close the section with some considerations concerning the interplay of induction and formation of duals. If we apply the induction from above after dualizing both $E$ and $L^2(X)$ we obtain the module
\[
\diE = \{g \colon X \to E^* \mid \forall \gamma \in \Gamma, x \in X g(\gamma x) = \pi(\gamma^{-1})^* g(x)\}
\]
with the norm $\left(\sum_{x \in \Omega} \norm{g(x)}^2\right)^{\frac{1}{2}}$.
The essence of what we are about to say is that induction and dualization commute in the category of topological vector spaces but not in the category of Banach spaces. That is, $\diE$ and $\iE^*$ are linearly homeomorphic but not isometric. To make this precise we define a dual pairing $\langle \cdot ,\cdot\rangle \colon \diE \times \iE \to \C$ by
\begin{equation}\label{eq:duality_pairing}
\gen{g,f}= \sum_{x \in \Omega} \frac{1}{\abs{\Gamma_x}} \gen{g(x),f(x)}\text{.}
\end{equation}
We then have

\begin{lemma}\label{lemma:dual} There is a constant $C'$ such that:
\begin{enumerate}
\item For every $g \in \diE$, the map $f \mapsto \langle g,f\rangle$ is a continuous linear form on $\iE$ of norm $\leq \|g\|$.
\item For every continuous linear form $\varphi \colon \iE \to \C$, there is a unique $g \in \diE$ such that $\varphi(f) = \langle g,f\rangle$ for every $f \in \iE$. It has norm $\norm{g} \leq C' \norm{\varphi}$.
\end{enumerate}
\end{lemma}
\begin{proof}
The fact that $\iE \to \C, f \mapsto \langle g,f\rangle$ has norm $\leq \norm{g}$ follows from the triangle and the Cauchy--Schwarz inequality:
\begin{align*}
\abs{\gen{g,f}} & \leq \sum_{x \in \Omega} \frac{1}{\abs{\Gamma_x}} \abs{\gen{g(x),f(x)}}\\
  &\leq \sum_{x \in \Omega} \frac{1}{\abs{\Gamma_x}} \norm{g(x)} \norm{f(x)} \\
  &\leq \bigg( \sum_{x \in \Omega} \frac{1}{\abs{\Gamma_x}} \norm{g(x)}^2\bigg)^{\frac{1}{2}} \bigg(\sum_{x \in \Omega} \frac{1}{\abs{\Gamma_x}} \norm{f(x)}^2\bigg)^{\frac{1}{2}}\text{.}
\end{align*}

For the second claim we first define auxiliary functions $f_{x,\xi} \colon X \to E$ for $x \in X$ and $\xi \in E$ by
\[
f_{x,\xi}(x') = \frac{1}{\abs{\Gamma_x}}\sum_{\substack{\gamma \in \Gamma\\\gamma x=x'}} \pi(\gamma) \xi\text{.}
\]
Conceptually $f_{x,\xi}$ is the image of $\xi$ under the composition $E \to E^{\Gamma_x} \to \tilde{E}$ where the first map is natural projection and the second is \eqref{eq:etilde_decomposition}.
These functions satisfy the equivariance relation $f_{\gamma x,\xi} = f_{x,\pi(\gamma)^{-1}\xi}$ for every $\gamma \in \Gamma$. Note also that $f_{x,\xi}$ is zero outside of the $\Gamma$-orbit of $x$ and that $f = \sum_{x \in \Omega} f_{x,f(x)}$ provided $f \in \tilde{E}$. Finally we observe that the norm of $f_{x,\xi}$ can be bounded by
\begin{equation}\label{eq:norm_of_f_xxi}
\norm{f_{x,\xi}(x)}_{E} \leq \max_{\gamma \in \Gamma_x} \norm{\pi(\gamma)} \norm{\xi}
\end{equation}
for $x \in \Omega$.

To check uniqueness let $g, g' \in \diE$ defining the same linear form on $\iE$. For $x \in \Omega$ and $\xi \in E$ we have 
\[
\gen{g(x),\xi} = \gen{g(x),f_{x,\xi}(x)} = \abs{\Gamma_x}\gen{g,f_{x,\xi}} = \abs{\Gamma_x} \gen{g',f_{x,\xi}} = \gen{g'(x),\xi}.
\] Since $\xi$ was arbitrary it follows that $g(x) = g'(x)$ and since $x$ was arbitrary, it follows that $g = g'$.

Now let $\varphi \colon \tilde{E} \to \C$ be arbitrary. We take $g \colon X \to E^*$ to be characterized by
\[
\forall \xi \in E, \langle g(x),\xi\rangle = \abs{\Gamma_x} \varphi(f_{x,\xi})
\]
The equivariance of the  auxiliary functions implies $g \in \diE$:
\[
\gen{g(\gamma x),\xi} = \varphi(f_{\gamma x,\xi}) = \varphi(f_{x,\pi(\gamma^{-1})\xi}) = \gen{g(x),\pi(\gamma^{-1}) \xi}\text{.}
\]

To see that $\gen{g,\cdot}$ coincides with $\varphi$ we take $f \in \iE$ and compute
\[
\gen{g,f} = \sum_{x \in \Omega} \frac{1}{\abs{\Gamma_x}} \gen{g(x),f(x)} = \sum_{x \in \Omega} \varphi(f_{x,f(x)}) = \varphi\bigg( \sum_{x \in \Omega}  f_{x,f(x)}\bigg) = \varphi(f)\text{.}
\]

Finally, we estimate the norm of $g$. For every $x \in \Omega$, pick some $\xi_x \in E$ such that $\langle g(x),\xi_x\rangle = \|g(x)\|^2$.
We have
\begin{multline}
    \norm{g}^2  = \sum_{x \in \Omega} \frac{1}{\abs{\Gamma_x}} \gen{g(x),\xi_x} 
     =\varphi(\sum_{x \in \Omega} f_{x,\xi_x})\\
     \leq \norm{\varphi} \cdot \Big\Vert \sum_{x \in \Omega} f_{x,\xi_x} \Big\Vert
     \leq \norm{\varphi} \left( \sum_{x \in \Omega} \frac{1}{\abs{\Gamma_x}} \max_{\gamma \in \Gamma_x} \|\pi(\gamma)\|^2 \|\xi_x\|^2\right)^{\frac 1 2},
\end{multline}   
where on the last line we used \eqref{eq:norm_of_f_xxi}. Taking the infimum over all $\xi_x$ satisfying $\langle g(x),\xi_x\rangle = \|g(x)\|^2$, we obtain
\[ \norm{g}^2 \leq  \norm{\varphi} \max \{ \|\pi(\gamma)\| \mid \gamma \in \cup_{x \in \Omega} \Gamma_x\} \norm{g}.\]
This proves the lemma with $C'= \max\{ \|\pi(\gamma)\| \mid \gamma \in \cup_{x \in \Omega} \Gamma_x\}$.
\end{proof}

The following observation implies in particular that the dual pairing of $\diE$ and $\iE$ does not depend on $\Omega$:

\begin{lemma}\label{lem:fundamental_domain_irrelevant}
For $f \in \iE$ and $g \in \diE$, the map $(x,y) \mapsto \langle f(x),g(y)\rangle$ is $\Gamma$-invariant. That is,
\[
\gen{ f(\gamma x),g(\gamma y)} = \gen{ f(x),g(y)}
\]
for $\gamma \in \Gamma$ and $x,y \in X$.
\end{lemma}

\begin{proof}
By the definitions we have
\[
\gen{f(\gamma x),g(\gamma y)} = \gen{\pi(\gamma)f(x),\pi(\gamma^{-1})^*g(y)} = \gen{f(x),g(y)}\text{.}\qedhere
\]
\end{proof}

\section{Functional analysis preliminaries}\label{sec:Banach_space_preliminaries}

This section contains reminders without proofs of known functional analysis facts and references that will be used in the paper.

To avoid confusion with points in the building at infinity denoted by $p$, we avoid as much as possible the standard letter $p$ for $L^p$ spaces and use $L^r$ spaces instead.
\subsection{Schatten norms }\label{sec:schatten_preliminaries}
If $\cH,\cK$ are Hilbert spaces and $1 \leq r<\infty$, the \emph{Schatten $r$-class} $S^r(\cH,\cK)$ is the space of bounded linear operators $T \colon \cH \to \cK$ such that $\|T\|_{S^r}:= (Tr(|T|^r))^{\frac 1 r} <\infty$, where $|T| = (T^*T)^{\frac{1}{2}} \in B(\cH)$. This is a norm for which $S^r(\cH,\cK)$ is complete, the \emph{Schatten $r$-norm}. This space can also be described as the compact operators $\cH \to \cK$ whose sequence of singular values belong to $\ell^r$. The operator norm corresponds to $r=\infty$.

The Schatten norms behave very much as the $\ell^r$ norms. For example, Hölder's inequality remains true for the Schatten norms \cite[Theorem~2.8]{simon05}:
\begin{equation}\label{eq:Hoelder}
\|AB\|_{S^r} \leq \|A\|_{S^{r_1}} \|B\|_{S^{r_2}}\quad\text{ whenever }\quad\frac 1 r = \frac{1}{r_1} + \frac{1}{r_2}.\end{equation}

When $\cH=\cK=\C^n$, the space $S^r(\cH,\cK)$ is simply denoted $S^r_n$. It is the space of $n \times n$ complex matrices with the norm
\[ \|a\|_{S^r_n} = \left(Tr( (a^*a)^{\frac r 2}\right)^{\frac 1 r}.\]

\subsection{Banach-space valued $L^r$ spaces, regular operators}        
We start with some reminders on Banach-space valued $L^r$ spaces ($r<\infty$), see for example \cite[\S 1.1]{MR3617459} for details. If $(\Omega,\mu)$ is a measure space and $E$ is a Banach space, a function $f \colon \Omega \to E$ is said to be Bochner-measurable if it is the pointwise limit of a sequence of simple (meaning taking finitely many different values) measurable functions. The space of Bochner-measurable functions $f$ such that $\int \|f(\omega)\|^r d\mu(\omega)<\infty$, modulo the space of almost everywhere vanishing functions, is denoted $L^r(\Omega,\mu;E)$ (the measure or the measure space may be dropped from notation if they are clear from the context). It is a Banach space for the norm $\left(\int \|f(\omega)\|^r d\mu(\omega)\right)^{\frac 1 r}$. For $r<\infty$, it contains $L^r(\Omega,\mu) \otimes E$ as a dense subspace, and can therefore also be abstractly viewed at the completion of $L^r(\Omega,\mu) \otimes E$ for the $L^r$ norm. 

Given an operator $T\colon L^r(\Omega) \to L^r(\Omega')$ and a Banach space $E$, it is in general a very subtle question to decide whether $T \otimes \mathrm{id}_E$ extends to a bounded operator $L^r(\Omega;E) \to L^r(\Omega';E)$ and to obtain estimates on its norm, which we denote $\norm{T}_{B(L^r(\Omega;E),B(L^r(\Omega';E)))}$. The operators that do extend to a bounded operator for every Banach space $E$ are called \emph{regular operators} and are rather well understood. For a regular operator $T$, the supremum
\[
\norm{T}_{\text{reg}} \defeq \sup_{\mathclap{E\text{ Banach}}} \ \norm{T}_{B(L^r(\Omega;E),B(L^r(\Omega';E)))}
\]
is finite and is called the \emph{regular norm} of $T$. As recalled in \cite[\S 1.2]{MR2732331}, the regular norm is also attained if the supremum is restricted to $E \in \{\ell^1_n, n\geq 1\}$ (finite-dimensional $\ell^1$-spaces), and also when it is restricted to $E \in \{\ell^\infty_n, n\geq 1\}$ (finite-dimensional $\ell^\infty$-spaces). We also say that an operator is a regular isometry if $T \otimes \mathrm{id}_E$ is a (not necessarily surjective) isometry from $L^r(\Omega;E)$ into $L^r(\Omega;E')$ for every $E$. For more on this, see \cite{MR1128093}, see also \cite{MR2732331} for a concise account.

\begin{example}\label{ex:regular_for_averages}Let $\nu$ be a complex measure on $\Omega \times \Omega'$, such that the marginals of its absolute value $|\nu|$ are absolutely continuous with respect to $\mu$ and $\mu'$ respectively. We denote by $T_\nu$ the operator given by $\int (T_\nu f) g = \int f(\omega) g(\omega') d\nu(\omega,\omega')$. This expression a priori makes sense only for bounded measurable $f,g$ (and by the absolute continuity assumption does not change if $f$ and $g$ are changed on measure $0$ subsets), therefore it a priori only defines a bounded operator $L^\infty(\Omega) \to L^1(\Omega')$. But $T_\nu$ extends to a regular operator $L^r(\Omega) \to L^r(\Omega')$ if and only if $T_{|\nu|}$ extends to a bounded operator $L^r(\Omega) \to L^r(\Omega')$, and $\|T_\nu\|_{reg} = \|T_{|\nu|}\|_{L^r(\Omega) \to L^r(\Omega')}$. When $\Omega$ and $\Omega'$ are finite measure spaces (or even when $\nu$ is absolutely continuous with respect to $\mu \otimes \mu'$), this is a classical fact, see for example \cite[\S 1.3]{MR2732331}. The general case follows by discretization.
\end{example}
In this article, we almost only consider $r=2$. In general, the question of estimating $\|T\|_{B(L^2(\Omega;E),B(L^2(\Omega';E)))}$ is often related to type and cotype properties of $E$. Recall (see \cite[\S 10.4]{MR3617459} and the references therein), that a space has type $p$ and cotype $q$ if there is a constant $C$ such that
\[
C^{-1} (\sum_i \|x_i\|^q)^{\frac 1 q} \leq \mathbb E\| \sum_i \varepsilon_i x_i\| \leq C (\sum_i \|x_i\|^p)^{\frac 1 p}
\]
for every finitely supported sequence $(x_i)$ in $E$, where the $\varepsilon_i$ are independent variables uniformly distributed in $\{-1,1\}$. Every Banach space $E$ has type $1$ and cotype $\infty$ (with $C=1$), and we say that $E$ has nontrivial type (resp. cotype) if it has type $p$ for some $p>1$ (resp. cotype $q$ for some $q<\infty$). We also say that $E$ has trivial type (resp. cotype) if it does not have nontrivial type (resp. cotype). A celebrated theorem of Maurey and Pisier \cite{MR443015, MR571585} asserts that $E$ has trivial type (resp. cotype) if and only if $E$ contains the $\ell^1_n$ (resp. $\ell^\infty_n$) uniformly, that is, there are $n$-dimensional subspaces $F_n \subset E$ and injective linear maps $u_n \colon \ell^1_n \to F_n$ (resp. $\ell^\infty_n \to F_n$) such that $\lim_n \|u_n\| \|u_n^{-1}\| = 1$. This implies that, for every $T \colon L^2(\Omega) \to L^2(\Omega')$
\begin{equation}\label{eq:regular_norm_and_type}\|T\|_{reg} = \|T\|_{B(L^2(\Omega;E),B(L^2(\Omega';E)))} \textrm{ if $E$ has trivial type}.\end{equation}

We conclude these Banach space preliminaries by recalling a result from \cite{MR3474958} that we will use later. For a Banach space $E$ and an integer $n$, let $d_n(E)$ be the supremum, over all subspaces of $E$ of dimension $n$, of its Banach-Mazur distance to $\ell^2_n$ (we set $d_n(E)=1$ if $\dim E<n$). It is known  that $E$ has nontrivial type if and only if there is $n$ such that $d_n(E)<n^{\frac 1 2}$, and in that case $\lim_n n^{-\frac 1 2} d_n(E) =0$ \cite{MR467255}, and that, if $E$ has type $p$ and cotype $q$, then there is a constant $C$ such that $d_n(E) \leq C n^{\frac 1 p - \frac 1 q}$ (combine \cite[Proposition 27.4]{MR993774} by König-Retheford-Tomczak-Jaegermann and \cite[Theorem 5.2]{MR3781331} by Pisier). Moreover we have
\begin{proposition}{\cite[(6) and Proposition 5.3]{MR3781331}}\label{prop:norm_and_Schatten} Let $E$ be a Banach space and $C>0,r>0,\varepsilon>0$ such that $d_n(E) \leq C n^{\frac 1 r - \varepsilon}$ for every $n$. Then there is a constant $L$ (depending on $r,C,\varepsilon$) such that for every operator $T\in S^r(L^2(\Omega),L^2(\Omega'))$
\[
\|T\|_{B(L^2(\Omega;E),B(L^2(\Omega';E)))} \leq L \|T\|_{S^r}.
\]
\end{proposition}

\subsection{A few facts about complex interpolation}\label{subsection:interpolation}
Complex interpolation of complex Banach spaces will be occasionally used. We very briefly recall some facts, and refer to standard textbooks as \cite{MR0482275} for details. 

The simplest incarnation of complex interpolation is the Riesz-Thorin theorem \cite[Theorem 1.1.1]{MR0482275}. It states that, if $(\Omega,\mu)$ and $(\Omega',\mu')$ are measure spaces, and for $i=0,1$ we have $1 \leq r_i,s_i \leq \infty$ and bounded linear operators $T_i: \LL^{r_i}(\Omega) \to \LL^{s_i}(\Omega')$ that agree on $\LL^{r_0} \cap \LL^{r_1}$, then for every $0<\theta<1$, if $\frac{1}{r_\theta} = \frac{1-\theta}{r_0}+ \frac{\theta}{r_1}$ and $\frac{1}{s_\theta} = \frac{1-\theta}{s_0}+ \frac{\theta}{s_1}$, then there is an operator $T_\theta : \LL^{r_\theta}(\Omega) \to \LL^{s_\theta}(\Omega')$ that agrees with $T_0$ (and $T_1$) on $\LL^{r_0}\cap \LL^{r_1}$, and its norm $M_\theta = \|T_\theta\|$ is log-convex. It is customary to write all operators $T_\theta$ by the same letter $T$, and in this phrasing the Riesz-Thorin theorem simply asserts that, if an operator is bounded $\LL^{r_0}\to \LL^{s_0}$ with norm $M_0$, and $\LL^{r_1}\to \LL^{s_1}$ with norm $M_1$, then it is bounded $\LL^{r_\theta}\to \LL^{s_\theta}$ with norm $\leq M_0^{1-\theta} M_1^{\theta}$.

We will also use a version of the Riesz-Thorin theorem which is valid for vector valued functions. It states that if $E,E'$ are any Banach space and if a linear operator $T$ is bounded, for $i=0,1$, $\LL^{r_i}(\Omega;E) \to \LL^{s_i}(\Omega';E')$ with norm $M_i$, then it is bounded $\LL^{r_\theta}(\Omega;E) \to \LL^{s_\theta}(\Omega';E')$ with norm $\leq M_0^{1-\theta} M_1^{\theta}$. Combine Theorem 5.1.2, Theorem 4.1.2 and Definition 2.4.3 in \cite{MR0482275}.

These two examples are two particular cases of complex interpolation, which is a way to construct, given two Banach spaces $E$ and $F$ that are compatible in the sense that they are both given with linear continuous embeddings into a common topological vector space, a family $((E,F)_\theta)_{0 \leq \theta \leq 1}$ of "intermediate" Banach spaces which can be thought of as a kind of "geometric mean of parameter $\theta$ between $E$ and $F$". For example, when $E$ is $\LL^{r_0}(\Omega)$ and $F$ is $\LL^{r_1}(\Omega)$ (both seen inside $\LL^0(\Omega)$, the space of complex measurable functions on $\Omega$ with the measure topology), then $(E,F)_\theta$ is $\LL^{r_\theta}(\Omega)$, and similarly for vector-valued spaces. A form of the Riesz--Thorin theorem is valid in this generality \cite[Theorem 4.1.2]{MR0482275}. This is not central for us, so we refer to \cite{MR0482275} for definitions and results.

\subsection{Operator spaces and their approximation properties}\label{sec:OSpreliminaries}
We briefly recall some definitions and basic facts about operator spaces that will be used in Section~\ref{sec:operator_algebras}. We refer to \cite{PisierOS,EffrosRuan} for further details.

An operator space $V$ is a subspace of $B(\cH)$ for some Hilbert space $\cH$. If $W\subset B(\cK)$ is another operator space, we denote by $V\otimes_{\min} W$ the closure of $V \otimes W$ in $B(\cH \otimes_2 \cK)$, where $\cH \otimes_2 \cK$ is the Hilbert space tensor product.

If $V,W$ are operator spaces, denote by $\CB(V,W)$ (or simply $\CB(V)$ if $V=W$) the space of completely bounded operators $V \to W$, that is, the space of linear maps such that $\mathrm{id}\otimes T \colon \cK(\ell^2) \otimes_{\min} V \to \cK(\ell^2) \otimes_{\min} W$ is bounded. Here $\cK(\ell^2) \subset B(\ell^2)$ is the operator space of all compact operators on $\ell^2$. The completely bounded norm, denoted $\|T\|_{\cb}$, is the norm of $\mathrm{id}\otimes T$. The stable point-norm topology on $\CB(V)$ is the locally convex topology on $\CB(V)$ defined by the seminorms $T \mapsto \|\mathrm{id}\otimes T(x)\|$ for $x \in \cK(\ell^2) \otimes_{\min} V$.

If $(\mathcal M,\tau)$ is a von Neumann algebra with a faithful normal trace $\tau$, the noncommutative $L^r$ space $L^r(\mathcal M,\tau)$ (or $L^r(\mathcal M)$ when $\tau$ is implicit) is the completion of $\{ x \in \mathcal M \mid \|x\|_r <\infty\}$ for the norm $\|x\|_r = (\tau((x^*x)^{r/2}))^{\frac 1 r}$. For example, when $\mathcal M$ is $B(\cH)$ with its usual trace, we recover $S^r(\cH)$. Other examples that are important for us are group von Neumann algebras. If $\Gamma$ is a discrete group, its von Neumann algebra $\VN(\Gamma)$ is the bicommutant $\lambda(\Gamma)'' \subset B(\ell^2(\Gamma))$, where $\lambda$ is the left-regular representation. It is equipped with a normal trace $\tau$ which it determined by $\tau(\lambda(\gamma)) = 1_{\gamma = e}$.

Noncommutative $L^r$ spaces admit a natural operator space structure \cite{PisierOS}, which satisfy that a map $T \colon L^r(\mathcal M) \to L^r(\mathcal N)$ is completely bounded if and only if $\mathrm{id}\otimes T$ is bounded $L^r(B(\ell^2) \overline{\otimes}\mathcal M) \to L^r(B(\ell^2) \overline{\otimes}\mathcal N)$, where $B(\ell^2) \overline{\otimes}\mathcal M$, is the von Neumann algebra tensor product, that is the von Neumann algebra generated by $B(\ell^2) \otimes \mathcal M$ inside $B(\ell^2 \otimes_2 \mathcal H)$ if $\mathcal M \subset B(\mathcal H)$. In that case, the completely bounded norm of $T$ is equal to the norm of $\mathrm{id}\otimes T$.

An operator space has the operator space approximation property (OAP) if the space of finite rank operators $V \to V$ is dense in $\CB(V)$ in the stable point-norm topology.

An operator space has the completely bounded approximation property (CBAP) if there is a constant $C$ such that every operator $T \in \CB(V)$ is a limit in the stable point norm-topology of a net of finite rank operators with completely bounded norm bounded by $C \|T\|_{\cb}$.

\section{Hjelmslev planes and biaffine planes}\label{sec:hjemlslev_biaffine}

In this section we discuss two independent notions that will play a role in establishing the operator estimates in the next section. Both arise by distinguishing elements of a building.

Hjelmslev planes arise when a vertex $o$ in a Euclidean building is fixed. The link of $o$ is the first example of a Hjelmslev planes. Further examples consist of the vertices that lie at a fixed distance $r$ along singular rays emanating from $o$.

The term biaffine plane is non-standard and is based on the following consideration. If in a projective plane one distinguishes a line $\ell$, one obtains an affine plane, consisting of the lines distinct from $\ell$ and the points not on $\ell$. A dual affine plane is obtained by distinguishing a point $p$ rather than a line. A \emph{biaffine plane} is obtained by distinguishing an incident pair of a point $p$ and a line $\ell$ and consists of the points not on $\ell$ and the lines that do not contain $p$. In our case this will be applied to the links of a vertex.

\subsection{Hjelmslev planes}

Let $X$ be a building of type $\tilde{A}_2$ and of order $q$. 

\subsubsection{Incidence and Gromov product}
Throughout this paragraph we will fix a base vertex $o \in X$. Virtually all the notions discussed here will depend on $o$. Our notation will not reflect this here in order to keep it clearer. Later, when the dependence becomes important, the base vertex $o$ will always be the first subscript: we will for instance write $\P_{o,s}$ instead of $\P_s$.

If $p \in \P$ is a point we let $p_s$ denote the vertex at distance $s$ from $o$ on the unique geodesic from $o$ to $p$. If $p_s = p'_s$ we write $p =_s p'$. The set $\P_s \defeq \{p_s \mid p \in \P\}$ is the set of \emph{points of level $s$}. Analogous notation applies to lines and we obtain a set $\L_s$ of \emph{lines of level $s$}. Note that the points and lines of level $1$ are simply the points and lines in the link of $o$.

\begin{definition}
A point $p_s$ and a line $\ell_s$ of level $s$ are \emph{incident} if the convex hull of $o$, $p_s$ and $\ell_s$ is a regular triangle. In that case we also say that $p$ and $\ell$ are \emph{incident of level $s$} and write $p \sim_s \ell$.
\end{definition}

The combinatorial structure formed by the triple $(\P_s, \L_s, \sim_s)$  is a prominent example of a \emph{Hjelmslev plane} for each $s$, see \cite{VanMaldeghem88,HanssensVanMaldeghem89} for a discussion in the context of Euclidean buildings. For $s \le t$ there is a projection $(\P_t,\L_t) \to (\P_s,\L_s)$ taking $x_t$ to $x_s$. A fact from the above references that we are not going to use is that system of planes together with the projections determines $X$. It will be convenient to allow $s = \infty$ where $\P_\infty = \P$ and $\L_\infty = \L$ and their elements are regarded as rays $[o,p)$ and $[o,\ell)$.

For future reference we observe the following:

\begin{lemma}\label{lem:incidence_from_infinity}
Let $p \in \P$ and $\ell \in \L$. The following are equivalent:
\begin{enumerate}
\item $p \sim_s \ell$
\item there is a $p' \in \P$ such that $p' \sim \ell$ and $p' =_s p$
\item there is an $\ell' \in \L$ such that $\ell' \sim p$ and $\ell' =_s \ell$.
\end{enumerate}
\end{lemma}

\begin{proof}
It suffices to prove that the first statement implies the second.  
Using induction it suffices to find $p'$ such that $p' \sim_{s+1} \ell$ and $p' =_s p$. For this purpose consider the combinatorial convex hull of $o$, $\ell_{s+1}$, and $p_{s}$. It is a combinatorial quadrangle $Q$ whose fourth vertex we denote $x$, see Figure~\ref{fig:NumberEquilateral} where the subscripts are smaller by one. By Theorem~\ref{thm:isomappart} it is contained in some apartment (and in fact in some half-apartments of this apartment). 

There are $(q+1)$ simplices $S$ attached to the edge $[p_s,x]$, $q$ of them being outside the quadrangle. By Lemma~\ref{lem:extensionroot} the set $Q\cup S$ is contained in some apartment, and therefore is in fact a regular triangle of size $s+1$. Hence we can take $p'_{s+1}$ to be any of the $q$ vertices 
incident with the edge $[p_s,x]$ not inside the quadrangle.
\end{proof}

\begin{definition}
Let $p\in \P$ be a point and let $\ell \in \L$ be a line. The \emph{Gromov product} based at $o$ is defined as
\[
(p,\ell)_o = \sup \{ s \mid p \sim_{s} \ell\} \in \N \cup \{\infty\}\text{.}
\]
\end{definition}

\begin{remark}\label{rmk:gromov} In the extreme cases we have
\begin{enumerate}
\item $(p,\ell)_o=\infty$ if and only if $p$ and $\ell$ are incident,
\item $(p,\ell)_o=0$ if and only if there is a geodesic ray from $p$ to $\ell$ passing through $o$.
\end{enumerate}
Indeed, the first case is equivalent (by Lemma~\ref{lem:incidence_from_infinity}) to the convex hull of $[o,p)$ and $[o,\ell)$ being a Weyl chamber. For the second case, the reverse implication is clear. For the direct implication, if $(p,\ell)_o=0$ then $\angle_o(p,\ell)$ being different from $60^\circ$ must be equal to $180^\circ$ (because $p_1$ and $\ell_1$ have distinct type). Hence $[o,p)\cup [o,\ell)$ is a local geodesic, and therefore a geodesic line (by the CAT(0) inequality).
\end{remark}

The basic counts in Hjelmslev planes are as follows

\begin{lemma}\label{lem:NumberEquilateral}
Let $1 \le t \le s$. Let $\ell_s \in \L_s$ be a line of level $s$ and let $p_t \in \P_t$ be a point of level $t$. Then
\[
\abs{\{p_s' \in \P_s \mid p_s' =_t p_t \text{ and } p_s' \sim_s \ell_s\}} = 
\begin{cases}
q^{s-t} & \text{ if } p_t \sim_t \ell_t\\
0 & \text{ if } p_t \not\sim_t \ell_t\text{.}
\end{cases}
\]
In particular,
\[
\abs{\{p_s' \in \P_s \mid p_s' \sim_s \ell_s\}} = (q+1)q^{s-1}\text{.}
\]

Furthermore
\[\abs{\{p_s'\in \P_s \mid p_s' =_t p_t \}}=q^{2(s-t)}
\]
The same is true with the roles of points and lines exchanged.
\end{lemma}

\begin{figure}[htb]
\centering
\begin{tikzpicture}[scale=.7]
\node[dot] (o) at (0,0) {};
\node[dot] (ls) at (120:5) {};
\node[dot] (pt) at (60:4) {};
\node[dot] (x) at ($(60:4) + (120:1)$) {};
\draw (o.center) -- (ls.center) -- (x.center) -- (pt.center) -- cycle;
\draw[dashed] (120:4) -- (60:4);
\coordinate (p) at ($(x) + (-8:.8)$);
\coordinate (pprime) at ($(x) + (10:1.1)$);
\draw (x) -- (p) -- (pt);
\draw (x) -- (pprime) -- (pt);
\draw[dotted] (p) -- (pprime);
\node[anchor=south east] at (ls) {$\ell_s$};
\node[anchor=west] at (pt) {$p_{s-1}$};
\node[anchor=north] at (o) {$o$};
\node[anchor=south] at (x) {$x$};
s\end{tikzpicture}
\caption{The configuration in  Lemma~\ref{lem:incidence_from_infinity} and Lemma~\ref{lem:NumberEquilateral} }
\label{fig:NumberEquilateral}
\end{figure}
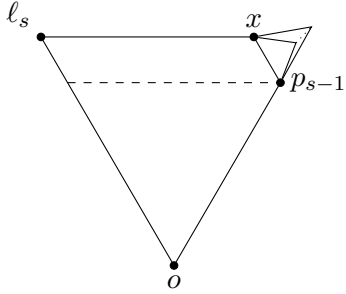

\begin{proof}
The second statement follows from the first since the number of points $p_1'$ of level $1$ incident with a line $\ell_1$ is $q+1$, as it is the number of simplices attached to the edge $[o \ell_1].$

For the first statement, remark first that if $p_t\not \sim \ell_t$, then it is clear that we cannot have $p_s' =_t p_t \text{ and } p_s' \sim \ell_s$. If $p_t\sim \ell_t$, then an induction (on $s$) will conclude
once we have proven the statement for $s = t+1$, so we consider this case. The combinatorial convex hull of $o$, $\ell_s$ and $p_t$ is a quadrangle whose fourth point $x$ is the last vertex of the open segment $(\ell_s,p_s)$, see Figure~\ref{fig:NumberEquilateral}. There are  $q$ vertices in the link of the edge $[p_{s-1},x]$ that are not the one contained in the quadrangle.  With the same argument as in the proof of Lemma~\ref{lem:incidence_from_infinity}, any vertex $p'_s$ in this link satisfies $p'_s\sim \ell_s$. Conversely, if $p'_s$ satisfies $p'_s\sim \ell_s$ and $p'_s=_{s-1} p_{s-1}$ then the regular triangle $o\ell_s p_s$ contains $p_{s-1}$, hence the edge $[p_{s-1},x]$. So $p'_s$ must be in the link of the edge $[p_{s-1},x]$.

For the last statement, by induction on $s$ it suffices to treat the case when $s=t+1$. Choosing a point $p_{t+1}$ amounts to choosing a vertex in the link of $p_t$ which is opposite to $p_{t-1}$. There are $q^2$ such vertices, as it is the number of points in an affine plane of order $q$.
\end{proof}

\subsubsection{Calculation of combinatorial distances}

In this section we use the Gromov product to calculate the combinatorial distance between various points.

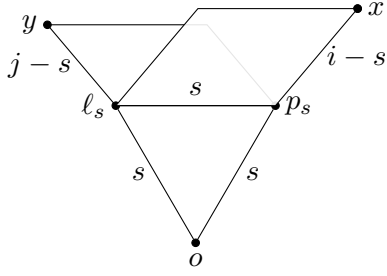
\begin{figure}[htb]
\centering
\begin{tikzpicture}[scale=.7]
\def\i{5.4}
\def\j{5}
\def\s{3}
\node[dot] (o) at (0,0) {};
\node[dot] (y) at ($(120:\s) + (130:\j-\s)$) {};
\coordinate (x) at ($(60:\s) + (50:\i-\s)$);
\node[dot] (x) at ($(60:\s) + (50:\i-\s)$) {};
\node[dot] (ps) at (60:\s) {};
\node[dot] (ls) at (120:\s) {};
\coordinate (a) at ($(120:\s) + (50:\i-\s)$);
\coordinate (b) at ($(60:\s) + (130:\j-\s)$);
\draw (o) -- (ls);
\draw (o) -- (ps);
\path[draw=black] (ls.center) -- (y.center) -- (b) -- (ps.center);
\path[draw=black,fill=white,fill opacity=.9] (x.center) -- (ps.center) -- (ls.center) -- (a) -- cycle;
\node[dot] at (x) {};
\draw (ps) edge node[anchor=south] {$s$} (ls);
\node[anchor=east] at (y) {$y$};
\node[anchor=east] at (ls) {$\ell_s$};
\node[anchor=west] at (x) {$x$};
\node[anchor=west] at (ps) {$p_s$};
\node[anchor=north] at (o) {$o$};
\node[anchor=west] at (60:\s/2) {$s$};
\node[anchor=east] at (120:\s/2) {$s$};
\node[anchor=west] at ($1/2*(ps) + 1/2*(x)$) {$i-s$};
\node[anchor=east] at ($1/2*(ls) + 1/2*(y)$) {$j-s$};
\end{tikzpicture}
\caption{The configuration in Lemma~\ref{lem:sigma_gromov}. Shown are three half-apartments that meet in a line containing the segment $[p_s,\ell_s]$. The upper two contain $x$ and $y$. The front two contain $o$, $\ell_s$ and $x$. The back two contain $o$, $p_s$ and $y$.}
\label{fig:sij}
\end{figure}

\begin{lemma}\label{lem:sigma_gromov}
Let $i,j \in \N$. Let $o \in X$, $p \in \P$, and $\ell \in \L$ and put $x = p_{o,i}$ and $y = \ell_{o,j}$ and $s = (p,\ell)_o$. Then $\sigma(x,y) = (t, i+j - 2t)$ where $t = \min\{i,j,s\}$.
\end{lemma}

\begin{proof}
First, using Theorem \ref{thm:isomappart} there exists an apartment containing $o,p_s$ and $\ell_s$. 

Assume first that $\min\{i,j,s\}=i$ (the case when $\min\{i,j,s\}=j$ being similar).
If $s \geq i$ and $s\geq j$ then it follows that this apartment also contains $x$ and $y$. In this apartment one can verify directly that $\sigma(x,y)=(i,j-i)$.

Assume now that $s\geq i$ and $s<j$. Again there is an apartment containing the triangle $T$ with vertices $o,p_s$ and $\ell_s$ (and therefore containing $x$). Applying twice  Lemma~\ref{lem:extensionroot} we find an apartment containing $T$ and $\ell_{s+1}$. By successive applications of this lemma we get that there is an apartment containing $T$ and $\ell_j=y$. Again in this apartment it is easy to see that $\sigma(x,y)=(i,j-i).$

So assume $s < i$ and $s<j$. By the same successive applications of Lemma~\ref{lem:extensionroot} as above, we see that there exists an apartment containing $o,\ell_s$ and $x$ (and therefore $p_s$). In this apartment we see that the convex hull of $x$ and $\ell_s$ is a parallelogram that has $p_s$ as a vertex, and we call the fourth vertex $a$. Similarly, the convex hull of $y$ and $p_s$ is a parallelogram with vertices $\ell_s$ and $b$, say (see Figure~\ref{fig:sij}).

The geodesic ray from $p_s$ to $p$ contains $x$ while the geodesic ray from $p_s$ to $\ell$ contains $b$. Since $(p,\ell)_o = s$ these rays include an angle of $\pi$: otherwise we would have $p \sim_{p_s,1} \ell$ meaning that $p \sim_{o,s+1} \ell$. Similarly the ray from $\ell_s$ to $p$ contains $a$ and the ray from $\ell_s$ to $\ell$ contains $y$ and these rays include an angle of $\pi$. It follows that $[x,b]$ contains $p_s$, that $[a,y]$ contains $\ell_s$.  Therefore the union of the two parallelogramps $axp_s\ell_s$ and $p_s\ell_syb$ is a parallelogram, which is contained in an apartment by Theorem~\ref{thm:isomappart}. Hence it is the convex hull of $x$ and $y$. In an apartment containing this convex hull it is readily verified that
\begin{multline*}
\sigma(x,y) = (d(x,a),d(x,b)) = (d(p_s,\ell_s),d(x,p_s) + d(p_s,b))\\
= (s, d(x,p_s) + d(\ell_s,y)) = (s, i-s + j -s)
\end{multline*}
as claimed.
\end{proof}

\begin{lemma}\label{lem:general_pl}
Let $o,x,y \in X$. If $\sigma(o,x) = (i,0)$ and $\sigma(o,y) = (0,j)$ then there exist $p \in \P$ and $\ell \in \L$ such that $x = p_{i}$, $y = \ell_{j}$, and $(p,\ell)_o \le \min\{i,j\}$.
\end{lemma}

\begin{proof}
We assume without loss of generality $i \ge j$. Let $p \in \P$ be such that $x = p_{i}$. Let $\ell \in \L$ be such that $y = \ell_{j}$ and $p \not\sim_{{j+1}} \ell$ which can be arranged by Lemma~\ref{lem:NumberEquilateral} : indeed, the number of $\ell'_j\in \L_{j+1}$ such that $\ell'=_j \ell_j$ is equal to $q^2$ whereas by Lemma~\ref{lem:NumberEquilateral} the number of such $\ell'$ also satisfying $p\sim_j \ell'$ is equal to either $0$ or $q$.
\end{proof}

\begin{corollary}\label{cor:Position_poi_loj}
Let $i, j\geq s$. Let $o\in X$, $p\in\P$, $\ell\in\L$ such that $(p,\ell)_o=s$. Let $x=p_{o,i}$ and $y=\ell_{o,j}$. Then $\sigma(x,y)=(s,i+j-2s)$.

Conversely, if $o, x,y\in X$ are such that $\sigma(x,y)=(s,i+j-2s)$, $\sigma(o,x)=(i,o)$ and $\sigma(o,y)=(0,j)$, then there exists $p\in\P$, $\ell\in\L$ such that $(p,\ell)_o=s$ and $x=p_{o,i}$, $y=\ell_{o,j}$.
\end{corollary}

\begin{proof}
The first claim is a special case of Lemma~\ref{lem:sigma_gromov}. For the second one take $p$ and $\ell$ as in Lemma~\ref{lem:general_pl}. Then in Lemma~\ref{lem:sigma_gromov} we have $t = s$ and the second claim follows.
\end{proof}

The following is a variant of Lemma~\ref{lem:general_pl} for two adjacent basepoints.

\begin{lemma}\label{lem:Position_poi_loj_p'oi}
Let $i \ge j\geq s>0$. 
Let $o \in X$, $p\in\P$, $\ell\in\L$ be such that $(p,\ell)_o=s$. Let $x=p_{o,i}$ and $y=\ell_{o,j}$.
Let further $o' \in X$ be such that $\sigma(o,o')=(0,1)$ and $p \sim_{o,1} o'$. Let $x'=p_{o',i-1}$.

Then
\[
\sigma(x,y)=(s,i+j-2s)\text{,} \quad \sigma(x,x') =(1,0)
\]
and
\begin{equation*}
\sigma(x',y) = \begin{cases}(s-1,i+j-2s) & \textrm{if }\ell_{o,1} = o'\text{,}\\
(s+1,i+j-2s-1) & \textrm{if }\ell_{o,1} \neq o' \textrm{ and }i>s\text{,}\\
(s,1) & \textrm{if }\ell_{o,1} \neq o' \textrm{ and }i=s\text{.}
\end{cases}
\end{equation*}
\end{lemma}

\begin{figure}[htb]
\centering
\begin{tikzpicture}[scale=.7]
\def\i{5.4}
\def\j{5}
\def\s{3}
\node[dot] (o) at (0,0) {};
\node[dot] (y) at ($(120:\s) + (130:\j-\s)$) {};
\coordinate (x) at ($(60:\s) + (50:\i-\s)$);
\node[dot] (x) at ($(60:\s) + (50:\i-\s)$) {};
\node[dot] (ps) at (60:\s) {};
\node[dot] (ls) at (120:\s) {};
\coordinate (a) at ($(120:\s) + (50:\i-\s)$);
\coordinate (b) at ($(60:\s) + (130:\j-\s)$);
\draw (o) -- (ls);
\draw (o) -- (ps);
\path[draw=black] (ls.center) -- (y.center) -- (b) -- (ps.center);
\draw[dashed] ($(ps) + (180:1)$) -- ($(b) + (180:1)$);
\draw[dotted] (ps) -- ($(ps) + (0:1)$) -- ($(b) + (0:1)$) -- (b);
\path[draw=black,fill=white,fill opacity=.9] (x.center) -- (ps.center) -- (ls.center) -- (a) -- cycle;
\node[dot] at (x) {};
\node[dot] (oprime) at (120:1) {};
\node[dot] (xprime) at ($(x) + (180:1)$) {};
\node[dot] (odouble) at (0:1) {};
\node[dot] (xdouble) at ($(x) + (50:-1) + (0:1)$) {};
\node[anchor=east] at (y) {$y$};
\node[anchor=east] at (ls) {$\ell_s$};
\node[anchor=south] at (x) {$x$};
\node[anchor=south] at (xprime) {$x'$};
\node[anchor=west] at (xdouble) {$x''$};
\node[anchor=west] at (ps) {$p_s$};
\node[anchor=north] at (o) {$o$};
\node[anchor=east] at (oprime) {$o'$};
\node[anchor=north west] at (odouble) {$\smash{o''}\phantom{o}$};
\draw[dashed] (oprime) -- ($(ps) + (180:1)$) -- (xprime);
\draw[dotted] (o) -- (odouble) -- ($(ps) + (0:1)$) -- (xdouble) -- ($(a) + (50:-1)$);
\end{tikzpicture}
\begin{tikzpicture}[scale=.7]
\def\s{3}
\node[dot] (o) at (0,0) {};
\node[dot] (ps) at (60:\s) {};
\node[dot] (ls) at (120:\s) {};
\draw (o) -- (ls) -- (ps) -- (o);
\node[dot] (otriple) at (0:1) {};
\node[dot] (xtriple) at ($(ps) + (-60:1)$) {};
\node[anchor=east] at (y) {$\phantom{y}$};
\node[anchor=south] at (ls) {$y$};
\node[anchor=south] at (x) {$\phantom{x}$};
\node[anchor=south west] at (xtriple) {$x'''$};
\node[anchor=south] at (ps) {$x$};
\node[anchor=north] at (o) {$o$};
\node[anchor=north west] at (otriple) {$\smash{o'''}\phantom{o}$};
\draw[dashed] (o) -- (odouble) -- (xtriple) -- (ps);
\draw[dashed] (xtriple) -- ($(ls) + (120:-1)$);
\end{tikzpicture}
\caption{The configurations in Lemma~\ref{lem:Position_poi_loj_p'oi}. The left picture illustrates the first two cases $x'$ (dashed) and $x''$ (dotted). The right picture illustrates the third case $x'''$.}
\label{fig:sij'}
\end{figure}
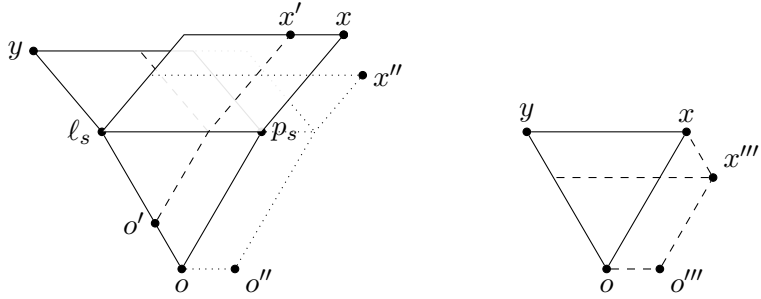

\begin{proof}
The formula for $\sigma(x,y)$ is Lemma \ref{lem:sigma_gromov}. The second claim can be easily seen in an apartment that contains the edges $[o,o']$ and $[x,x']$.

For the last claim we consider the three cases separately and compare them to the situation in Lemma~\ref{lem:sigma_gromov}, see Figure~\ref{fig:sij'}. We take $\Sigma$ to contain $y$ as well as the edge $[x,x']$. In the first case $p_{s-1}$ is the point at distance $1$ from $p_s$ on the segment $[p_s,\ell_s]$. Hence, in $\Sigma$ the segment $[p_{s-1},x']$ is parallel to $[p_s,x]$ and lies at distance $1$ inside the parallelogram for $x$ and $y$. The claim can now be seen by inspecting the situation in $\Sigma$.

In the second case $p_s$ is the point at distance $1$ from $p_{o',s}$ on the segment $[p_{o',s},\ell_s]$. Thus $[p_{o',s},x'']$ is parallel to $[p_s,p_{o,s-1}]$ at distance one outside the parallelogram for $x$ and $y$. Again the situation can be inspected inside $\Sigma$.

In the last case there is an apartment containing the triangle with vertices $o$, $x$, and $y$ as well as $o'$ and this apartment also contains $x'$. The situation can be inspected in this apartment.
\end{proof}

\subsection{Biaffine planes}\label{sec:biaffine_planes}

An \emph{affine plane} can be obtained from any projective plane by removing a ``line at infinity" and all the points incident with it. Symmetrizing points and lines we introduce the following:

\begin{definition}
Let $(P,L,\sim)$ be a projective plane and let $a = (p^0,\ell^0)$ be a pair of a point $p^0 \in P$ and a line $\ell^0$ that are incident, $p \sim \ell$. The associated \emph{biaffine plane} $(P^a,L^a,\sim)$ consists of the points and lines that are incident with neither $p^0$ nor $\ell^0$: $P^a = \{p \in P \mid p \not\sim \ell^0\}$ and $L^a = \{\ell \in L \mid \ell \not\sim p^0\}$.
\end{definition}

In other words, a \emph{biaffine plane} is obtained by removing from an affine plane a parallel class of lines.

As in the affine setting there is a natural notion of parallelism, now between points as well as between lines: two lines $\ell, \ell'$ are \emph{parallel}, denoted $\ell \parallel \ell'$ if they intersect in $p^0$. Similarly two points $p,p'$ are \emph{parallel}, denoted $p \parallel p'$ if they span $\ell^0$.

Let $(\P,\L)$ be the projective plane in the boundary of a $\tilde{A}_2$-building $X$. We fix an incident pair $a=(p^0,\ell^0)$ of a point $p^0\in \P$ and a line $\ell^0 \in \L$. We can apply the above construction to obtain a biaffine plane $(\P^a,\L^a)$ but this is not what we are going to do.

Instead we want to combine it with the Hjelmslev setup from the last section. So let us again fix a base vertex $o \in X$, which this time we partly represent in our notation. Let $(\P_{o,1},\L_{o,1})$ be the Hjelmslev plane of level $1$, the link of $o$. We let $(\P^a_{o,1},\L^a_{o,1})$ be the biaffine plane associated to the pair $(p^0_{o,1},\ell^0_{o,1})$. For every $s \ge 1$ (including $s = \infty$) we let $(\P^a_{o,s},\L^a_{o,s})$ be the preimage of $(\P
^a_{0,1},\L^a_{o,1})$ under the projection $(\P_{o,s},\L_{o,s}) \to (\P_{o,1},\L_{o,1})$. It consists of points $p_s$ and lines $\ell_s$ such that $p_1 \in \P^a_{o,1}$ respectively $\ell_1 \in \L^a_{o,1}$.

A different description is as follows:

\begin{lemma}\label{lem:infinite_biaffine_plane}
The set $\Pa_o = \Pa_{o,\infty}$ consists of points $p \in \P$ such that $o$ lies on a geodesic line from $p$ to $\ell^0$. Similarly $\La_o = \La_{o,\infty}$ consists of lines $\ell \in \L$ such that $o$ lies on a geodesic line from $\ell$ to $p^0$. Moreover, for $1 \le s < \infty$
\[
\P^a_{o,s} = \{p_s \mid p \in \P^a_{o}\}\quad \text{and} \quad \L^a_{o,s} = \{\ell_s \mid \ell \in \L^a_{o}\}\text{.}
\]
\end{lemma}

\begin{proof}
A point $p \in \P$ lies in $\P^a_o$ if $p \not\sim_{o,1} \ell^0$, that is, if $(p,\ell^0)_o = 0$. The first statement now follows from Remark~\ref{rmk:gromov}. For the second it suffices to notice that whether or not $p \sim_{o,1} \ell^0$ only depends on $p_1$ which is determined by $p_s$.
\end{proof}

Note that $\P^a_{o} \subsetneq \P^a$ and $\L^a_{o} \subsetneq \L^a$ are proper inclusions as are $\P^a_{o,s} \subsetneq \P_{o,s}$ and $\L^a_{o,s} \subsetneq \L_{o,s}$ for $s > 1$.

The following lemma, which is not valid for $(\P,\L)$ instead of $(\Pa_o,\La_o)$, is the main reason to work in the biaffine setting. It also explains why we fixed $a = (p^0,\ell^0)$ in the plane at infinity $(\P,\L)$ rather than in $(\P_{o,1},\L_{o,1})$.

\begin{lemma}\label{lem:continuation}
Let $s\geq 1$ and $x\in \La_{o,s}$, $y \in \Pa_{o,s}$. Then
\[
\La_{x}=\{\ell \in \La_o\mid \ell_{o,s}=x\}\quad\text{and}\quad\Pa_{y} = \{p \in \Pa_o\mid p_{o,s} = y\}\text{.}
\]
Consequently,
\[
\La_{x,t}=\{ \ell\in \La_{o,s+t}\mid \ell_s=x\}\quad\text{and}\quad\Pa_{y,t} =  \{p \in \Pa_{o,s} \mid p_{o,s+t} = y\}\text{.}
\]
for every $t \geq 0$.
\end{lemma}

\begin{proof}
Let $\ell\in \La_o$ be such that $\ell_{o,s}=x$. By Lemma~\ref{lem:infinite_biaffine_plane} there is a geodesic line $\rho$ from $p^0$ to $\ell$ through $o$. Since $x$ lies on the ray $[o,\ell)$ it follows that $x \in \rho$, i.e.\ $\ell \in \La_x$. This also shows that $o \in [x,p^0)$.

Now let $\ell \in \La_x$, i.e.\ there is a geodesic line $\rho$ from $p^0$ to $\ell$ through $x$. We just saw that $o \in [x,p^0) \subseteq \rho$, so $\ell \in \La_o$ and $x \in [o,\ell)$. That $\ell_{o,s} = x$ is clear since $d(o,x) = s$ by assumption.

The statement for points is completely analogous. The second statement follows from the second statement of Lemma~\ref{lem:infinite_biaffine_plane}.
\end{proof}

\begin{corollary}\label{cor:regular_tree}
The spaces $\bigcup \{[o,\ell] \mid \ell \in \La\}$ and $\bigcup \{[o,p] \mid p \in \Pa\}$ are rooted $q^2$-regular trees on the vertex sets $\bigcup_{s \ge 0} \La_{o,s}$ respectively $\bigcup_{s \ge 0} \Pa_{o,s}$.
\end{corollary}

\begin{proof}
For every vertex $x$ the sets $\Pa_{x,1}$ and $\La_{x,1}$ consist of $q^2$ elements. The claim follows from Lemma~\ref{lem:continuation} by induction on $s$.
\end{proof}

We do some basic counting in $(\Pa_s,\La_s)$, which can be compared to Lemma~\ref{lem:NumberEquilateral}:

\begin{lemma}\label{lem:counting}
For $0 \leq t < s$ and $p \in \Pa_{o},\ell \in \La_{o}$, we have
\[
\abs{\{\ell'_s \in\La_{o,s} \mid \ell'=_{o,t} \ell\}}=q^{2(s-t)}
\]
and
\[
\abs{\{\ell'_s \in\La_{o,s} \mid \ell'=_{o,t}\ell \text{ and } \ell'\sim_{o,s} p\}}=
\begin{cases}
q^{s-t} & p \sim_{o,t} \ell\\
0 & p \not\sim_{o,t} \ell\text{.}
\end{cases}
\]
The same is true with the roles of points and lines exchanged. 
\end{lemma}

\begin{proof}
The first statement is immediate from Corollary~\ref{cor:regular_tree}.

The second statement for $t \ge 1$ follows from Lemma~\ref{lem:NumberEquilateral} using that $\ell' =_{o,1} \ell \in \La_o$ implies $\ell' \in \La_o$. The case $(s,t) = (1,0)$ is classical: among the $q+1$ neighbors of $p_1$ one is incident with $p^0_1$ and the other $q$ lie in $\La_{o,1}$. The remaining cases with $t = 0$ now follow again using Lemma~\ref{lem:NumberEquilateral}.
\end{proof}

\begin{lemma}\label{lem:characterization_of_Sa} 
Let $\lambda = (s,t) \in \Lambda$. For $y \in S_\lambda(o)$ there exist unique vertices $\posproj{y} \in \P_{o,s}$ and $\negproj{y} \in \L_{o,t}$ such that $\sigma(o,\posproj{y}) = (s,0) = \sigma(\negproj{y},y)$ and $\sigma(o,\negproj{y}) = (0,t) = \sigma(\posproj{y},y)$. Moreover, the following are equivalent.
\begin{enumerate}
 \item $\posproj{y} \in \Pa_{o,s}$ and $\negproj{y} \in \La_{o,t}$.\label{item:sy_yt}
 \item $\posproj{y} \in \Pa_{o,s}$ and $y \in \La_{\posproj{y},t}$.\label{item:u}
 \item $\negproj{y} \in \Pa_{o,t}$ and $y \in \La_{\posproj{y},t}$.\label{item:v}
\end{enumerate}
\end{lemma}

\begin{figure}[htb]
\centering
\begin{tikzpicture}[rotate=180]
\draw (300:3) node[left] {$\ell_{s+t}$}--(0,0) node[below] {$o$}--(240:3) node[right] {$p_{s+t}$};
\draw[dashed] (240:3)--(300:3);
\draw[shift=(240:1)] (0,0) node[right] {$\posproj{y} = p_s$} -- (300:2) node[above] {$y$};
\draw[shift=(300:2)] (0,0) node[left] {$\ell_t = \negproj{y}$} -- (240:1);
\end{tikzpicture}
\caption{$S_\lambda^a(o)$}\label{fig:ps_ls_y}
\end{figure}
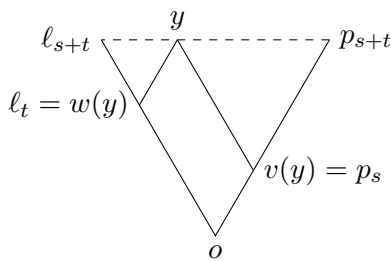

\begin{remark}
The notation $\posproj{y}$ and $\negproj{y}$ does not reflect the dependence of $\lambda$. After this lemma the notation will only be used in Proposition~\ref{prop:sousadditif} where $\lambda$ will be fixed and conciseness will be more important than expressiveness.
\end{remark}

\begin{proof}
In the first statement existence can be seen by taking any apartment that contains $o$ and $y$. Uniqueness follows from the fact that the combinatorial convex hull of $o$ and $y$ is contained in every apartment that contains $o$ and $y$.

The combinatorial convex hull of $o$ and any one of the pairs $(\posproj{y},\negproj{y})$, $(\posproj{y},y)$, and $(\negproj{y},y)$ visibly contains the projection of $y$ onto $o$ (the maximal simplex of $\Conv(o,y)$ containing $o$). Each of the three conditions is satisfied if that convex hull meets $(\P_{o,1},\L_{o,1})$ in $(\Pa_{o,1},\La_{o,1})$.
\end{proof}

\begin{definition}\label{def:Sa_lambda} We denote by $S^a_\lambda(o)$ the set of $y \in S_\lambda(o)$ satisfying the equivalent conditions in Lemma \ref{lem:characterization_of_Sa}.
\end{definition}

\section{Norm estimates}\label{sec:norm_estimates}

Throughout this section we fix a base vertex $o \in X$. This allows us in to make use of the results in the last section. We will again mostly omit the subscript referring to the origin $o$.

\subsection{Biaffine setting}

In this paragraph we put ourselves in a biaffine setting as in Section~\ref{sec:biaffine_planes}. That is, we fix a pair $a=(p^0,\ell^0)$ of a point $p^0\in \P$ and a line $\ell^0 \in \L$. The reason for doing so is that in this situation we can use an induction that would not as easily be feasible in a projective setting. The dependence on the chamber will be removed in the next paragraph.

We equip the finite sets $\Pa_s$ and $\La_s$ with their uniform probability measure, and equip $\Pa$ and $\La$ with the probability measure  of the inverse limit, which we denote $\mu^a$ (or $\mu^a_o$) in both cases. These measures are characterized by the property that their images under the maps $\Pa \to \Pa_s$ and $\La \to \La_s$ are the uniform probability measures on $\Pa_s$ and $\La_s$ respectively.

We denote by $\LL^2(\Pa)$ and (for $0 \le s < \infty$) by $\LL^2(\Pa_s)$ the space of $\LL^2$ functions on $\Pa$ and $\Pa_s$ (with respect to $\mu^a$). Furthermore in this section we denote by $\int_{\Pa} \cdot dp$ (and $\int_{\La}\cdot d\ell$) the integration with respect to $\mu^a$. Our previous discussion implies that there is an isometric inclusion $\LL^2(\Pa_s) \to \LL^2(\Pa)$ sending $f$ to the function $p
\mapsto f(p_s)$. We make these inclusions implicit, regarding $\LL^2(\Pa_s)$ as a subspace of $\LL^2(\Pa)$. Its elements are
functions that are constant on the fibers of the map $p \mapsto p_s$.  We denote by $E_s$ the corresponding conditional expectations,
that is, the orthogonal projections $\LL^2(\Pa) \to \LL^2(\Pa_s)$ given by (the normalization being provided by Lemma~\ref{lem:counting})
\[
E_s(f)(p) = q^{2s} \int_{\Pa} f(p')1_{p' =_s p} dp'\text{.}
\]
Everything we just said applies completely analogously to lines and we use the same notation for the operator $E_s \colon \LL^2(\La) \to \LL^2(\La_s)$.

The purpose of this paragraph is to estimate the differences between the following operators. For $s \in \N$ define $M^a_{o,s}=M^a_s \colon \LL^2(\La) \to \LL^2(\Pa)$ by
\[
M^a_s(f)(p) = \E[f(\ell) \mid p \sim_s \ell]\text{.}
\]
Let $f \in \LL^2(\La)$ and $p \in \Pa$. We have $\int_{\La} 1_{p \sim_s \ell} d\ell = q^{-s}$ by Lemma \ref{lem:counting}, so 
\[
M^a_s(f)(p) = q^s \int_{\La} f(\ell) 1_{p \sim_s \ell} d \ell\text{.}
\]
In particular,
\begin{align}
(q M^a_s - M^a_{s+1})(f)(p) &= q^{s+1} \int_{\La} f(\ell) (1_{p \sim_s \ell} - 1_{p \sim_{s+1} \ell}) d \ell\nonumber\\
  &= q^{s+1} \int_{\La} f(\ell) 1_{(p,\ell)_o=s} d \ell\text{.}\label{eq:Madiff_explicit}
\end{align}
Since $\int_{\La} 1_{(p,\ell)_o=s} d\ell = (q-1)/q^{s+1}$ we obtain
\begin{equation}\label{eq:Tas}
\E[f(\ell) \mid (p,\ell)_o=s] = \frac{1}{q-1} (q M^a_s - M^a_{s+1})(f)(p)\text{.}
\end{equation}

\begin{lemma}\label{lemma:formula_for_EtTs} For arbitrary $s,t \in \N$, we have
\[ E_t M^a_s = M^a_s E_t = M^a_{\min(s,t)}.\]
\end{lemma}
\begin{proof}
It is clear from the definition of $M^a_s$ that $M^a_sf(p)$ depends only on $p_s$, that is, $M^a_s f \in \LL^2(\Pa_s)$. For $t \ge s$ it follows that $E_t M^a_s = M^a_s$.

Assume now $t<s$. Let $f \in \LL^2(\La)$. Applying successively the definitions, Fubini and Lemma \ref{lem:counting}, we obtain
  \begin{align*} E_t M^a_s f(p) &= q^{s+2t} \int_{\Pa} \int_{\La} f(\ell) 1_{\ell \sim_s p'} 1_{p=_t p'} d\ell dp'\\
    &= q^{s+2t} \int_{\La} f(\ell) \left(\int_{\Pa} 1_{p'\sim_s \ell \textrm{ and } p=_t p'} dp' \right) d\ell\\
    &= q^{s} \int_{\La} f(\ell) q^{-(s-t)}1_{p \sim_t \ell} d\ell \\
    & = M^a_t f(p).
  \end{align*}
That is, $E_tM_s^a = M_t^a$.

Putting both cases together we get $E_t M^a_s = M^a_{\min(s,t)}$. The other equality follows with the same proof, or by taking the adjoint and exchanging the roles of points and lines.
\end{proof}

Rather than complex-valued functions we want to consider Banach space-valued functions. So let $E$ be a Banach space. As recalled in Section~\ref{sec:Banach_space_preliminaries}, tensoring an operator $\LL^2(\La) \to \LL^2(\Pa)$ with the identity $E \to E$ we obtain an operator $\LL^2(\La;E) \to \LL^2(\Pa;E)$ that we will systematically denote by the same symbol.

\begin{proposition}\label{prop:sousadditif}
For every integer $s \geq 1$ let $C_{s,E}$ denote (the supremum over all choices of a base vertex $o$ of) the norm of
\[
M^a_s - M^a_{s-1} \colon \LL^2(\La;E) \to \LL^2(\Pa;E)\text{.}
\]
We have $C_{s+t,E} \leq C_{s,E} C_{t,E}$ for every integers $s,t \geq 1$. 
\end{proposition}
\begin{proof}
Fix $s,t \ge 1$ and let $\lambda = (s,t) \in \Lambda$. We will use the set $S^a_\lambda(o)$ from Definition \ref{def:Sa_lambda}, and for $y \in S^a_\lambda(o)$ we let $\posproj{y}$ and $\negproj{y}$ denote the unique vertices in $\Conv(o,y) \cap S_{(s,0)}(o)$ and $\Conv(o,y) \cap S_{(0,t)}(o)$ respectively, see Lemma~\ref{lem:characterization_of_Sa} and Figure~\ref{fig:ps_ls_y}.

Observe that, for every $\ell_{s+t} \in \La_{o,{s+t}}$ and $p_{s+t} \in \Pa_{o,{s+t}}$, the number of $y \in S^a_\lambda(o)$ such that $ y \sim_{\posproj{y}} p_{s+t}$ and $y \sim_{\negproj{y}} \ell_{s+t}$ is $1$ if $\ell \sim_{s+t} p$ (in which case $y$ is the unique point at distance $t$ from $\ell_{s+t}$ on $[\ell_{s+t}, p_{s+t}]$), and $0$ otherwise. It follows that for $f \in \LL^2(\La_{s+t};E)$ and $x \in \Pa_{s+t}$ we can decompose
\begin{align*}
 M^a_{s+t} (f)(x)& = \frac{1}{q^{s+t}} \sum_{z \in \La_s} f(z) 1_{z \sim_{o,s+t} x}\\
   & =  \bigg(\frac{1}{q^{t}} \sum_{y \in S^a_\lambda(o)} 1_{y \sim_{\posproj{y},t} x}\bigg) \bigg(\frac{1}{q^{s}} \sum_{z \in \La_s} f(z) 1_{y \sim_{\negproj{y},s} z}\bigg).
   \end{align*}

This means that there is a decomposition $M^a_s = AB$ into
\[
\LL^2(\La_s;E) \overset{B}{\longrightarrow} \LL^2(S^a_\lambda(o);E) \overset{A}{\longrightarrow} \LL^2(\Pa_s;E)
\]
where $S^a_\lambda(o)$ carries the uniform probability measure. The operators are given by
\begin{align*}
A(f)(x) &= \frac{1}{q^{t}} \sum_{y \in S^a_\lambda(o)} f(y) 1_{y \sim_{\posproj{y},t} x}\text{,}\\
B(f)(y) &= \frac{1}{q^{s}} \sum_{z \in \La_s} f(z) 1_{y \sim_{\negproj{y},s} z}\text{.}
\end{align*}
Using Lemma \ref{lemma:formula_for_EtTs}, we obtain
\[
M^a_{s+t-1} - M^a_{s+t-1} = (E_{s+t} - E_{s+t-1}) M^a_{s+t} (E_{s+t} - E_{s+t-1}) = (E_{s+t} - E_{s+t-1}) A B(E_{s+t} - E_{s+t-1})\text{.}
\] 
We claim that the norm of
\[
B(E_{s+t} - E_{s+t-1}) \colon \LL^2(\La;E) \to \LL^2(S^a_\lambda(o);E)
\]
is bounded by $C_{s,E}$ and that the norm of
\[
(E_{s+t} - E_{s+t-1})A \colon \LL^2(S^a_\lambda(o);E)) \to \LL^2(\Pa;E)
\]
is bounded by $C_{t,E}$. Partitionning the set of $z\in\La_{s+t}$ according to the value of $z_{o,t}$ we can decompose $\LL^2(\La;E)$ as the $\ell^2$-direct sum indexed by $v \in \La_{t}$, of $\LL^2(\{\ell \in \La \mid \ell_{t} = v\};E)$ which are simply $\LL^2(\La_{v,s};E)$ by Lemma \ref{lem:continuation}.

In this decomposition, $B$ decomposes as the direct sum of the operators $M^a_{v,s}$, and $E_{s+t} - E_{s+t-1}$ corresponds to the direct sum of $E_{v,s} - E_{v,s-1}$. By Lemma \ref{lemma:formula_for_EtTs} again, we obtain that $B(E_{s+t} - E_{s+t-1})$ decomposes as the direct sum of the operators $M^a_{s,v} - M^a_{s-1,v}$. This yields the desired inequality.

Similarly, we can partition $S^a_\lambda(o)$ according to the value of $\posproj{y}$ and decompose $\LL^2(S^a_\lambda(o);E)$ as the $\ell^2$-direct sum indexed by $u \in \Pa_s$ of the $\LL^2(\{y \in S^a_\lambda(o) \mid \posproj{y} = u\})$) which are $\LL^2(\Pa_{u,s})$ by Lemma~\ref{lem:characterization_of_Sa}).

In this decomposition $A$ decomposes as the direct sum of the $M^a_{u,s}$, and $(E_{s+t} - E_{s+t-1})$ decomposes as the direct sum of $E_{u,t} - E_{u,t-1}$.
\end{proof}

\begin{remark}\label{rem:factorization_of_Ms-Ms-1}
The proof shows that, when $s,t \geq 1$, the operator $M^a_{s+t}-M^a_{s+t-1}$ can be factorized as

\[
\begin{tikzcd}
\bigoplus_{z' \in \La_{s_2}} L^2(\Pa_{z',s_1}) \arrow[r,"V"] 
&  \bigoplus_{z \in \Pa_{s_1}} L^2(\La_{z,s_2}) \arrow[d,"B"] \\
\bigoplus_{z' \in \La_{s_2}} L^2(\La_{z'}) \arrow[u,"A"] 
& \bigoplus_{z \in \Pa_{s_1}} L^2(\Pa_{z})\arrow[d,"U"]\\
L^2(\La_o) \arrow[r,"M^a_{s+t}-M^a_{s+t-1}"]\arrow[u,"W"] & L^2(\Pa_o)
\end{tikzcd}
\]

where $U ,V,W$ are unitaries, $A = \bigoplus_{z \in \Pa_{s_1}} (M^a_{s_2,z} - M^a_{s_2-1,z})$, and $B = \bigoplus_{z' \in \La_{s_2}} (M^a_{s_1,z'} - M^a_{s_1-1,z'})$. More precisely, the operators $U,V,W$ are not only unitaries but also regular isometries  (see Section~\ref{sec:Banach_space_preliminaries} for the terminology).
\end{remark}

This proposition implies that we have a dichotomy:

\begin{corollary}
Depending on the Banach space $E$
\begin{enumerate}
\item either $C_{s,E}\geq 1$ for every $s$\label{item:E_bad}
\item or there are $C,\varepsilon>0$ such that $C_{s,E} \leq Ce^{-\varepsilon s}$ for every $s$.\label{item:E_good}
\end{enumerate}
\end{corollary}

The spaces in case \eqref{item:E_good} are the ones for which our proof works. There are, however, also Banach spaces in case \eqref{item:E_bad}. The admissible spaces will be discussed in more detail in Section~\ref{sec:good_banach_spaces} below. We close this paragraph by showing that Hilbert spaces belong to the good case \eqref{item:E_good}.

\begin{proposition}\label{prop:operator_norm_of_Ms-Ms-1}
Assume that $E$ is a Hilbert space. Then for every $s\in \N^*$ we have $\norm{M^a_s-M^a_{s-1}}_{\LL^2(\La;E) \to \LL^2(\Pa;E)} \leq \frac{1}{\sqrt q^s}.$
\end{proposition}

\begin{proof} Decomposing in an orthonormal basis of $E$, we can assume that $E=\C$.

In view of Proposition \ref{prop:sousadditif}, it suffices to prove
the inequality for $s=1$, i.e $\Vert M^a_1-M^a_0\Vert \leq
\frac{1}{\sqrt q}$.  For $f\in \LL^2(\La)$ and $g\in \LL^2(\Pa)$
we have
\[
\gen{M^a_s f,g}= q^s\int_{\La\times\Pa} 1_{p \sim_s \ell} \gen{f(\ell),g(p)} d\ell dp\text{,}
\]
so that for all $s>0$ the operator ${M^a_s}^*:\LL^2(\Pa)\to \LL^2(\La)$ is given by the formula ${M^a_s}^*(f)(\ell)= \E[ f(p) \mid \ell\sim_s p]$. 

Using Lemma \ref{lemma:formula_for_EtTs} and that $E_s = E_s^*$ we see that ${M^a_1}^*M^a_0={M^a_0}^*M^a_1={M^a_0}^*M^a_0$. Hence
\[
(M^a_1-M^a_0)^*(M^a_1-M^a_0)={M^a_1}^*M^a_1-{M^a_0}^*M^a_0\text{.}
\]
Moreover, $M^a_1-M^a_0=(M^a_1-M^a_0)E_1$ is supported on $\LL^2(\Pa_1)$, so we only need to compute the norm on this subspace.

From the definitions we see that ${M^a_0}^*M^a_0 f=\E(f)$. On level $1$ we compute the following where all $x,x'\in \Pa_1$ and $z,z'\in\La_1$:
\begin{align*}
{M^a_1}^*M^a_1 f(z)&=\frac{1}{q}\sum_{x \sim z} M^a_1f(x)\\
&=\frac{1}{q^2}\sum_{x\sim z}\sum_{z' \sim x} f(z')\\
&=\frac{1}{q^2}\left( qf(z)+\sum_{z'\not\parallel z} f(z')\right)\\
&=\frac{1}{q} f(z)+\frac{1}{q^2}\sum_{z'} f(z')-\frac{1}{q^2} \sum_{z'\parallel z} f(z')\text{.}
\end{align*}

Hence we get $({M^a_1}^*M^a_1-{M^a_0}^*M^a_0)(f)(\ell)=\frac{1}{q}f(\ell)-\frac{1}{q^2}\sum_{\ell'\parallel \ell } f(\ell')$. In other words, ${M^a_1}^*M^a_1-{M^a_0}^*M^a_0=\frac{1}{q}(\mathrm{Id}-P)$ where $P:\LL^2(\La_1)\to \LL^2(\La_1)$ is the orthogonal projection on the subspace formed by functions constant on parallelism classes (i.e.\ such that $f(\ell)=f(\ell')$ whenever $\ell \parallel \ell'$). It follows that $q({M^a_1}^*M^a_1-{M^a_0}^*M^a_0)$ is an orthogonal projection (of rank $q^2-q$) and thus has norm $1$. Hence $\norm{M^a_1-M^a_0} =\frac{1}{\sqrt q}$.  
\end{proof}

\subsection{Projective setting}\label{sec:norm_estimates_projective}

Our goal is now to prove estimates similar to the previous one, but on 
$\P$ and $\L$ instead of $\Pa$ and $\La$. We first $\P_s$ and $\L_s$ with the uniform measures. We write $\proba_o$ for the measure which is the inverse limit of these probability measures on $\P$, and the same for the measure on $\L$. We also write $\proba_o$ for the product measure on $\P\times \L$.

Observe that the probability measure $\mu^a$ on $\Pa$ is also the probability measure $\proba_o$ conditionned on $\Pa$: for every measurable $A \subset\Pa$ we have $\mu^a(A)=\proba_o(A)/\proba_o(\Pa)$. 

From now on the integration of a function $f$ with respect to the measure $\proba_o$ (on $\L$ or on $\P$) will be denoted $\int_{\L} f(\ell) d\ell $ or $\int_{\P} f(p) dp$.

We will be interested in the operators $T_{s} \colon L^2(\L,\proba_o)\to L^2(\P,\proba_o)$ defined by 
\[
T_s (f)(p)=\E[f(\ell) \mid (\ell,p)_o=s]\text{.}
\]
In view of Lemma~\ref{lem:NumberEquilateral}, for $s \geq 1$ and every $p \in \P$, 
\begin{align*}\proba_o((\ell,p)_o=s)&=\proba_o(\ell_s \sim_s \ell_s) - \proba_o(\ell_{s+1} \sim_{s+1} \ell_{s+1})\\ 
& = \frac{(q+1) q^{s-1}}{(q^2+q+1)q^{2s-2}} - \frac{(q+1) q^{s}}{(q^2+q+1)q^{2s}}\\
& = \frac{(q-1)(q+1)}{(q^2+q+1)q^{s}}
\end{align*}
the explicit formula for $s \ge 1$ is
\[
T_s(f)(p) = \frac{(q^2+q+1)q^{s}}{(q-1)(q+1)}\int_{ \L} f(\ell) 1_{(\ell,p)_o = s} d\ell\text{.}
\]

We will see that they are obtained as some average, over all possible pairs $a=(p^0,\ell^0)$, of the operators in \eqref{eq:Tas}. So we will derive norm estimates on the $T_s$ from those on the $M_s^a$ from the last paragraph by averaging over $a$.

Concretely let $\Pi$ be the set of incident point-line pairs $a = (p,\ell)$ in $\lk(o)$. Let $N = (q^2+q+1)(q+1)$ be the cardinal of $\Pi$.
Recall that for $a\in \Pi$, $(\P^{a}_o,\L^{a}_o)$ is well-defined since the biaffine structure depends only on level $1$. Let $\hat{\L}$ be the disjoint union of the $\L^{a}$ (for $a\in \Pi$) equipped with the probability measure $\mu$ that is $1/N$ times the sum of the measures $\mu^{a}$. Thus
\[
\int_{\hat\L} f(\ell) d\mu(\ell) = \frac{1}{N} \sum_{a\in \Pi}\int_{\L^{a}} f(\ell) d\mu^{a}(\ell)\text{.}
\]

Define the operator $U \colon \LL^2(\L,\proba_o) \to L^2(\hat{\L},\mu)$ by $U(f)(\ell) = f(\ell)$ for all $f \in \LL^2(\L)$ and $\ell \in \L^{a}$ with $a\in\Pi$. The same discussion applies to points and we obtain an operator $V \colon \LL^2(\P) \to \LL^2(\hat{\P})$.

\begin{lemma}\label{lem:proj_biaff_isometries}
The operators $U \colon \LL^2(\L) \to \LL^2(\hat{\L})$ and $V \colon \LL^2(\P) \to \LL^2(\hat{\P})$ are (non surjective) regular isometries.
\end{lemma}
Recall (see Section~\ref{sec:Banach_space_preliminaries}) that a regular isometry is a map that remains an isometry when tensorizing with $E$ for every Banach space $E$.
\begin{proof}
We only prove the statement for $U$, the statement for $V$ being analogous. It suffices to show that the measure induced from $\mu$ via $U$ is $\proba_o$, that is,
\[
\int_{\L} f(\ell)d\ell = \int_{\hat{\L}} U(f)(\ell) d\mu(\ell)
\]
for $f \in \LL^2(\L)$. 
This follows from three observations: First, the restriction of both measures to each $\L^{a}$ is proportional to  measure $\mu^a$ on $\L^{a}$. Second, the proportionality constant is the same for each $a\in \Pi$: indeed, for $\ell_1 \in \L_1$ the number of $a\in \Pi$ such that $\ell_1 \in \L^{a}_1$ is independent of $\ell_1$, so that the induced measure of $\{\ell' \in \L \mid \ell' \sim_1 \ell\}$ is independent of $\ell_1$. Third, both measures are probability measures. It follows that the measures coincide.
\end{proof}

\begin{lemma}\label{lem:ts_averaged}
With $U$ and $V$ as above we have
\[
T_{s} = V^*\left(\bigoplus_{a\in \Pi} \frac{1}{q-1} (q M^{a}_s - M^{a}_{s+1}) \right) U\text{.}
\]
\end{lemma}

\begin{proof}
First note that the two measures on incident point line pairs given by
\[
\int_{\L}\int_{\P} f(\ell,p) 1_{\ell \sim_1 p} dp d\ell \quad \text{and} \quad \frac{1}{N} \sum_{a\in \Pi} \int_{\L^{a}}\int_{\P^{a}} f(\ell,p) 1_{\ell \sim_1 p} d\mu^a(p) d\mu^a(\ell)
\]
are proportional (because the number of $a\in \Pi$ that contains $(\ell,p)$ is independent of $(\ell,p)$). Moreover the total measure ($f \equiv 1$) is $(q+1)/(q^2+q+1)$ for the left one and is $1/q$ for the right one.

Now let $B$ denote the operator on the right-hand side of the claim. For $s = 0$ the statement is trivial so we assume $s \ge 1$.
For $f \in L^2(\L,\proba_o)$ and $g \in L^2(\P,\proba_o)$ and $s \ge 1$ we get using \eqref{eq:Madiff_explicit} and the above comparison

\begin{align*}
\gen{Bf,g} &= \Big\langle \bigg(\bigoplus_{a\in \Pi} \frac{1}{q-1} (q M^{a}_s - M^{a}_{s+1}) \bigg)Uf, Vg\Big\rangle\\
&= \frac{1}{N} \sum_{a} \frac{q^{s+1}}{q-1} \int_{\L^{a}} \int_{\P^{a}} f(\ell) \overline{g(p)} 1_{(p,\ell)_o=s} dp d\ell\\
&= \frac{q^2+q+1}{(q+1)q} \frac{q^{s+1}}{q-1} \int_{\L} \int_{\P} f(\ell) \overline{g(p)} 1_{(p,\ell)_o=s} dp d\ell\\
&= \frac{(q^2+q+1)q^s}{(q+1)(q-1)} \int_{\L} \int_{\P} f(\ell) \overline{g(p)} 1_{(p,\ell)_o=s} dp d\ell\\
& = \gen{T_{s} f,g}\text{.}\qedhere
\end{align*} 
\end{proof}

From this we conclude the main result of the section. Combined with Proposition~\ref{prop:operator_norm_of_Ms-Ms-1} it gives the estimate that we need for Hilbert spaces:

\begin{proposition}\label{prop:operator_norm_of_Ts-Ts-1}
For every Banach space $E$,
\begin{equation*}
\norm{T_{o,s} - T_{o,s+1}} \leq \frac{q+1}{q-1}\sup_a \sup_{t\geq s} \norm{M^{a}_t - M^{a}_{t+1}}
\end{equation*}
as operators $L^2(\L;E) \to L^2(\P;E)$, where the norm is the operator norm.
\end{proposition}

\begin{proof}
From Lemma~\ref{lem:ts_averaged} we have
\begin{align*}
\norm{T_s -T_{s+1}} &\le \frac{1}{q-1} \sup_a \norm{q M_s^a - M^a_{s+1} - qM^a_{s+1} + M^a_{s+2}}\\
& \le \frac{1}{q-1} \sup_a \bigg(q\norm{M^a_s - M^a_{s+1}} + \norm{M^a_{s+1} - M^a_{s+2}}\bigg)
\end{align*}
and the claim follows.
\end{proof}

\subsection{Schatten norms estimates}\label{sec:schatten_norms}

The following generalizes Proposition \ref{prop:operator_norm_of_Ms-Ms-1} to Schatten norms (see Subsection \ref{sec:schatten_preliminaries}).

\begin{proposition}\label{prop:Schatten_norm_of_Ms-Ms-1}
For every $s\in \N^*$ and $r \in [1,\infty)$, we have
\begin{equation}\label{eq:schatten_norm_estimateMs} \Vert M^a_s-M^a_{s-1}\Vert_{S^r(\LL^2(\La),\LL^2(\Pa))} \leq q^{-s(\frac 1 2 - \frac{2}{r})}.\end{equation}
\end{proposition}
\begin{proof} We prove by induction on $s$ that \eqref{eq:schatten_norm_estimateMs} holds for every choice of an origin $o$. When $s=1$, we have seen in the proof of Proposition \ref{prop:operator_norm_of_Ms-Ms-1} that $q(M^a_1-M^a_0)^*(M^a_1-M^a_0)$ is an orthogonal projection of rank $q^2-q$. So we have
\[
\norm{M^a_1 - M^a_0}_{S^r} = q^{-\frac{1}{2}} (q^2-q)^{\frac{1}{r}} \leq q^{-\frac{1}{2} + \frac{2}{r}}\text{.}
\]
Combining Hölder's inequality \eqref{eq:Hoelder} with Remark \ref{rem:factorization_of_Ms-Ms-1} we obtain that, $\|M^a_{s+t} - M^a_{s+t-1}\|_{S^r}$ is bounded above by
\[
\bigg(\sum_{z \in \Pa_{s}} \norm{M^a_{t,z} - M^a_{t-1,z}}_{S^{r_1}}^{r_1}\bigg)^{\frac{1}{r_1}}  \bigg(\sum_{z' \in \La_{t}} \norm{M^a_{s,z'} - M^a_{s-1,z'}}_{S^{r_2}}^{r_2}\bigg)^{\frac{1}{r_2}}\text{.}
\]
Recall from Lemma~\ref{lem:counting} that $\Pa_{s}$ and $\La_{t}$ have cardinality $q^{2s}$ respectively $q^{2t}$. So if we assume that \eqref{eq:schatten_norm_estimateMs} holds for $s,t$ and every choice of the origin, we obtain
\[ \norm{M^a_{s+t} - M^a_{s+t-1}}_{S^r} \leq \left(q^{2s} q^{-tr_1(\frac{1}{2} - \frac{2}{r_1})}\right)^{\frac{1}{r_1}} \left(q^{2t} q^{-sr_2(\frac{1}{2} - \frac{2}{r_2})}\right)^{\frac{1}{r_2}} = q^{\frac{2(s+t)}{r}-\frac{s+t}{2}}\text{.}
\]
In summary \eqref{eq:schatten_norm_estimateMs} for $s$ and $t$ and every choice of the origin implies \eqref{eq:schatten_norm_estimateMs} for $s+t$ and every choice of the origin. Together with the case $s = 1$ this completes the proof.
\end{proof}

The following proposition is the analogue of Proposition~\ref{prop:operator_norm_of_Ts-Ts-1} for Schatten norms.

\begin{proposition}\label{prop:Schatten_norm_of_Ts-Ts-1}
For every integer $s$ and every $r\in [4,\infty)$,
\begin{equation*} \Vert T_{s} - T_{s+1}\Vert_{S^r(\LL^2(\La),\LL^2(\Pa))} \leq 4 q^{\frac 3 r} q^{-s(\frac 1 2 - \frac{2}{r})}.\end{equation*}
\end{proposition}
\begin{proof}
Observe that if $T = \bigoplus_{i=1}^N T_i$ for $T_i \colon \cH_i \to \cK_i$, then
\[\|T\|_{S^r} = \bigg(\sum_i \|T_i\|_{S^r}^r\bigg)^{\frac 1 r} \leq N^{\frac 1 r} \max_i \|T_i\|_s.\]
From Lemma~\ref{lem:ts_averaged} and Hölder's inequality we deduce
\begin{align*}
\norm{T_s -T_{s+1}}_{S^r} & \le \frac{1}{q-1}  \bigg\lVert \bigoplus_a q M_s^a - M^a_{s+1} - qM^a_{s+1} + M^a_{s+2} \bigg\rVert_{S^r}\\
& \le \frac{1}{q-1} N^\frac{1}{r} \sup_a \bigg(q\norm{M^a_s - M^a_{s+1}}_{S^r} + \norm{M^a_{s+1} - M^a_{s+2}}_{S^r}\bigg)
\end{align*}
Using Proposition~\ref{prop:Schatten_norm_of_Ms-Ms-1}, whose bound is monotonically decreasing in $s$ for $r \ge 4$, our last quantity can be bounded above by $\frac{q+1}{q-1} N^\frac{1}{r} q^{-s(\frac{1}{2} - \frac{2}{r})}$. The claim follows because $\frac{q+1}{q-1} N^{\frac 1 r} \leq 4 q^{\frac 3 r}$ for every $r \geq 4$.
\end{proof}

\subsection{Banach space-valued estimates}\label{sec:good_banach_spaces}
We say that a Banach space $E$ is \emph{admissible for $X$}  (or just \emph{admissible} when $X$ is fixed) if there exist $L$ and $\theta > 0$ such that for every $o \in X$ and every $s\geq 0$
\begin{equation}\label{eq:E-valued_Hoelder_bound}
\norm{T_{o,s} - T_{o,s+1}}_{B(L^2(\P;E),L^2(\L;E))} \leq L e^{-\theta s}
\end{equation}
and the same statement holds for $\bar{X}$.

The main result of the section is:

\begin{theorem}\label{thm:Banach_spaces_with_exponential_decay} The class of admissible Banach spaces
\begin{enumerate}
\item contains all spaces such that $d_n(X) = O(n^{\frac 1 4-\varepsilon})$ for some $\varepsilon>0$,
\item is stable under duality, subspaces, quotients, ultrapowers,
\item is stable under the operation $E \mapsto (E,F)_\theta$ for every other Banach space $F$ and every $\theta<1$, 
\item is stable by the operation $E\mapsto \LL^r(\Omega,\mu;E)$
for every $1<r<\infty$.
\end{enumerate}
\end{theorem}
In particular, it contains all spaces of type $p$ and cotype $q$ satisfying $\frac 1 p - \frac 1 q<\frac 1 4$, all $\LL^r$ spaces or non-commutative $L^r$ spaces in the range $1<r<\infty$ and their subspaces and quotients. The spaces appearing in Theorem \ref{thm:Banach_spaces_with_exponential_decay} are exactly the spaces for which strong property (T) of $\SL_3(\R)$ is known. 
\begin{proof} The first item is obtained by combining Proposition \ref{prop:Schatten_norm_of_Ts-Ts-1} and Proposition \ref{prop:norm_and_Schatten}. The second and third items are well-known, see for example \cite[Lemma 3.1]{MR3474958}. The last item is Fubini if $r=2$. The general case follows also from Fubini and from the following consequence of complex interpolation, see Subsection~\ref{subsection:interpolation}: the class of Banach spaces for which there exist $L$ and $\theta$ such that 
\[\forall s, \|T_{o,s} - T_{o,s+1}\|_{B(\LL^r(\P;E),\LL^r(\L;E))} \leq L e^{-\theta s}\]
does not depend on $1<r<\infty$.
\end{proof}

Not every Banach space is admissible:
\begin{lemma}\label{lem:not_nontrivial_type}
If $E$ has trivial type, then
\[
\norm{T_{s} - T_{s+1}}_{B(\LL^2(\L;E),\LL^2(\P;E))} = 2\]
and the constant from Proposition~\ref{prop:sousadditif} is
\[
C_{s,E} = 2 - \frac 2 q
\]
for every $s \geq 0$.
\end{lemma}
\begin{proof} The operators $T_{o,s}$ are of the form $T_{\nu_s}$ (as in Example \ref{ex:regular_for_averages}) for probability measures on $\L \times \P$ supported on $\{(\ell,p)\mid (\ell,p)_o = s\}$. Since the $\nu_s$ are disjointly supported, we have from \eqref{eq:regular_norm_and_type} and Example \ref{ex:regular_for_averages} that
\[  \|T_{o,s} - T_{o,s+1}\|_{B(\LL^2(\L;E),\LL^2(\P;E))} =  \|T_{o,s} + T_{o,s+1}\|_{B(\LL^2(\L),\LL^2(\P))}.\]
This quantity is at most $2$ by the triangle inequality, and equal to $2$ as it maps the constant $1$ to the constant $2$. The same argument applies to show that for every $a$ and $s$, the operators $M_{a,s} - M_{a,s+1}$ from Proposition \ref{prop:sousadditif} have regular norm $2-\frac 2 q$, the total variation of the difference between the uniform measures on $\{(\ell,p) \in \La \times \Pa |\mid \ell \sim_s p\}$ for $s$ and $s+1$. 
\end{proof}
We conjecture that Lemma~\ref{lem:not_nontrivial_type} is the only obstruction to admissibility:

\begin{conjecture}\label{conj:nontrivial_type}
A Banach space is admissible if and only if it has nontrivial type.
\end{conjecture}

Evidence for this conjecture is that it is known for classical buildings \cite{MR2574023}. We should mention, however, that a Fourier transform argument in the proof reduces to a property of certain biaffine monodromy groups that does not hold for arbitrary $\tilde{A_2}$-buildings. More precisely, we could not verify the property for $\tilde{A}_2$-buildings whose links are classical projective planes but computer experiments have shown that the property is definitely not satisfied in buildings in which even the link of some vertex is exotic and of order less than $24$ (and there are such buildings that admit uniform lattices \cite{MR3656296}).

\section{Averaging over the basepoint}\label{sec:averaging_basepoint}

Virtually everything we have done so far depended on a basepoint $o \in X$. In this section we will prove results that eliminate this dependency by averaging over the possible basepoints. If the action of $\Gamma$ on $X$ is cocompact, as in the main theorem on strong property (T), this is obviously meaningful. For use in Theorem~\ref{thm:lp_cohomology} it will be convenient to also allow for $\Gamma$ to not act cocompactly and, in particular, for it to be trivial. In those cases we need an assumption that ensures the average to be well-defined and we will assume that the function to be integrated is real-valued and non-negative.

We begin with a statement on functions on the vertices $X$ that will be used in Section~\ref{sec:convergence_harmonic} to show convergence of the operators $A_\lambda$ to projection on harmonic functions and also in the results on operator algebras in Section~\ref{sec:operator_algebras}. We then introduce a set $X^{(1)}$ of directed edges and prove analogous but more technical results for functions on it. Those will be used in Section~\ref{sec:harmonic_constant} to show that harmonic functions are constant and also to prove vanishing of $\LL^p$-cohomology in Section~\ref{sec:lp_cohomology}.

\begin{proposition}\label{prop:double_counting}
Let $i \ge j \geq s$. Let $\lambda = (s,i+j-2s)$. Let $F\colon X \times X \to \C$ be a $\Gamma$-invariant function. Assume that $\Gamma$ acts cocompactly or that $F$ is real-valued, non-negative.
We have
\begin{equation}\label{eq:double_counting}
\sum_{o \in \Omega} \frac{1}{\abs{\Gamma_o}} \E_o[F(p_{o,i},\ell_{o,j}) \mid (p,\ell)_o=s] = \sum_{x \in \Omega} \frac{1}{\abs{\Gamma_x}} \frac{1}{\abs{S_\lambda(x)}} \sum_{y \in S_\lambda(x)} F(x,y)\text{.}
\end{equation}
\end{proposition}

\begin{proof}
We consider the set 
\begin{equation*}
A \defeq \{(o,x,y) \in X\mid \sigma(o,x)=(i,0),\sigma(o,y)=(0,j),\sigma(x,y) = \lambda\}
\end{equation*}
and compute in two ways the integral $I$ of $F(x,y)$ over the set $\Gamma \backslash A$ with respect to the measure giving weight $1/\abs{\Gamma_{o,x,y}}$ to the element $\Gamma(o,x,y)$ (where $\Gamma_{o,x,y} = \Gamma_o \cap \Gamma_x \cap \Gamma_y$).
We first disintegrate with respect to $o$. Denoting $A_o = \{(x,y)| (o,x,y) \in A\}$, we obtain
\[
I = \sum_{o \in \Omega} \sum_{(x,y)\in \Gamma_o\backslash A_o} \frac{1}{|\Gamma_{o,x,y}|} F(x,y) = \sum_{o \in \Omega} \sum_{(x,y)\in A_o} \frac{1}{|\Gamma_o|} F(x,y)
\]
because $\abs{\Gamma_o}/\abs{\Gamma_{o,x,y}}$ is the cardinality of the $\Gamma_o$-orbit of $(x,y)$. By Lemma \ref{lemma:counting}, the cardinality of $A_o = \bigcup_{x \in S_{(i,0)}(o)} S_\lambda(x) \cap S_{(0,j)}(o)$ does not depend on $o$, denote it by $Z$.

By Corollary~\ref{cor:Position_poi_loj}, the set $A_o$ is equal to the set of pairs $(p_{o,i},\ell_{o,j})$, for $p\in\P$, $\ell\in \L$ with $(p,\ell)_o = s$.
So by the definition of the measure $\proba_o$, the uniform average of $F$ on $A_o$ is equal to $\E_o[F(p_{o,i},\ell_{o,j}) \mid (p,\ell)_o=s]$. We therefore obtain
\[
I = Z \sum_{o \in \Omega} \frac{1}{\abs{\Gamma_o}} \E_o[F(p_{o,i},\ell_{o,j}) \mid (p,\ell)_o=s]\text{.}
\]

Let us now disintegrate with respect to $x$. Denoting $A^x = \{(o,y) \mid (o,x,y) \in A\}$, we obtain
\[
I = \sum_{x \in \Omega} \sum_{(o,y) \in A^{x}/\Gamma_{x}} \frac{1}{\abs{\Gamma_{o,x,y}}} F(x,y) = \sum_{x \in \Omega} \sum_{(o,y) \in A^x} \frac{1}{\abs{\Gamma_{x}}} F(x,y)\text{.}
\]
By Lemma \ref{lemma:counting}, the number $Z_1$ of $o$ such that $(o,x,y) \in A$ does not depend on $x$ nor $y$ as long as $\sigma(x,y)=\lambda$. So we can write
\[
I = \sum_{x \in \Omega} \frac{Z_1 \abs{S_\lambda(x)}}{\abs{\Gamma_x}} \frac{1}{\abs{S_\lambda(x)}} \sum_{y \in S_\lambda(x)} F(x,y)\text{.}
\]
By Lemma \ref{lemma:counting} again, $Z':=Z_1 \abs{S_\lambda(x)}$ does not depend on $x$ and we obtain
\[
I = Z' \sum_{x \in \Omega} \frac{1}{\abs{\Gamma_x}} \frac{1}{\abs{S_\lambda(x)}} \sum_{y \in S_\lambda(x)} F(x,y)\text{.}
\]
It remains to justify that $Z=Z'$. If $\Gamma$ is cocompact this can be verified by evaluating the expression for the function $F$ that is constant $1$. 

Otherwise we observe that for $(o,x,y) \in A$ we have
\begin{align*}
Z &= \abs{S_{(i,0)}(o)} \cdot \abs{S_\lambda(x) \cap S_{(0,j)}(o)}\\
Z' &= \abs{S_{(-i,0)}(x)} \cdot  \abs{S_\lambda(x) \cap S_{(0,j)}(o)}
\end{align*}
and that $\abs{S_{(i,0)}(o)} = \abs{S_{(-i,0)}(x)}$.
\end{proof}

For many of the results that we have proven for the set $X$ of vertices, we will need to verify similar but more intricate statements for the set of edges. More precisely we introduce the set
\[
X^{(1)} \defeq \{(x,x')\mid \sigma(x,x') = (1,0)\}
\]
consisting of edges with orientation given by the cyclic orientation of types. If we assume that the action of $\Gamma$ is type rotating, which in view of Lemma~\ref{lem:type-preserving} means no loss of generality, then $\Gamma$ acts on $X^{(1)}$ and we let $\Omega^{(1)}$ denote a set of representatives for $\Gamma \backslash X^{(1)}$.

Given an edge $(x,x') \in X^{(1)}$ it will be important to understand the possible values of $\sigma(x',y)$ knowing $\sigma(x,y)$. In order to describe these, we associate to $\lambda \in \Lambda$ the elements $\lambda^+, \lambda^-, \lambda^\sim \in \Lambda$ as follows (see Figure~\ref{fig:lambda_star}). If $\lambda = (0,0)$ then $\lambda^+=\lambda^-=\lambda^\sim=(0,1)$. If $\lambda = (i,j) \neq (0,0)$ then
\begin{align*}
\lambda^+ &=(i,j+1)\text{,} &
\lambda^- &= \begin{cases} (i-1,j) & \textrm{if }i\neq 0\\
(1,j-1) & \textrm{if }i=0,\end{cases} &
\lambda^\sim &= \begin{cases} (i+1,j-1) & \textrm{if }j\neq 0\\
 (i,1) & \textrm{if }j=0.\end{cases} &
\end{align*}

\begin{figure}
\centering
\begin{tikzpicture}[scale=1.6]
\begin{scope}
\clip(-1.7,-.7) rectangle (1.7,4);
\foreach \ang in {0,60,...,300} {
\draw[dashed] (0,0) -- (\ang:5);
}
\end{scope}
\node[dot,fill=white] (x) at (0,0) {};
\node[dot, outer sep=5pt] (y) at ($(90:3)$) {};
\node[dot] (minus) at ($(y)-(60:1)$) {};
\node[dot] (plus) at ($(y)-(-60:1)$) {};
\node[dot] (tilde) at ($(y)-(180:1)$) {};
\draw (x) edge[->] node[near end,anchor=west] {$\lambda$} (y);
\draw (x) edge[-{>[width=3]}] node[near end,anchor=east] {$\lambda^-$} (minus);
\draw (x) edge[-{>[width=3]}] node[near end,anchor=west] {$\lambda^\sim$} (tilde);
\draw (x) edge[-{>[width=3]}] node[near end,anchor=east] {$\lambda^+$} (plus);
\draw (0,0) edge[->] node[anchor=west] {$(1,0)$} (60:1);
\draw (0,0) edge[->] node[anchor=east] {$(0,1)$} (120:1);
\end{tikzpicture}
\hspace{.5cm}
\begin{minipage}[b]{4.8cm}
\begin{tikzpicture}[scale=1.6]
\begin{scope}
\clip(-2,-.2) rectangle (1,2);
\foreach \ang in {0,60,...,300} {
\draw[dashed] (0,0) -- (\ang:5);
}
\end{scope}
\node[dot,fill=white] (x) at (0,0) {};
\node[dot, outer sep=5pt] (y) at ($(120:2)$) {};
\node[dot] (minus) at ($(y)-(60:1)$) {};
\node[dot] (tilde) at ($(y)-(180:1)$) {};
\draw (x) edge[->] node[near end,anchor=west] {$\lambda$} (y);
\draw (x) edge[dashed,-{>[width=3]}]  (minus);
\draw (x) edge[-{>[width=3]}] node[near end,anchor=west] {$\lambda^-$} (tilde);
\end{tikzpicture}
\begin{tikzpicture}[scale=1.6]
\begin{scope}
\clip(-1,-.2) rectangle (2,2);
\foreach \ang in {0,60,...,300} {
\draw[dashed] (0,0) -- (\ang:5);
}
\end{scope}
\node[dot,fill=white] (x) at (0,0) {};
\node[dot, outer sep=5pt] (y) at ($(60:2)$) {};
\node[dot] (plus) at ($(y)-(-60:1)$) {};
\node[dot] (tilde) at ($(y)-(180:1)$) {};
\draw (x) edge[->] node[near end,anchor=west] {$\lambda$} (y);
\draw (x) edge[dashed,-{>[width=3]}]  (tilde);
\draw (x) edge[-{>[width=3]}] node[near end,anchor=east] {$\lambda^\sim$} (plus);
\draw (0,0) edge[->] node[anchor=west] {$(1,0)$} (60:1);
\draw (0,0) edge[->] node[anchor=east] {$(0,1)$} (120:1);
\end{tikzpicture}
\end{minipage}

\caption{The vectors $\lambda^+, \lambda^\sim, \lambda^-$. The generic cases are shown on the left. The singular cases reflect the folding of the building and are shown on the right.}
\label{fig:lambda_star}
\end{figure}
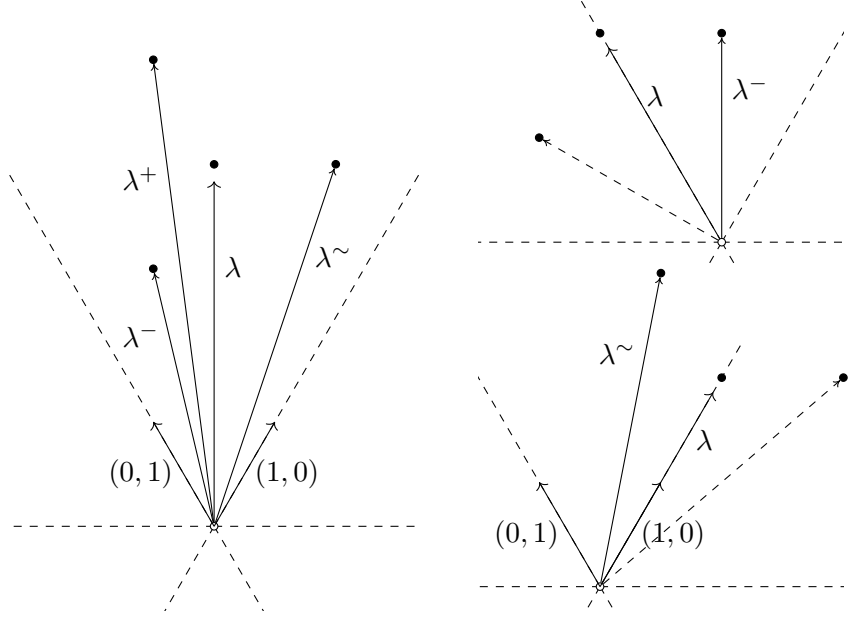

The relationship with elements of $X^{(1)}$ is given by the following, \cite[Lemma 2.1]{CartwrightMlotkowski94}.

\begin{lemma}\label{CartwrightMlotkowski} Let $x,y \in X$, and $\lambda = \sigma(x,y)$. When $x'$ varies in $S_{(1,0)}(x)$, the possible values for $\sigma(x',y)$ are $\lambda^+,\lambda^\sim$ and $\lambda^-$. Moreover the sphere $S_{(1,0)}(x)$ can be partitionned into three sets of respective cardinality $q^2,q,1$ such that $\sigma(x',y)=\lambda^+$ on the first, $\lambda^\sim$ on the second and $\lambda^-$ on the third.
\end{lemma}

A basic but important observation is

\begin{lemma}\label{lem:important_fact}
Let $\lambda = (i,j)$. If $j = 0$ then $\lambda^+ = \lambda^\sim$. If $i = 0$ then $\lambda^\sim = \lambda^-$.\qed
\end{lemma}

This observation implies, in particular, that the values of $\sigma(x',y)$ in Lemma~\ref{CartwrightMlotkowski} need not be distinct: if $\lambda = (i,0)$ there will be a set of cardinality $q^2+q$ on which $\sigma(x',y) = \lambda^+ = \lambda^\sim$ and if $\lambda = (0,j)$ then there will be a set of cardinality $q+1$ on which $\sigma(x',y) = \lambda^\sim = \lambda^-$. For future reference we also record the following consequence of Lemma~\ref{CartwrightMlotkowski}.

\begin{corollary}\label{cor:unique_plus_inverse}
Let $\lambda=(i,j) \in \Lambda$, and let $x',y \in X$ with $\sigma(x',y) = \lambda^+$. There is a unique $x \in X$ with $\sigma(x,x') = (1,0)$ and $\sigma(x,y) = \lambda$.
\end{corollary}

\begin{proof}
For $\lambda = (i,j)$ recall that $\bar{\lambda} = (j,i)$. Swapping the order of types $S_{(0,1)}(x')$ is partitioned as in Lemma~\ref{CartwrightMlotkowski} into $x$ according to the value of $\overline{\sigma}(x,y)$. The claim now follows by observing that $\bar{\mu} = \lambda^+$ implies $\mu^- = \bar{\lambda}$.
\end{proof}

Let $(x,x') \in X^{(1)}$ and $\lambda \in \Lambda$. Phrasing Lemma~\ref{CartwrightMlotkowski} differently we see that $S_\lambda(x)$ decomposes into up to three different sets, obtained by imposing one of the three possible combinatorial distances from $x'$. These are the \emph{spheres}
\begin{equation*}
S_\lambda^\truc(x,x') = \{y \in X\mid \sigma(x,y)=\lambda,\sigma(x',y) = \lambda^\truc \}
\end{equation*}
for $\truc\in \{+,-,\sim\}$.

For a vector space $E$, we can define analogues of the operators $A_\lambda$ from \eqref{eq:def_Alambda} as the linear maps from the functions $X\to E$ to the functions $X^{(1)}\to E$, defined by
\begin{equation}\label{eq:def_Alambdatruc}
A_\lambda^\truc(f)(x,x') = \frac{1}{\abs{S_\lambda^\truc(x,x')}} \sum_{y \in S_\lambda^\truc(x,x')}f(y).
\end{equation}
We will use the following consequence of Lemma~\ref{CartwrightMlotkowski}.
\begin{lemma}\label{lem:relation_Alambda_Alambdatruc} For every $\lambda \in \Lambda$, every $\truc \in \{+,-,\sim\}$, every $f \colon X \to E$ and every $(x,x')\in X^{(1)}$,
  \[ A_\lambda f(x) = \frac{1}{q^2+q+1}\left(q^2 A_\lambda^+f(x,x') + q A_\lambda^\sim f(x,x')+ A_\lambda^-f(x,x')\right).\]
\end{lemma}
\begin{proof} By Lemma~\ref{lemma:counting}, $\abs{S_\lambda^\truc(x,x')}$ depends only on $\lambda$ and $\truc$ and not on $(x,x') \in X^{(1)}$; denote by $\abs{S_\lambda^\truc}$ this common value. In the generic case ($\lambda = (i,j)$ with $i,j>0$), summing over $x' \in S_{(1,0)}(x)$ and using Lemma~\ref{CartwrightMlotkowski} gives $(q^2+q+1) \abs{S_\lambda^+} = q^2 \abs{S_\lambda}$, $(q^2+q+1) \abs{S_\lambda^\sim} = q \abs{S_\lambda}$ and $(q^2+q+1) \abs{S_\lambda^-} = \abs{S_\lambda}$. The lemma follows.

  When $\lambda=(i,0)$ with $i>0$, a similar reasoning leads to $\abs{S_\lambda^+}=\abs{S_\lambda^\sim} = \frac{q^2+q}{q^2+q+1}\abs{S_\lambda}$ and $\abs{S_\lambda^-}=\frac{1}{q^2+q+1}\abs{S_\lambda}$. This also implies the lemma. The other singular cases ($\lambda = (0,j), j > 0$ and $\lambda = (0,0)$) are proved in the same way. 
\end{proof}

For $o \in X$, we define
\[
\P^{(1)}_o \defeq \{(p,o')\in \P\times \L_{o,1} \mid p \sim_{o,1} o'\}\text{.}
\]
Note that if $p \in \P$ then $(p_{o,i},p_{o,i+1}) \in X^{(1)}$. Note also that if $(p,o') \in \P^{(1)}_o$ then $(p_{o,i},p_{o',i-1}) \in X^{(1)}$ by Lemma~\ref{lem:Position_poi_loj_p'oi}.

We define
\begin{align*}
\lun{ X} &=\{(x,x') \in X\times X \mid \sigma(x,x') = (0,1)\}\\
&=\{(x,x') \in X \times X \mid (x',x) \in X^{(1)}\}\text{.}
\end{align*}
We also define $\lun{\Omega} = \{(x,x') \mid (x',x) \in \Omega\}$ and note that it is a set of representatives for $\Gamma \backslash \lun{X}$.

For  $(x,x') \in \lun X$ and $\lambda=(i,j) \in \Lambda$ we define 
\begin{align*}
\lplus \lambda &=(i+1,j)\text{,} &
\lmoins \lambda &= \begin{cases} (i,j-1) & \textrm{if }j\neq 0\\
(i-1,1) & \textrm{if }j=0,\end{cases} &
\lsim \lambda &= \begin{cases} (i-1,j+1) & \textrm{if }i\neq 0\\
 (1,j) & \textrm{if }i=0.\end{cases} &
\end{align*}

And we also denote accordingly
\begin{equation*}
\ltruc S_\lambda(x,x') = \{y \in X\mid \sigma(x,y)=\lambda,\sigma(x',y) = \ltruc\lambda \}
\end{equation*}
for $\truc\in \{+,-,\sim\}$.

Recall that $\overline X$ is the same building as $X$ with the types exchanged, and that we use $\overline{\cdot}$ for the object corresponding to $\cdot$, but in $\overline X$. For example, $\overline{\ltruc S_\lambda}$ denotes $\ltruc S_\lambda$ in the building $\overline{X}$. 
\begin{lemma}\label{lemma:spheres_opposite_building}
We have $\lun{ X} =\smash{\overline{X}}^{(1)}$. Moreover, for $(x,x')\in X^{(1)} = \overline{\lun{X}}$ 
we have $\overline{\ltruc S_\lambda}(x,x')=S^\truc_{\overline \lambda}(x,x')$.
\end{lemma}

\begin{proof}
For $x,x' \in X$, we have 
\[ (x,x') \in \smash{\overline{X}}^{(1)} \iff \overline{\sigma}(x,x')=(1,0) \iff \sigma(x,x')=(0,1) \iff (x,x') \in \lun{ X}.\]
This proves that $\lun{ X} =\smash{\overline{X}}^{(1)}$. Now for $(x,x') \in X^{(1)}$, for every $y \in X$ we have
\begin{align*} y \in \overline{\ltruc S_\lambda}(x,x') &\iff \overline{\sigma}(x,y) = \lambda, \overline{\sigma}(x',y) = \ltruc\lambda\\
& \iff \sigma(x,y) = \overline{\lambda}, \sigma(x',y) = \overline{\ltruc\lambda}.\end{align*}
But it follows from the definitions that $\overline{\ltruc\lambda} = \overline{\lambda}^\truc$. We therefore obtain
\begin{align*} y \in \overline{\ltruc S_\lambda}(x,x')& \iff \sigma(x,y) = \overline{\lambda}, \sigma(x',y) = \overline{\lambda}^\truc\\
& \iff y \in S_{\overline{\lambda}}^\truc(x,x').
\end{align*}
\end{proof}
Combining Lemma~\ref{CartwrightMlotkowski} and Lemma~\ref{lemma:spheres_opposite_building} we see that
$S_\lambda(x)$ is partitioned into $\lplus S_\lambda(x,x')$, $\lsim S_\lambda(x,x')$, and $\lmoins S_\lambda(x,x')$ when $(x,x') \in \lun{}X$, with respective cardinalities easily computed.

We define an operator $\ltruc A_\lambda$, similar to $A_\lambda^\truc$, from the space of functions $X\to E$ to the space of functions $X\to \lun E$ 
 by the formula 
\begin{equation}\label{eq:def_ltrucAlambda}
\ltruc A_\lambda (f)(x,x') = \frac{1}{\abs{\ltruc S_\lambda(x,x')}} \sum_{y \in \ltruc S_\lambda(x,x')}f(y).
\end{equation}

\begin{lemma}\label{lem:relationAlambda_ltruc}
 For every $\lambda \in \Lambda$, writing $\lambda=(i,j)$, we have
 for every $f \colon X \to E$ and every $(x,x')\in X^{(1)}$,
 \begin{enumerate}
     \item If $i,j>0$:
         \[A_\lambda f(x') = \frac{1}{q^2+q+1} \left( q^2 A_{\lplus\lambda}^- f(x,x') + q A_{\lsim\lambda}^\sim f(x,x') + A_{\lmoins\lambda}^+ f(x,x')
         \right),\]
    \item If $i>0,j=0$:
        \[A_\lambda f(x') = \frac{1}{q^2+q+1} \left( q^2 A_{\lplus\lambda}^- f(x,x') + (q+1) A_{\lsim\lambda}^\sim f(x,x')
 \right),\]
    \item If $i=0,j>0$:
        \[A_\lambda f(x') = \frac{1}{q^2+q+1} \left( (q^2+q) A_{\lplus\lambda}^- f(x,x') +  A_{\lmoins\lambda}^+ f(x,x')
 \right),\]
        
 \end{enumerate}

\end{lemma}

\begin{proof}
Applying Lemma~\ref{lem:relation_Alambda_Alambdatruc} to the building $\overline X$ we get that for every $(x',x)\in \overline X^{(1)}$ and $\lambda\in \Lambda$
\[ \overline A_\lambda f(x') = \frac{1}{q^2+q+1}\left(q^2 \overline A_\lambda^+f(x',x) + q \overline A_\lambda^\sim f(x',x)+ \overline A_\lambda^-f(x',x)\right).\]
Using Lemma~\ref{lemma:spheres_opposite_building} we get that $\overline A_\lambda^\truc = \ltruc A_{\overline\lambda}$, so that, applying the previous equality to $\overline \lambda$, we get that for every $\lambda\in \Lambda$
\[ A_\lambda f(x') = \frac{1}{q^2+q+1}\left(q^2 \cdot \lplus A_\lambda f(x',x) + q\cdot \lsim A_\lambda f(x',x)+ \lmoins A_\lambda f(x',x)\right).\]

But note that $\lplus S_\lambda(x',x) = \{y\mid \sigma(x,y) = \lplus \lambda , \sigma(x',y) = \lambda \} = S_{\lplus\lambda}^- (x,x') $ since $(\lplus \lambda)^- = \lambda$. Therefore $\lplus A_\lambda f(x',x) = A_{\lplus\lambda}^- f(x,x')$, for every $\lambda$.

If $i,j>0$ we can easily check that $(\lmoins \lambda)^+=(\lsim \lambda)^\sim =\lambda$, so that we obtain $\lmoins A_\lambda f(x',x)= A_{\lmoins\lambda}^+ f(x,x')$ and $\lsim A_\lambda f(x',x)= A_{\lsim\lambda}^\sim f(x,x')$. The first claim follows.

If $i>0$ and $j=0$ we have in fact $\lmoins \lambda = \lsim \lambda$, and we still get $(\lmoins \lambda)^\sim = (\lsim \lambda)^\sim = \lambda$, so that $\lsim A_\lambda f(x',x) = \lmoins A_\lambda f(x',x) = A_{\lsim\lambda}^\sim f(x,x')$, and the second claim follows. Similarly for $i=0$ and $j>0$
we get $\lsim \lambda=\lplus \lambda$ and therefore $(\lsim\lambda)^- = (\lplus\lambda)^- = \lambda$. 
\end{proof}

The next lemma is a direct consequence of Lemma~\ref{lem:sigma_gromov} and of Lemma~\ref{lem:Position_poi_loj_p'oi} using the notation of this section.  

\begin{lemma}\label{lemma:respective_position_in_X1} Let $i \ge j\geq s>0$ be integers, and let $\lambda=(s,i+j-2s)$.

If $p \in \P$ and $\ell \in \L$ satisfy $(p,\ell)_o =s$, then
\[ \sigma(p_{o,i},\ell_{o,j})=\lambda\textrm{ and }\sigma(p_{o,i+1},\ell_{o,j})=\lambda^+.\]

If $(p,o') \in \P^{(1)}_o$ and $\ell \in \P$ satisfy $(p,\ell)_o =s$, then 
\begin{equation*} \sigma(p_{o,i},\ell_{o,j})=\lambda\textrm{ and }\sigma(p_{o',i-1},\ell_{o,j}) = \begin{cases}\lambda^- & \textrm{if }\ell_{o,1} = o'\\ \lambda^\sim & \text{otherwise.}\end{cases}
\end{equation*}
\par\vspace{-2\baselineskip}\flushright\qed
\end{lemma} 

\begin{corollary}\label{cor:respectivepositions}
Let $i \ge j\geq s>0$ and let $\lambda=(s,i+j-2s)$.  Fix $o\in X$, $p\in \P_{o,i}$ and $\ell\in \L_{o,j}$ such that $(p,\ell)_o=s$.

Then, if  $i>s$, we have
\begin{enumerate}
\item $\{x'\in S_{(1,0)}(p_{o,i})\mid \sigma (x',\ell_{o,j})=\lambda^+\} = \{p'_{o,i+1}\mid p'\in \P,p'=_{o,i}p\}$
\item $\{x'\in S_{(1,0)}(p_{o,i})\mid \sigma (x',\ell_{o,j})=\lambda^-\} = \{p_{o',i-1}\mid o'=\ell_{o,1}\}$
\item $\{x'\in S_{(1,0)}(p_{o,i})\mid \sigma (x',\ell_{o,j})=\lambda^\sim\} = \{p'_{o',i-1}\mid (p',o')\in \P^{(1)},p'=_{o,i}p,  o'\neq \ell_{o,1}\}$ 
\end{enumerate}

When $i=j=s$ the second equality remains true while 
\begin{multline*}
   \{x'\in S_{(1,0)}(p_{o,i})\mid \sigma (x',\ell_{o,j})=\lambda^+\} = \{p'_{o,i+1}\mid p'\in \P,p'=_{o,i}p\}\\\cup \{p'_{o',i-1}\mid (p',o')\in \P^{(1)},p'=_{o,i}p,  o'\neq \ell_{o,1}\}
\end{multline*}
\end{corollary}

\begin{proof}
By Lemma \ref{lemma:respective_position_in_X1} the three sets on the right are included in the sets of the left. But their cardinalities are respectively $q^2,1$ and $q$, which is the same on the left by Lemma \ref{CartwrightMlotkowski}. If $i=j=s$ the same argument applies except that $\lambda^+=\lambda^\sim$, so that the set on the left is of cardinality $q^2+q$.
\end{proof}

By similar arguments, one can check the following.

\begin{lemma}\label{lem:respectivepositions_l}
Let $i \ge j\geq s>0$ and let $\lambda=(s,i+j-2s)$.  Fix $o\in X$, $p\in \P_{o,i}$ and $\ell\in \L_{o,j}$ such that $(p,\ell)_o=s$.

Then, if $i>s$, we have
\begin{enumerate}
\item $\{x'\in S_{(0,1)}(p_{o,i})\mid \sigma (x',\ell_{o,j})=\lplus\lambda\} = \{p'_{o',i}\mid (p',o')\in \P^{(1)},p'=_{o,i}p,  o'\neq \ell_{o,1}\}$ 
\item $\{x'\in S_{(0,1)}(p_{o,i})\mid \sigma (x',\ell_{o,j})=\lmoins \lambda\} = \{p_{o,i-1}\}$
\item $\{x'\in S_{(0,1)}(p_{o,i})\mid \sigma (x',\ell_{o,j})=\lsim \lambda\} = \{p'_{o',i}\mid (p',o')\in \P^{(1)},p'=_{o,i}p,  o'=\ell_{o,1}\}$ 
\end{enumerate}

When $i=j=s$ the first equality remains true while 
$$
   \{x'\in S_{(0,1)}(p_{o,i})\mid \sigma (x',\ell_{o,j})=\lmoins \lambda\} = \{p_{o,i-1}\}\cup \{p'_{o',i}\mid (p',o')\in \P^{(1)},p'=_{o,i}p,  o'=\ell_{o,1}\}$$
\end{lemma}

Finally, we prove the following refinement of Proposition~\ref{prop:double_counting}.

\begin{proposition}\label{prop:double_counting_diamond}
 Assume that the action of $\Gamma$ is type-rotating.
Let $i \ge j \geq s>0$ and $\lambda = (s,i+j-2s)$. Let  $F\colon X^{(1)} \times X \to \C$ be a $\Gamma$-invariant function. Assume that $\Gamma$ is cocompact or $F$ is real non-negative.  We have
\begin{multline}\label{eq:double_counting_+}\sum_{o \in \Omega} \frac{1}{|\Gamma_o|} \E_o[F(p_{o,i},p_{o,i+1},\ell_{o,j}) \mid (p,\ell)_o=s] = \\ \frac{1}{q^2+q+1}\sum_{(x,x')\in \Omega^{(1)}} \frac{1}{|\Gamma_{x,x'}|} \frac{1}{|S_\lambda^+(x,x')|} \sum_{y \in S_\lambda^+(x,x')} F(x,x',y). \end{multline}
\begin{multline}\label{eq:double_counting_-} \sum_{o \in \Omega} \frac{1}{|\Gamma_o|} \sum_{o' \in \L_{o,1}} \E_o[F(p_{o,i},p_{o',i-1},\ell_{o,j}) 1_{\ell_{o,1}=o'} \mid (p,\ell)_o=s] =\\ \frac{1}{q^2+q+1}\sum_{(x,x') \in \Omega^{(1)}} \frac{1}{|\Gamma_{x,x'}|} \frac{1}{|S_\lambda^-(x,x')|} \sum_{y \in S_\lambda^-(x,x')} F(x,x',y).
\end{multline}
\begin{multline}\label{eq:double_counting_sim} \sum_{o \in \Omega} \frac{1}{q|\Gamma_o|} \sum_{o' \in \L_{o,1}} \E_o[F(p_{o,i},p_{o',i-1},\ell_{o,j}) 1_{p \sim_{o,1} o',\ell_{o,1} \neq o'} \mid (p,\ell)_o=s] = \\ \frac{1}{q^2+q+1}\sum_{(x,x')\in \Omega^{(1)}} \frac{1}{|\Gamma_{x,x'}|} \frac{1}{|S_\lambda^\sim(x,x')|} \sum_{y \in S_\lambda^\sim(x,x')} F(x,x',y). \end{multline}
\end{proposition}

\begin{proof}
Let us first prove \eqref{eq:double_counting_+}.  Given $(x',y) \in X$ satisfying $\sigma(x',y)=\lambda^+$, by Corollary~\ref{cor:unique_plus_inverse} there is a unique $x \in X$ such that $(x,x') \in X^{(1)}$ and $\sigma(x,y) = \lambda$.
We can therefore define a $\Gamma$-invariant function $G \colon X\times X \to \C$ by $G(x',y) = F(x,x',y)$ if $\sigma(x,y) = \lambda$ and $\sigma(x',y)=\lambda^+$ and $G(x',y)=0$ otherwise. We apply Proposition~\ref{prop:double_counting} with the values $i+1, j, s$ to the function $G$. Note that the vector in question is then $(s, (i+1) + j - 2s) = \lambda^+$. The left-hand sides of \eqref{eq:double_counting} and \eqref{eq:double_counting_+} coincide because $G(p_{o,i+1},\ell_{o,j}) = F(p_{o,i},p_{o,i+1},\ell_{o,j})$ if $(p,\ell)_o=s$ in view of Lemma~\ref{lemma:respective_position_in_X1}. 

To identify the right-hand sides, we use another double counting argument; let $A=\{(x,x',y)\mid (x,x')\in X^{(1)} \text{ and } y\in S_\lambda^+(x,x')\}$. Let $I$ be the integral of $F$ on $\Gamma\backslash A$ with respect to the measure giving weight $\frac{1}{|S_{\lambda^+}(x')||\Gamma_{x,x',y}|}$ to the orbit of $(x,x',y)$. Let $A_{x'}=\{(x,y)\mid (x,x',y)\in A\}$ ; note that $A_{x'}$ is in bijection with $S_{\lambda^+}(x')$. We have 
\begin{align*}
   I &= \sum_{x'\in \Omega}\sum_{(x,y)\in \Gamma_{x'}\backslash A_{x'}}\frac{1}{|S_{\lambda^+}(x')||\Gamma_{x,x',y}|}F(x,x',y) \\
   &=\sum_{x'\in \Omega}\frac{1}{|\Gamma_{x'}|} \frac{1}{|S_{\lambda^+}(x')|}\sum_{(x,y)\in A_x} F(x,x',y)\\
   &= \sum_{x'\in \Omega}\frac{1}{|\Gamma_{x'}|} \frac{1}{|S_{\lambda^+}(x')|} \sum_{y\in S_{\lambda^+}(x')}G(x',y)
\end{align*}
So the left-hand side of \eqref{eq:double_counting_+} is equal to $I$.

Now let us proceed to \eqref{eq:double_counting_-} and \eqref{eq:double_counting_sim}.
Let $\truc\in \{-,\sim\}$. For $x\in X$ and $y\in S_\lambda(x)$, let $N^\truc$ be the number of $x'\in S_{(0,1)}(x)$ such that $y\in S_\lambda^\truc(x,x')$ (recall that $N^\truc$ can be calculated by Lemma~\ref{CartwrightMlotkowski}, and depends only on $\truc$ and $\lambda$). Consider the function 
\[
K(x,y) = \frac{1}{N^\truc}\sum_{\substack{x' \in S_{(1,0)}(x) \\\sigma(x',y)=\lambda^\truc}} F(x,x',y)\quad\text{if}\quad\sigma(x,y)=\lambda
\]
and $0$ if $\sigma(x,y)\neq \lambda$.
We apply Proposition~\ref{prop:double_counting} to $K$. In view of Corollary~\ref{cor:respectivepositions}, if either $i>s$ or $\truc=-$, the left-hand side of \eqref{eq:double_counting} for $K$ is the same as the one of Proposition~\ref{prop:double_counting_diamond} in each case.
The right-hand side on the other hand is 
\[R := \sum_{x\in \Omega} \frac{1}{|\Gamma_x|} \frac{1}{|S_\lambda(x)|}\sum_{y\in S_\lambda(x)}K(x,y)\]

Now let $A=\{(x,x',y)\in X\mid (x,x')\in X^{(1)} \text{ and } y\in S^\truc_\lambda(x,x')\}$. We consider again the integral $I$ of $F$ on $\Gamma\backslash A$ with respect to the measure given weight $\frac{1}{|\Gamma_{x,x',y}|}$ on the orbit of $(x,x',y)$. Using the orbit formula we get 
\begin{align*}
  I= \sum_{(x,x',y)\in \Gamma\backslash A}\frac{1}{|\Gamma_{x,x',y}|} F(x,x',y) & = \sum_{x\in \Omega}\frac{1}{|\Gamma_x|} \sum_{y\in S_{\lambda}(x)} \sum_{\substack{x' \in S_{(1,0)}(x) \\\sigma(x',y)=\lambda^\truc}} F(x,x',y)\\
   &= \sum_{x\in \Omega} \frac{1}{|\Gamma_x|} \sum_{y\in S_{\lambda(x)} }N^\truc K(x,y)\\
   &= N^\truc |S_{\lambda}(x)| R
\end{align*}
(since $|S_\lambda(x)|$ does not depend on $x$).
Now disintegrating with respect to $(x,x')$ we get
$$I= \sum_{(x,x')\in \Omega^{(1)}} \frac{1}{|\Gamma_{x,x'}|} \sum_{y\in S_{\lambda}^\truc(x,x')}F(x,x',y)$$
which is $(q^2+q+1)|S_\lambda^{\truc}(x,x')|$ times the right-hand side of the three equalities in Proposition~\ref{prop:double_counting_diamond}. Hence in each case, the right-hand side of the equalities is $\frac{(q^2+q+1)N^\truc |S_\lambda(x)|}{S_\lambda^{\truc}(x,x')} R$.

To evaluate this quantity, for $x\in X$, let $B=\{(x',y) \mid x\in S_{(1,0)}(x) \text{ and } y\in S_{\lambda}^\truc(x,x')  \}$. Then $|B|=\sum_{x'\in S_{(1,0)}(x)} |S_{\lambda}^\truc(x,x')| = (q^2+q+1) |S_\lambda^\truc(x,x')|$. On the other hand, $|B|=\sum_{y\in S_\lambda(x)} N^\truc = |S_\lambda(x)| N^\truc$, so that $\frac{(q^2+q+1)N^\truc |S_\lambda(x)|}{S_\lambda^{\truc}(x,x')} =1$, which concludes the proof in the case when $i>s$ or $\truc = -$.

Now assume that $i=j=s$. Then it remains to prove \eqref{eq:double_counting_sim}. Let $R$ be the right-hand side of \eqref{eq:double_counting_sim} (and therefore of \eqref{eq:double_counting_+} since $\lambda^+=\lambda^\sim$), $L_+$ be the left-hand side of \eqref{eq:double_counting_+} and $L_\sim$ be the left-hand side of \eqref{eq:double_counting_sim}. We have already proven that $L_+=R$. Furthermore the above argument, using again Corollary~\ref{cor:respectivepositions}, also proves that $R= \frac{q^2}{q^2+q} L_+ + \frac{q}{q^2+q}L_\sim$. Since $L_+=R$ it follows that $L_\sim=R$, which concludes the proof.
\end{proof}

The following proposition is the analogue of the previous one with $X^{(1)}$ replaced by $\lun{X}$. It is not a formal consequence, by applying the previous proposition to $\overline{X}$, because the roles of points and lines are not swapped.

\begin{proposition}\label{prop:double_counting_diamond_left}
 Assume that the action of $\Gamma$ is type-rotating.
Let $i \ge j \geq s>0$ and $\lambda = (s,i+j-2s)$. Let  $F\colon \lun X \times X \to \C$ be a $\Gamma$-invariant function. Assume that $\Gamma$ is cocompact or $F$ is real non-negative.  We have

\begin{multline}\label{eq:double_counting_lmoins}\sum_{o \in \Omega} \frac{1}{|\Gamma_o|} \E_o[F(p_{o,i},p_{o,i-1},\ell_{o,j}) \mid (p,\ell)_o=s] = \\ \frac{1}{q^2+q+1}\sum_{(x,x')\in \lun \Omega} \frac{1}{|\Gamma_{x,x'}|} \frac{1}{|\lmoins S_\lambda(x,x')|} \sum_{y \in \lmoins S_\lambda(x,x')} F(x,x',y). \end{multline}
\begin{multline}\label{eq:double_counting_lsim} \sum_{o \in \Omega} \frac{1}{|\Gamma_o|} \sum_{o' \in \L_{o,1}} \E_o[F(p_{o,i},p_{o',i},\ell_{o,j}) 1_{\ell_{o,1}=o'} \mid (p,\ell)_o=s] =\\ \frac{1}{q^2+q+1}\sum_{(x,x') \in \lun \Omega} \frac{1}{|\Gamma_{x,x'}|} \frac{1}{|\lsim S_\lambda(x,x')|} \sum_{y \in \lsim S_\lambda(x,x')} F(x,x',y).
\end{multline}
\begin{multline}\label{eq:double_counting_lplus} \sum_{o \in \Omega} \frac{1}{q|\Gamma_o|} \sum_{o' \in \L_{o,1}} \E_o[F(p_{o,i},p_{o',i},\ell_{o,j}) 1_{p \sim_{o,1} o',\ell_{o,1} \neq o'} \mid (p,\ell)_o=s] = \\ \frac{1}{q^2+q+1}\sum_{(x,x')\in \lun\Omega} \frac{1}{|\Gamma_{x,x'}|} \frac{1}{|\lplus S_\lambda(x,x')|} \sum_{y \in \lplus S_\lambda(x,x')} F(x,x',y). \end{multline}
\end{proposition}

\begin{proof}
First, as in Proposition \ref{prop:double_counting_diamond}, we apply Proposition \ref{prop:double_counting} to the function 
\[
K(x,y) = \frac{1}{\ltruc N}\sum_{\substack{x' \in S_{(0,1)}(x) \\\sigma(x',y)=\ltruc \lambda}} F(x,x',y)\quad\text{if}\quad\sigma(x,y)=\lambda
\]
where $\ltruc N$ is the number of $x' \in S_{(0,1)}(x)$ such that $\sigma(x',y)=\ltruc \lambda$ (which is independent of $x$ and $y$).

The same calculation as in Proposition \ref{prop:double_counting_diamond}, using Lemma \ref{lem:respectivepositions_l} instead of Corollary \ref{cor:respectivepositions}, then proves the proposition in the cases when $i>s$ or  $\truc =+$. In the case when $i=j=s$, denoting  $\lmoins L$  the left-hand side of \eqref{eq:double_counting_lmoins},$\lsim L$  the left-hand side of \eqref{eq:double_counting_lsim}, and $R$ their common right-hand side, this argument proves that $R= \frac{1}{1+q}\lmoins L + \frac{q}{1+q}\lsim L$. 

In order to conclude the proof, it therefore suffices to prove that $\lmoins L=R$ (in the case when $i=j=s$). 
We use yet another double counting argument. Let 
$A$ denote the set of all tuples $(o,x,x',y)$ in $X$ such that $x=p_{o,s}$ , $y=\ell_{o,s}$, $x'=p_{o,s-1}$ for some $p\in \P, \ell\in \L$ with $ (p,\ell)_o=s$, 
and let $I$ be the integral of $F(x,x',y)$ on $\Gamma\backslash A$ with respect to the measure giving weight $\frac{1}{|\Gamma_{o,x,x',y}|}$ to the orbit of $(o,x,x',y)$. For $o\in X$, let $A_o=\{(x,x',y)\mid (o,x,x',y)\in A\}$. Note that for every $o$ we have $|A_o| =(q^2+q+1)(q+1)q^{3s-3}$. Indeed, $x$ can be any point of the form $p_{o,s}$, and there are $(q^2+q+1)q^{2(s-1)}$ such points, by Lemma~\ref{lem:NumberEquilateral} (as there are $q^2+q+1$ possible choices for $p_{o,1}$). Then $x'=p_{o,s-1}$ is uniquely determined, and for each choice of $x$ there are $(q+1)q^{s-1}$ possible choices for $y=\ell_{o,s}$, again by Lemma~\ref{lem:NumberEquilateral}.

Thus we get
\begin{align*}
    I&= \sum_{o\in \Omega} \sum_{(x,x',y) \in A_o} \frac{1}{|\Gamma_{o}|} F(x,x',y)\\
    &= (q^2+q+1)(q+1)q^{3s-3}\sum_{o\in \Omega} \frac{1}{|\Gamma_{o}|} \E_o[F(p_{o,s},p_{o,s-1},\ell_{o,s}) \mid (p,\ell)_o=s]
\end{align*}

On the other hand, fix $(x,x')\in \lun X$ and $y\in \lmoins S_\lambda(x,x')$. Then $o\in X$ is such that $(o,x,x',y)\in A$ if and only if $o=\ell_{x,s}$ and $x'=\ell_{x,1}$ for some $\ell \in \L$ with $(\ell_s,y)_x=s$ . By
Lemma~\ref{lem:NumberEquilateral}, it follows that the set  $ \{ o \mid (o,x,x',y)\in A\}$ has cardinality $q^{s-1}$ and therefore 

\begin{align*}
    I&= \sum_{(x,x')\in \lun X} \frac{1}{|\Gamma_{x,x'}|} \sum_{y\in \lmoins S_\lambda(x,x')} q^{s-1} F(x,x',y)
\end{align*}

Since $\lmoins S_\lambda(x,x')$ has cardinality $(q+1)q^{2s-2}$ (using again  Lemma~\ref{lem:NumberEquilateral}) this proves \eqref{eq:double_counting_lmoins} for $i=j=s$, and \eqref{eq:double_counting_lsim} follows by the above argument.
\end{proof}

\section{Convergence to projection on harmonic functions}\label{sec:convergence_harmonic}

Throughout this section we fix an admissible Banach space $E$ and let $L$ and $\theta$ be appropriate constants in \eqref{eq:E-valued_Hoelder_bound}. We consider a proper, cocompact and type-preserving action of a finitely generated group $\Gamma$ on $X$. The case where $\Gamma$ is not cocompact will be dealt with in Section~\ref{sec:non-uniform}. Let also $a$ and $b$ denote the quasi-isometry constants from Lemma~\ref{lem:orbit_map_quasiisometric}. We fix
\begin{equation}\label{eq:alpha_small_enough}
\alpha < \frac{\theta}{3a}
\end{equation}
and let $\pi$ be a representation of $\Gamma$ on $E$ satisfying
\begin{equation}\label{eq:assumption_on_growth_rate_of_pi}
\norm{\pi(\gamma)} \leq C e^{\alpha |\gamma |}\text{.}
\end{equation}

Our goal is to prove the first main ingredient to Theorem~\ref{thm:strong_T_on_building}. Recall from Section~\ref{sec:outline} that $\tilde{E}$ is the induced module associated to $E$ and the action of $\Gamma$ on $X$, that $A_\lambda \colon \tilde{E} \to \tilde{E}$ is the averaging operator on the $\lambda$-sphere $S_\lambda$ around a vertex. Recall also that a function $f \in \tilde{E}$ is $\Lambda_0$-harmonic if it is $A_\lambda$-invariant for all $\lambda \in \Lambda_0$. Observe that, through the duality \eqref{eq:duality_pairing}, $\Lambda_0$-harmonic elements of $\diE$ correspond to elements of $(\iE)^*$ which are $A_\lambda^*$-invariant for all $\lambda \in \Lambda_0$.

\begin{proposition}\label{prop:harmonic_limit}
Assume that $E$ is admissible.
The net $(A_{\lambda})_{\lambda \in \Lambda_0}$ converges in the norm of $B(\iE)$ to an
operator $P$ that is characterized by the following two properties:
\begin{enumerate}
\item $P$ is a projection on the space of $\Lambda_0$-harmonic elements of $\iE$.
\item $P^*$ is a projection on the space of $\Lambda_0$-harmonic elements of $\diE$.
\end{enumerate}
The same is true with $\Lambda_0$ replaced by $\Lambda_1$ or $\Lambda_2$.
\end{proposition}

The main work will be to show that the limit exists, based on the estimates from the last section.

For $o \in \Omega$ and $i,j \in \N$ we define operators $I_{o,i} \colon \diE \to L^2(\P;E^*)$ and $J_{o,j} \colon \tilde{E} \to L^2(\L;E)$ by $I_{o,i}(g)(p) = g(p_{o,i})$ and $J_{o,j}(f)(\ell) = f(\ell_{o,j})$.

\begin{lemma}\label{lem:norm_of_alphai}
For $o \in \Omega$ and $i, j \in \N$ the operators $I_{o,i}$ and $J_{o,j}$ have operator norms bounded by 
\[ \norm{I_{o,i}} \le \max_{x \in X} \sqrt{|\Gamma_{x}|} Ce^{\alpha(ai + b)}.\]
and
\[ \norm{J_{o,j}} \le \max_{x \in X} \sqrt{|\Gamma_{x}|} Ce^{\alpha(aj + b)}.\]
More precisely,
\[
\left(\sum_{o \in \Omega} \frac{1}{\abs{\Gamma_o}} \norm{I_{o,i}(g)}^2\right)^{\frac 1 2} \le Ce^{\alpha(ai + b)} \|g\|,
\]
and similarly for $J_{o,j}(f)$.
\end{lemma}
\begin{proof}
Let $p \in \P$ and $g \in \diE$. By Lemma \ref{lem:orbit_map_quasiisometric}, there is $\gamma \in \Gamma$ such that $\gamma p_{o,i} \in \Omega$ and $|\gamma| \leq ai+b$. This implies that $g(p_{o,i}) = \pi(\gamma)^* g(\gamma p_{o,i})$ and
\[
\norm{g(p_{o,i})}_{E^*} \leq \norm{\pi(\gamma)} \norm{g(\gamma p_{o,i})}_{E^*} \leq C e^{\alpha(ai+b)} \norm{g(\gamma p_{o,i})}\text{.}
\]
We can therefore bound $\sum_{o \in \Omega} \frac{1}{\abs{\Gamma_o}} \norm{I_{o,i}(g)}^2$ by
\[
\sum_{o \in \Omega}\sum_{o' \in \Omega} \frac{1}{\abs{\Gamma_o}} C^2 e^{2\alpha(ai+b)}\|g(o')\|^2 \proba_o(p_{o,i} \in \Gamma o'),\]
which is equal to $C^2e^{2\alpha(ai+b)} \norm{g}^2$ by rearranging the terms. The argument for $J_{o,s}$ is analogous.
\end{proof}

The following lemma reduces the proof that $(A_\lambda)_{\lambda \in \Lambda_i}$ converges to establishing two estimates \eqref{eq:vertical_estimates} and \eqref{eq:nonvertical_estimates}. Its proof follows \cite[Proposition 3.6]{Lafforgue08}. The reason for formulating it separately is that it will be used again in the following sections.

\begin{figure}
\begin{center}
\begin{tikzpicture}[scale=.7]
\foreach \r in {0,...,10}
\foreach \s[evaluate={\t=int(mod(99+\r - \s,3));\ma=int(max((\r+\s),10))}] in {0,...,10}
{%
\ifthenelse{\ma=10}{
\ifthenelse{\t=0}{
\node[dot,fill=white] at ($(60:\r)+(120:\s)$) {};}
{\ifthenelse{\t=1}
{\node[dot,fill=gray] at ($(60:\r)+(120:\s)$) {};}
{\node[dot] at ($(60:\r)+(120:\s)$) {};}}
{}}}
\node[anchor=north] at (0,0) {$(0,0)$};
\node[anchor=east] at (120:1) {$(1,0)$};
\node[anchor=west] at (60:1) {$(0,1)$};
\draw[dashed] (60:3) -- (120:6);
\node[anchor=west] at (60:3) {$\{(s,t) \mid s+ 2t =6\}$};
\draw[dashed] (0,0) -- ($(60:5)+(120:5)$);
\node[anchor=south] at ($(60:5)+(120:5)$) {$\{(s,t) \mid s = t\}$};
\end{tikzpicture}
\end{center}
\caption{Part of $\Lambda$. The elements of $\Lambda_0$ are drawn in white. Some subsets that occur in the proof of Lemma~\ref{lem:alambda_converge} are indicated.}
\label{fig:lambda_paths}
\end{figure}
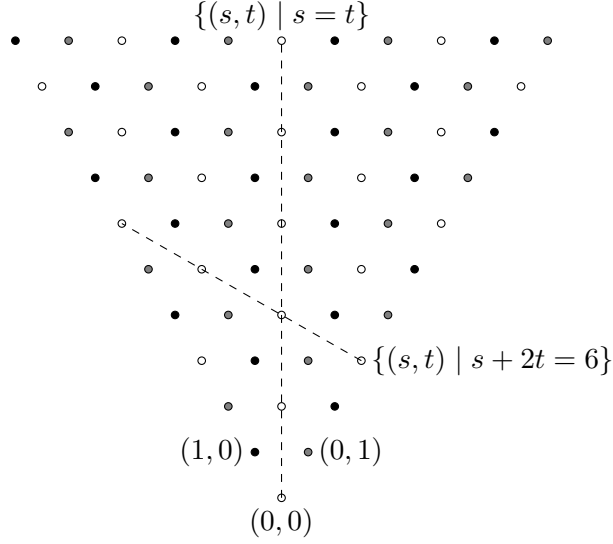

\begin{lemma}[Zig-zag argument]\label{lem:lafforgue}
Let $B$ be a complete metric space with metric $d$. Let $(A_\lambda)_{\lambda \in \Lambda}$ be a net in $B$ satisfying the following conditions: there exists $C, D, E \geq 0$ with $E<D$ such that for every $(s,t)\in \mathbb N^2$ we have
\begin{align}\label{eq:vertical_estimates}
d(A_{(s,t)},A_{(s+1,t-2)}) \leq C e^{-sD +tE}&& \text{if }t \ge 2\text{,}\\
\label{eq:nonvertical_estimates}d(A_{(s,t)},A_{(s-2,t+1)}) \leq C e^{-tD +sE}&& \text{if }s \ge 2\text{.}
\end{align} 
The net $(A_{\lambda})_{\lambda \in \Lambda_0}$ converges. In other words, if $(\lambda_n)_{n \in \N}$ tends to infinity in $\Lambda_0$ then $(A_{\lambda_n})_{n \in \N}$ converges to an element $A_\infty$ that does not depend on the sequence. The same is true with $\Lambda_0$ replaced by $\Lambda_1$ or $\Lambda_2$.

Moreover, the limits $P_i$ satisfy 
\begin{equation}\label{eq:convergence_rate_Alambda} d(A_\lambda,P_i) \leq \frac{3C e^{4D}}{e^D-e^E} e^{-\max(s+2t,t+2s) \frac{D-E}{3}}\end{equation} whenever $\lambda=(s,t)$ belongs to $\Lambda_i$.
\end{lemma}

\begin{proof}
We first prove the statement for $\Lambda_0$. It follows from \eqref{eq:vertical_estimates} and \eqref{eq:nonvertical_estimates} that
\begin{align*}
d(A_{(s,t)},A_{(s+k,t-2k)}) &\leq \frac{C}{1-e^{-(D+2E)}}  e^{(-sD+tE)} &&\text{if } 0 \leq k \leq \frac{t}{2}\text{ and}\\
d(A_{(s,t)},A_{(s-2k,t+k)}) &\leq \frac{C}{1-e^{-(D+2E)}} e^{(-tD + sE)} && \text{if } 0 \leq k \leq \frac{s}{2}.
\end{align*}

These inequalities say that the variation of $A_\lambda$ on the strips $\{(s,t) \in \Lambda \mid t \leq s+10, t+2s = n\}$ and $\{(s,t) \in \Lambda \mid s \leq t+10, 2t+s = n\}$ is exponentially small in $n$, see Figure~\ref{fig:lambda_paths} for orientation. 
In particular, considering the intersection of these lines with the line $s=t$, we deduce for $(s,t) \in \Lambda_0$ that
\begin{align*}
d(A_{(s,t)}, A_{(\frac{t+2s}{3},\frac{t+2s}{3})}) &\leq C' e^{-\frac{t+2s}{3}(D-E)} &&\textrm{ if } 0 \leq t \leq s\text{ and}\\
d(A_{(s,t)}, A_{(\frac{2t+s}{3},\frac{2t+s}{3})}) &\leq C' e^{-\frac{2t+s}{3}(D-E)} &&\textrm{ if } 0 \leq s \leq t
\end{align*}
where $C' = \frac{C}{1-e^{-(D+2E)}}$. We also deduce
\[
d(A_{s,s}, A_{s+1,s+1}) \leq d(A_{s,s}, A_{s+2,s-1}) + d(A_{s+2,s-1}, A_{s+1,s+1}) \leq C'' e^{-s(D-E)}
\]
where $C''=C (e^{D+2E} +e^{-D+E})$. It follows that for $(s,t), (s',t') \in \Lambda_0$
\begin{align*}
d(A_{s,t}, A_{s',t'}) &\le d(A_{s,t}, A_{u,u}) + d(A_{u,u}, A_{u',u'}) + d(A_{u',u'}, A_{s',t'})\\
& \leq C' e^{-u(D-E)} + \frac{C''}{1-e^{-(D-E)}}) e^{-\min(u,u')(D-E)} + \leq C' e^{-u(D-E)}
\end{align*}
where $u = (t+2s)/3$ if $s \ge t$ and $u = (2t+s)/3$ if $t \ge s$ and similarly for $u'$. The lemma and the estimate \eqref{eq:convergence_rate_Alambda} follow, using the rough estimates $C' \leq \frac{C e^D}{e^D-e^{D}}$ and $C'' \leq 2C e^{3D}$.

The only modification needed for $\Lambda_1$ and $\Lambda_2$ is that $\frac{t+2s}{3} \in \{(t,s) \mid t+2s = n\} \cap \{(t,s) \mid t=s\}$ needs to be replaced by the intersection of $\{(t,s) \mid t+2s = n\}$ with $\{(t,s) \mid t = s+1\}$ respectively $\{(t,s) \mid t+1 = s\}$.
\end{proof}

\begin{lemma}\label{lem:alambda_converge}
The net $(A_{\lambda})_{\lambda \in \Lambda_0}$ converges. In other words, if $(\lambda_n)_{n \in \N}$ tends to infinity in $\Lambda_0$ then $(A_{\lambda_n})_{n \in \N}$ converges to an operator that does not depend on the sequence. The same is true with $\Lambda_0$ replaced by $\Lambda_1$ or $\Lambda_2$.
\end{lemma}

\begin{proof}
Let $i,j,s$ be such that $i,j > s$. Write $t = i+j-2s$. Let further $f \in \iE$ and $g \in \diE$ be arbitrary. 
We define $F(x,y) = \gen{g(x),f(y)}$ and note that this is a $\Gamma$-equivariant function by Lemma \ref{lem:fundamental_domain_irrelevant}. So we can apply Proposition~\ref{prop:double_counting}. The left-hand side of the equation in the conclusion is
\begin{align*}
&\mathrel{\phantom{=}}\sum_{o\in\Omega} \frac{1}{\abs{\Gamma_o}} \E_o[\gen{g(p_{o,i}),f(\ell_{o,j})} \mid (p,\ell)_o=s]\\
&=\sum_{o\in\Omega}\frac{1}{\abs{\Gamma_o}}\E_o[\gen{I_{o,i}(g)(p),J_{o,j}(f)(\ell)} \mid (p,\ell)_o=s]\\
&=\sum_{o\in\Omega}\frac{1}{\abs{\Gamma_o}} \gen{I_{o,i}(g),T_{o,s} J_{o,j}(f)}
\end{align*}
by definition of $T_{o,s}$ (Section~\ref{sec:norm_estimates_projective}).

The right-hand side is,
\begin{align*}
&\mathrel{\phantom{=}}\sum_{x \in \Omega} \frac{1}{\abs{\Gamma_x}} \frac{1}{\abs{S_\lambda(x)}} \sum_{y \in S_\lambda(x)} \langle g(x),f(y)\rangle\\
&= \sum_{x\in\Omega} \frac{1}{\abs{\Gamma_x}} \langle g(x),A_{(s,t)} f(x)\rangle\\
&=\langle g,A_{(s,t)} f\rangle
\end{align*}
by definition of $A_{\lambda}$ and the pairing between $\diE$ and $\iE$ (Section~\ref{sec:outline}).

In summary we have
\begin{equation}\label{eq:a_lambda_t_s}
\sum_{o \in \Omega} \frac{1}{\abs{\Gamma_o}} \gen{I_{o,i}(g),T_{o,s}J_{o,j}(f)} = \gen{ g,A_{(s,t)} f}\text{.}
\end{equation}

We can apply Lemma~\ref{lem:norm_of_alphai} and take the difference of \eqref{eq:a_lambda_t_s} for $s$ and $s+1$ and obtain
\begin{align*}
\abs{\gen{g,(A_{s,t} - A_{s+1,t-2})(f)}} &\le \sum_{o \in \Omega}\frac{1}{\abs{\Gamma_o}} \norm{T_{o,s} - T_{o,s+1}} \norm{I_{o,i}(g)} \norm{J_{o,j}(f)}\\
& \le L e^{-\theta s}C^2 e^{2\alpha b + \alpha(i+j)} \norm{f} \norm{g}
\end{align*}
by \eqref{eq:E-valued_Hoelder_bound} which holds by the standing assumption from the beginning of the section.

Taking the supremum over $f$ and $g$ in the unit balls of $\iE$ and $\diE$ respectively and using Lemma \ref{lemma:dual} we obtain
\[
\|A_{(s,t)} - A_{(s+1,t-2)}\| \leq C'' e^{-s(\theta - 2\alpha a) +t\alpha a}
\]
for every $(s,t) \in \Lambda$ with $t \geq 2$, where $C''=C'LC^2 e^{2\alpha b}$.

By applying the same argument, but in the building $\overline X$, we get that 
$
\|\overline A_{(t,s)} - \overline A_{(t+1,s-2)}\| \leq C'' e^{-t(\theta - 2\alpha a) +s\alpha a}
$,
that is
\[
\|A_{(s,t)} - A_{(s-2,t+1)}\| \leq C'' e^{-t(\theta - 2\alpha a) +s\alpha a}
\]
for every $(s,t) \in \Lambda$ with $s \ge 2$. The claim now follows from Lemma~\ref{lem:lafforgue}, as the assumption \eqref{eq:alpha_small_enough} is equivalent to the inequality $\alpha a < \theta - 2\alpha a$.
\end{proof}

\begin{proof}[Proof of Proposition~\ref{prop:harmonic_limit}]
Let $P$ be the limit of the $A_\lambda$ over $\lambda \in \Lambda_0$ which exists by Lemma~\ref{lem:alambda_converge}.

We first show that every function in the image of $P$ is $\Lambda_0$-harmonic, equivalently that $A_\lambda P=P$ for every $\lambda \in \Lambda_0$. Fix $\lambda \in \Lambda_0$, and take another $\lambda' \in \Lambda_0$ which will go to infinity. By (a simple induction on) \cite[Proposition 2.3]{CartwrightMlotkowski94} we can decompose $A_\lambda A_{\lambda'}$ as an average $\sum_{\lambda''} p_{\lambda,\lambda'}(\lambda'') A_{\lambda''}$ with each $\lambda'' \in \Lambda_0$ lying at $\ell_1$-distance $\leq \abs{\lambda}_1$ from $\lambda'$. In the limit $\lambda' \to \infty$ it follows that $A_\lambda P=P$.

Letting now $\lambda$ go to infinity, we get $P^2 = P$, that is, $P$ is a projection. The argument for $P^*$ is analogous.
\end{proof}

The proof of Proposition~\ref{prop:harmonic_limit} gives slightly more. To formulate it, for $k \in \mathbb Z/3\mathbb Z$ let $P_k = \lim_{\lambda \in \Lambda_k} A_\lambda$ which exist by Lemma~\ref{lem:alambda_converge}.

\begin{corollary}
For $i,j \in \mathbb Z/3\mathbb Z$ and $\lambda \in \Lambda_i$ we have $A_\lambda P_j = P_{i+j} = P_j A_\lambda$ and $P_i P_j = P_{i+j}$.
\end{corollary}

\begin{proof}
As in the proof of Proposition~\ref{prop:harmonic_limit} $A_\lambda A_\lambda'$ can be decomposed as an average over $A_{\lambda''}$. If $\lambda \in \Lambda_k$ and $\lambda \in \Lambda_l$ then the $\lambda''$ lie in $\Lambda_{k+l}$. Letting now $\lambda'$ go to infinity gives the first equation. Letting then $\lambda$ go to infinity gives the third equation. The second equation is obtained a decomposition with $\lambda''$ lying at $\ell_1$-distance $\le \abs{\lambda'}_1$ from $\lambda$ and letting $\lambda$ go to infinity.
\end{proof}

Note that since $A_{(0,0)}$ is the identity operator, saying that $f \in \iE$ is $\Lambda_0$-harmonic can also be phrased as saying that $A_\lambda f$ is independent of $\lambda \in \Lambda_0$. More generally we have:

\begin{corollary}\label{cor:about_harmonic_functions}
Let $k \in \mathbb Z/3\mathbb Z$. If $f \in\iE$ is $\Lambda_0$-harmonic then $A_\lambda f$ is independent of $\lambda \in \Lambda_k$. If $g \in \diE$ is $\Lambda_0$-harmonic then $A_\lambda^* g$ is independent of $\lambda \in \Lambda_k$.
\end{corollary}

\begin{proof}
We have $P_k(f) = A_\lambda P_0(f) = A_\lambda(f)$ and similarly for $g$.
\end{proof}

\section{Harmonic functions are constant}\label{sec:harmonic_constant}

As in the previous section, we fix an admissible Banach space $E$ throughout and let $L$ and $\theta$ be the constants in \eqref{eq:E-valued_Hoelder_bound}. We also fix a proper, cocompact and type-preserving action of a finitely generated group $\Gamma$ on $X$. Again, the non-cocompact case will be dealt with in Section~\ref{sec:non-uniform}. In this section we will complete the proof of Theorem~\ref{thm:strong_T_on_building} by showing:

\begin{proposition}\label{prop:harmonic=constants}
Assume that $E$ is admissible, and that $\pi$ is a representation of $\gamma$ on $E$ satisfying~\eqref{eq:assumption_on_growth_rate_of_pi} for some $C>0$ and $\alpha>0$ satisfying~ \eqref{eq:alpha_small_enough}
Every $\Lambda_0$-harmonic element of $\iE$ is constant on vertices of each type.
\end{proposition}

\begin{remark}\cite[Lemme 3.5]{Lafforgue08} The proposition is very easy in the case when $E$ is strictly convex and $\pi$ is a representation by isometries. Indeed, in that case if $f \in \tilde{E}$ and $k \in \Z/3\Z$, the fact that there are only finitely many $\Gamma$-orbits of vertices implies that there is $x \in X$ of type $k$ such that $\|f(x)\|$ is maximal among all vertices of type $k$. If $f$ is moreover assumed to be $\Lambda_0$-harmonic, then for every $\lambda \in \Lambda_0$, the hypothesis $f(x) = A_\lambda f(x)$ expresses $f(x)$ as an average of vectors $f(y)$ for $y \in S_\lambda(x)$ of norm less than $\|f(x)\|$. By strict convexity, this is possible only if all these vectors are equal to $f(x)$. Since $\lambda \in \Lambda_0$ is arbitrary, this proves that $f(y) = f(x)$ for every $y$ of type $k$, and that $f$ is constant on each type.
\end{remark}

In the general case, the proof is much more involved, and is an elaboration of the proof of Proposition \ref{prop:harmonic_limit}.
Consider the Banach spaces associated to $X^{(1)}$:
\begin{align*}
\iE^{(1)} &= \{f \colon X^{(1)} \to E\mid \forall \gamma \in \Gamma, (x,x') \in X^{(1)}, f(\gamma x,\gamma x') = \pi(\gamma) f(x,x')\}\text{ and}\\
\diE^{(1)} &= \{g \colon X^{(1)} \to E^*\mid \forall \gamma \in \Gamma, (x,x') \in X^{(1)}, g(\gamma x,\gamma x') = \pi(\gamma^{-1})^* g(x,x')\}\text{.}
\end{align*}
We equip $\iE^{(1)}$ and $\diE^{(1)}$ with the norms
\[
\norm{f} \defeq \bigg(\sum_{(x,x') \in \Omega^{(1)}} \frac{1}{\abs{\Gamma_{x,x'}}} \norm{f(x,x')}^2\bigg)^{\frac{1}{2}}\quad\text{and}\quad\norm{g} \defeq \bigg(\sum_{(x,x') \in \Omega^{(1)}}\frac{1}{\abs{\Gamma_{x,x'}}}\norm{g(x,x')}^2\bigg)^{\frac{1}{2}}\text{.}
\]

We define analogously $\lun\iE$ and $\lun \diE$ with analogous formulas, but with $X^{(1)}$ replaced by $\lun X$.

Recall that we defined operators $A_\lambda^\truc$ in \eqref{eq:def_Alambdatruc} and $\ltruc A_\lambda$ in \eqref{eq:def_ltrucAlambda} by the formulas
\begin{equation*}
A_\lambda^\truc(f)(x,x') = \frac{1}{\abs{S_\lambda^\truc(x,x')}} \sum_{y \in S_\lambda^\truc(x,x')}f(y).
\end{equation*}
and
\begin{equation*}
\ltruc A_\lambda (f)(x,x') = \frac{1}{\abs{\ltruc S_\lambda(x,x')}} \sum_{y \in \ltruc S_\lambda(x,x')}f(y).
\end{equation*} 
It is easy to see that these operators map equivariant functions to equivariant functions, and so restrict to operators (still denoted the same) $A_\lambda^\truc:\iE\to \iE^{(1)}$ and $\ltruc A_\lambda:\iE\to \lun \iE$.

For $\lambda \in \Lambda$ and $\truc \in \{+,\sim,-\}$ denote, in consistence with Notation \ref{notation:bar}, $\overline{S}_\lambda^\truc(x,x')$ and $\overline {A}_\lambda^\truc$ the spheres and operators defined above, but for $\overline{X}$. It follows from Lemma~\ref{lemma:spheres_opposite_building} that $\overline A_\lambda^\truc = \ltruc A_{\overline \lambda}$, and that $\ltruc\overline{A}_\lambda= A_{\overline \lambda}^\truc$.

Proposition~\ref{prop:double_counting_diamond} motivates the definition of operators
\[
I_{o,i}^+, I_{o,o',i}^\simeq \colon \diE^{(1)} \to L^2(\P;E^*) \quad \text{and} \quad J^-_{o,o',j}, J^\sim_{o,o',j} \colon \tilde{E} \to L^2(\L;E)
\]
for $o, o' \in X$ with $\sigma(o,o') = (0,1)$ and $i,j \in \N$
by $I_{o,i}^+g(p)=g(p_{o,i},p_{o,i+1})$ and $I_{o,o',i}^\simeq g(p)= g(p_{o,i},p_{o',i-1})1_{p \sim_{o,1} o'}$ as well as $J_{o,o',j}^-f(\ell) = f(\ell_{o,j})1_{\ell_{o,1}=o'}$ and $J_{o,o',j}^\sim f(\ell) = f(\ell_{o,j})1_{\ell_{o,1}\neq o'}$. In particular, $J^-_{o,o',j} + J^\sim_{o,o',j} = J_{o,j}$. 

Similarly define $\lmoins I_{o,i}g(p)=g(p_{o,i},p_{o,i-1})$, $\lplussim I_{o,o',i} g(p)=g(p_{o,i},p_{o',i})1_{p \sim_{o,1} o'}$ motivated by Proposition~\ref{prop:double_counting_diamond_left}.

\begin{lemma}\label{lem:norm_of_alphai2}
For $o \in \Omega$, $o' \in \L_{o,1}$ and $i,j \in \N$ the above operators have operator norms bounded by
\[
\norm{I_{o,i}^+}, \norm{I_{o,o',i}^\simeq}, \norm{\lmoins I_{o,i}g}, \norm{\ \lplussim I_{o,o',i} g} \le Ce^{\alpha(ai + b)}
\]
and
\[\norm{J^-_{o,o',j}}, \norm{J^\sim_{o,o',j}} \le Ce^{\alpha(aj + b)}
\]
for some constant $C$.
\end{lemma}

Now the following lemma follows from Proposition~\ref{prop:double_counting_diamond} as Lemma~\ref{lem:alambda_converge} follows from Proposition~\ref{prop:double_counting}.

\begin{lemma}\label{lem:alambda_converge2}
For each $\truc \in \{+,-,\sim\}$, the net $(A_{\lambda}^\truc)_{\lambda \in \Lambda_0}$ converges. In other words, if $(\lambda_n)_{n \in \N}$ tends to infinity in $\Lambda_0$ then $(A_{\lambda_n}^\truc)_{n \in \N}$ converges to an operator that does not depend on the sequence. The same is true with $\Lambda_0$ replaced by $\Lambda_1$ or $\Lambda_2$.
\end{lemma}
\begin{proof}
Let $i,j,s$ be such that $i \ge j \ge s>0$ and write $t = i+j-2s$. Let $f \in \iE$ and $g \in \diE^{(1)}$. We apply Proposition \ref{prop:double_counting_diamond} to the $\Gamma$-invariant function $F(x,x',y) = \langle g(x,x'),f(y)\rangle$. Then \eqref{eq:double_counting_+}, \eqref{eq:double_counting_-} and \eqref{eq:double_counting_sim} imply
\begin{align*}
(q^2+q+1)\sum_{o \in \Omega} \frac{1}{|\Gamma_o|} \gen{I_{o,i}^+g,T_{o,s} J_{o,j}f} &= \gen{g,A_{\lambda}^+ f}\text{,}\\
(q^2+q+1)\sum_{o \in \Omega} \frac{1}{|\Gamma_o|} \sum_{o' \in \L_{o,1}} \gen{  I_{o,o',i}^{\simeq}g, T_{o,s} J_{o,o',j}^-f}&= \gen{g,A_{\lambda}^- f}\text{,}\\
(q^2+q+1)\sum_{o \in \Omega} \frac{1}{q|\Gamma_o|} \sum_{o' \in \L_{o,1}} \gen{ I_{o,o',i}^{\simeq}g, T_{o,s} J_{o,o',j}^\sim f}&= \gen{g,A_{\lambda}^\sim f}\text{.}
\end{align*}
Taking the difference between $s$ and $s+1$ and using, taking the supremum over $f$ and $g$ of norm $1$, and using Lemma~\ref{lem:norm_of_alphai2} we get for $\truc \in \{+,-,\sim\}$
\begin{equation}
\label{eq:vertical_estimatesbis} \|A_{(s,t)}^\truc - A_{(s+1,t-2)}^\truc\| \leq C e^{-s(\theta - 2\alpha a) +t\alpha a}\textrm{ if }t \geq 2 \textrm{ and }s>0\text{,}
\end{equation}
for some constant $C$.

Similarly, replacing Proposition \ref{prop:double_counting_diamond} by Proposition \ref{prop:double_counting_diamond_left}, we get that 
\begin{align*}
(q^2+q+1)\sum_{o \in \Omega} \frac{1}{|\Gamma_o|} \gen{\lmoins I_{o,i}g,T_{o,s} J_{o,j}f} &= \gen{g,\lmoins A_{ \lambda}f}\text{,}\\
(q^2+q+1)\sum_{o \in \Omega} \frac{1}{|\Gamma_o|} \sum_{o' \in \L_{o,1}} \gen{ \lplussim I_{o,o',i}g,T_{o,s} J_{o,o',j}^-f}&= \gen{g,\lsim A_{\lambda} f}\text{,}\\
(q^2+q+1)\sum_{o \in \Omega} \frac{1}{q|\Gamma_o|} \sum_{o' \in \L_{o,1}} \gen{\lplussim I_{o,o',i} g, T_{o,s}  J_{o,o',j}^\sim f}&= \gen{g,\lplus A_{\lambda} f}\text{.}
\end{align*}

It follows that for $\truc \in \{+,-,\sim\}$
\begin{equation}
\|\ltruc A_{(s,t)}- \ltruc A_{(s+1,t-2)}\| \leq C e^{-s(\theta - 2\alpha a) +t\alpha a}\textrm{ if }t \geq 2 \textrm{ and }s>0\text{,}
\end{equation}

By applying this equation in the building $\overline X$, we get that for every $t>0$ and $s\geq 2$, 
\begin{equation*} \|\ltruc \overline A_{(t,s)} - \ltruc \overline A_{(t+1,s-2)}\| \leq C e^{-t(\theta - 2\alpha a) +s\alpha a}
\end{equation*}

Since $\ltruc \overline A_{(s,t)}= A^\truc_{(t,s)}$, we finally get that 

\begin{equation}
\label{eq:nonvertical_estimatesbis} \|A_{(s,t)}^\truc - A_{(s-2,t+1)}^\truc\| \leq C e^{-t(\theta - 2\alpha a) +s\alpha a}\textrm{ if } s\geq 2 \textrm{ and }t>0\text{.}
\end{equation}

The lemma now follows from \eqref{eq:vertical_estimatesbis}, \eqref{eq:nonvertical_estimatesbis} and Lemma~\ref{lem:lafforgue}.
\end{proof}

\begin{proof}[Proof of Proposition \ref{prop:harmonic=constants}]   
For $k \in \mathbb Z/3\mathbb Z$ and $\truc \in \{+,\sim,-\}$, denote by $P^\truc_k\in B(\iE,\iE^{(1)})$ the limit of $A_\lambda^\truc$ as $\lambda \in \Lambda_k$ tends to infinity, which exists by Lemma~\ref{lem:alambda_converge2}.

Now we use the important observation from Lemma \ref{lem:important_fact}: $A_\lambda^\sim = A_\lambda^+$ if $\lambda=(i,0)$, and $A_\lambda^\sim=A_\lambda^-$ if $\lambda = (0,j)$. Taking $\lambda = (i,0) \to \infty$ respectively $\lambda = (0,j) \to \infty$ we find that $P_k^+ = P_k^\sim=P_k^-$. In particular, $\lim_{\lambda \to \infty} \norm{A_\lambda^+-A_\lambda^\sim} = \lim_{\lambda \to \infty} \norm{A_\lambda^--A_\lambda^\sim} = 0$.

For $f \in \iE$, if we partition $S_\lambda(x)$ according to the value of $\sigma(x',y)$ as in Lemma~\ref{CartwrightMlotkowski}, we can write $A_\lambda f(x)$ as a convex combination of $A_\lambda^+ f(x,x')$, $A_\lambda^\sim f(x,x')$ and $A_\lambda^- f(x,x')$, and we deduce that for every $(x,x') \in X^{(1)}$ and $\truc \in \{+,-,\sim\}$
\begin{equation}\label{Alambda-_depends_hardly_on_x'}
\lim_\lambda \| A_\lambda^\truc f(x,x') - A_\lambda f(x)\|=0.
\end{equation}
Let $f \in \iE$ be $\Lambda_0$-harmonic. We have to prove that $f$ is constant on each type. Let us fix $(x,x') \in X^{(1)}$. Let $i$ be a positive integer. Then we know that $f(x') = A_{(i,i)} f(x')$ is the average of $f$ on $\{y\mid \sigma(x',y)=(i,i)\}$. By Lemma~\ref{lem:relation_Alambda_Alambdatruc} we see that $f(x')$ is some convex combination of $A_{(i,i-1)}^+ f(x,x')$, $A_{(i+1,i)}^- f(x,x')$ and $A_{(i-1,i+1)}^\sim f(x,x')$. By \eqref{Alambda-_depends_hardly_on_x'}, we obtain that $f(x')$ is close (as $i \to \infty$) to some convex combination of $A_{(i,i-1)} f(x)$, $A_{(i+1,i)} f(x)$ and $A_{(i-1,i+1)} f(x)$. But we know from Corollary \ref{cor:about_harmonic_functions} that each of the terms $A_{(i,i-1)} f(x)$, $A_{(i+1,i)} f(x)$ and $A_{(i-1,i+1)} f(x)$ are equal to $A_{(1,0)}f(x)$. This proves $f(x') = A_{(1,0)}f(x)$. In particular, for every $x \in X$, the value of $f(x')$ is constant on $\{x'\mid \sigma(x,x')=(1,0)\}$.

Similarly, using $\ltruc A_\lambda$ instead of $A_\lambda^\truc$, we get that for every $x \in X$, the value of $f(x')$ is constant on $\{x'\mid \sigma(x,x')=(0,1)\}$. In other words, $f(x_1) = f(x_2)$ whenever $x_1,x_2$ have same type and share a common neighbour. This implies that $f$ is constant on each type (the graph obtained by joining two vertices of the same type if they have a common neighbour in $X$ is connected), and proves Proposition \ref{prop:harmonic=constants}.
\end{proof} 

\section{Strong (T) for non-cocompact lattices}\label{sec:non-uniform}

\subsection{Non-cocompact lattices}
So far we have only considered groups acting properly and cocompactly on $X$, that is, \emph{cocompact} (or uniform) $\tilde{A}_2$-lattices.  More generally, we define
\begin{definition}\label{def:A2lattice}
An $\tilde{A}_2$-lattice is a group $\Gamma$ that acts properly on an $\tilde{A}_2$-building $X$ with the property that, a (not necessarily finite) set $\Omega\subset X$ of vertex-orbit representatives satisfies
\[
\sum_{x \in \Omega} \frac{1}{\abs{\Gamma_x}} <\infty.
\]
\end{definition}
An example for a non-cocompact $\tilde{A}_2$-lattice is $\SL_3(\F_q[t^{-1}]) \le \SL_3(\F_q\lseries{t})$. All non-arithmetic lattices known at this point are cocompact. However, this is indicative of the fact that cocompact non-arithmetic lattices are easier to construct (as fundamental groups of finite graphs of groups) than non-cocompact ones, so there is no reason to assume that non-cocompact non-arithmetic lattices should not exist. We therefore prove the main theorem for non-cocompact lattices as well.

We will need that $\tilde{A}_2$-lattices are exponentially integrable in the following sense. This generalizes \cite[Theorem 5.3]{dlSActa19}.
Let us fix once and for all an origin $x_0 \in X$, and let $\Omega \subset X$ be a fundamental domain obtained by choosing, in every vertex orbit, a point at minimal distance from $x_0$, in formulae,
\begin{equation}\label{eq:Dirichlet_fundamental_domain}
   \forall x \in \Omega, \forall \gamma \in \Gamma, d(x,x_0) \leq d(\gamma x,x_0)\text{.}
\end{equation}
\begin{lemma}\label{lem:exponential_integrability}
If $\Gamma$ is an $\tilde{A}_2$-lattice, then the above fundamental domain satisfies 
\[
\sum_{x \in \Omega} \frac{1}{|\Gamma_x|} c^{d(x,x_0)} <\infty
\]
for every $c$ with $\abs{c} <\sqrt{q} \frac{q^2+q+1}{q^2+q+q^{3/2}}$.
 
In particular, there is $\varepsilon>0$ independent from $\Gamma$ such that
\begin{equation}\label{eq:exponential_integrability}
\sum_{x \in \Omega} \frac{1}{|\Gamma_x|} e^{\varepsilon d(x,x_0)} <\infty.
\end{equation}
\end{lemma}
In the same way as property (T) is inherited by arbitrary lattices in locally compact groups, we have the following generalization of \cite{CartwrightMlotkowskiSteger}, which is probably well-known to experts (and that we prove below for convenience).
\begin{theorem}\label{thm:propertyT} $\tilde{A}_2$-lattices have property (T). In particular, they are finitely generated.
\end{theorem}

Both Lemma~\ref{lem:exponential_integrability} and  Theorem~\ref{thm:propertyT} will be deduced from a spectral estimate, which we now present, after introducing some notations. The proof we give relies on \cite{CartwrightMlotkowskiSteger} and \cite{Zuk}, but the ideas in the proof of Theorem~ \ref{thm:strong_T_on_building} could also be applied if we did not care about the constant.

Restricting to a finite index subgroup of $\Gamma$, we can assume that the action of $\Gamma$ preserves the types.
Let $(\pi,E)$ be a unitary representation of $\Gamma$. Anticipating the notation and construction in the next subsection, we denote by $\iE_2$ the Hilbert space of $\Gamma$-invariant functions $f \colon X \to E$ such that $\|f\|^2:=\sum_{o \in \Omega} \frac{1}{|\Gamma_o|} \|f(o)\|^2<\infty$. Then the operators $A_\lambda \in B(\iE_2)$ defined by the formula \eqref{eq:def_Alambda} form a unitary representation of the $*$-algebra $\mathcal{A}$ considered in \cite{CartwrightMlotkowski94} (the word unitary means that $A_{(i,j)}^* = A_{(j,i)}$). It follows that the operator $A_{\lambda}$ is a normal operator.

\begin{lemma}\label{lem:spectrum}
The spectrum of $A_{(1,0)}$ is contained in 
\begin{equation}\label{eq:spectrum_of_A10}
   \{1,e^{2i\pi/3}, e^{-2i\pi/3}\} \cup \bigg\{z \mathrel{\bigg\vert} |\Re z|\leq \rho_q\bigg\}.
\end{equation}
where $\rho_q=\frac{1}{\sqrt{q}} \frac{q^2+q+q^{3/2}}{q^2+q+1}$
\end{lemma}

An optimal bound, stronger than \eqref{eq:spectrum_of_A10}, for the spectrum of $A_{(1,0)}$ is proved in \cite{CartwrightMlotkowskiSteger} under the assumption that the links in $X$ are Desarguesian and that $\Gamma$ acts simply transitively on $X$. The inclusion of the spectrum of $A_{(1,0)}$ into 
\[
   \{1,e^{\frac{2i\pi}{3}}, e^{-\frac{2i\pi}{3}}\} \cup \bigg\{z \mathrel{\bigg\vert} -1 \leq \Re z\leq \rho_q\bigg\}
\]
follows from \cite{Zuk} (see also \cite{BdlHV}). But we need \eqref{eq:spectrum_of_A10} which is stronger.

\begin{proof}[Proof of Lemma~\ref{lem:spectrum}]
There are two ingredients in the proof. The first is that the spectrum of $A_{(1,0)}$ is invariant under a rotation of angle $\frac{2\pi}{3}$. Indeed, conjugating $A_{(1,0)}$ by the unitary operator on $\iE_2$ of pointwise multiplication by the function $x \mapsto e^{\frac{2\pi i}{3} \typ(x)}$ produces the operator $e^{\frac{2\pi i}{3}} A_{(1,0)}$.

The second ingredient is a refinement of \cite{Zuk}. Recall that the vertex set of the link $\lk x$ is $S_{(1,0)}(x) \cup S_{(0,1)}(x)$ and we briefly write $L^2(\lk x)$ for $L^2$-space on these vertices with respect to the uniform probability measure. For every $x \in \Omega$ let $M_x \colon L^2(\lk x) \to L^2(\lk x)$ denote the Markov operator on the link of $x$:
\[
M_x \xi(y) = \frac{1}{q+1}\sum_{z \sim y} \xi(z).
\] By a classical computation by Feit and Higman \cite[Proposition 5.7.6]{BdlHV}, $M_x$ is a self-adjoint operator with spectrum equal to $\{\pm 1,\pm \varepsilon_q\}$ for $\varepsilon_q=\frac{\sqrt{q}}{q+1}$, and the eigenspaces for the eigenvalues $1$ and $-1$ are the one dimensional spaces of constant functions and the functions proportional to $\chi_{S_{(1,0)}(x)} - \chi_{S_{(0,1)}(x)}$ respectively. In other words, if $P_x$ and $Q_x$ are the rank one orthogonal projections on these eigenspaces, then 
\begin{equation}\label{eq:feitHigman}
-\varepsilon_q (\mathrm{id} - P_x-Q_x)\leq M_x-P_x+Q_x\leq\varepsilon_q (\mathrm{id} - P_x-Q_x)
\end{equation}
(where $A \le B$ means that $B - A$ is a positive self-adjoint operator).
Define a linear map $U_x \colon \iE_2 \to L^2(\lk x)$ by $U_xf(y) =\frac{1}{|\Gamma_x|^{\frac 1 2}} f(y)$. Then the map
$U \colon \tilde{E}_2 \to \oplus_{x \in \Omega} L^2(\lk x)$ sending $f$ to $(U_xf)_{x \in \Omega}$ is an isometry: $U^* U = \mathrm{id}$. Moreover, using that $\lk(x)$ is the graph on $S_{(0,1)}(x)\cup S_{(1,0)}(x)$ small computations show that 
\begin{align*}
U^* (\oplus_{x \in \Omega} M_x )U & = \frac{1}{2}(A_{(1,0)}+A_{(0,1)}),\\ 
U^* (\oplus_{x \in \Omega} P_x )U & = \frac{1}{4}(A_{(1,0)}+A_{(0,1)})^2,\\ 
U^* (\oplus_{x \in \Omega} Q_x )U & = -\frac{1}{4}(A_{(1,0)}-A_{(0,1)})^2.
\end{align*}
So using that $A_{(1,0)}$ and $A_{(0,1)}$ commute we get
\begin{align*}
U^* (\oplus_{x \in \Omega} P_x+Q_x )U & = A_{(1,0)}A_{(0,1)}\\ U^* (\oplus_{x \in \Omega} P_x-Q_x )U & = \frac{1}{2}(A_{(1,0)}^2 + A_{(0,1)}^2). \end{align*}
By \eqref{eq:feitHigman}, we obtain \[ 
-\varepsilon_q (\mathrm{id} - A_{(1,0)}A_{(0,1)})\leq \frac{1}{2}(A_{(1,0)}+A_{(0,1)} - A_{(1,0)}^2-A_{(0,1)}^2)\leq\varepsilon_q (\mathrm{id} - A_{(1,0)}A_{(0,1)}).
\]
Using that $A_{(1,0)}$ is normal and $A_{(1,0)}^* = A_{(0,1)}$, we obtain that the spectrum of $A_{(1,0)}$ is contained in 
\[
\big\{z \in \C \mid |\Re (z-z^2)| \leq \varepsilon_q (1-|z|^2)\big\}\text{.}
\]

Putting together the two ingredients, we obtain that the spectrum of $A_{(1,0)}$ is contained in 
\[
  \Big\{z \in \C \mid \max_{k \in \Z/3\Z} |\Re(ze^{2ik\pi/3}-z^2e^{-2ik\pi/3})| \leq \varepsilon_q (1-|z|^2)\Big\}\text{.}
\]

Now we claim that $\max_{k \in \Z/3\Z} |\Re(ze^{2ik\pi/3} - z^2e^{-2ik\pi/3}))| \leq \varepsilon_q (1-|z|^2)$ implies that $z \in \{e^{2ik\pi/3} \mid k \in \Z/3\Z\}$ or $-\rho_q \leq \Re z \leq \rho_q$, so that \eqref{eq:spectrum_of_A10} follows.

Indeed,  first, taking $k=0$, the inequality $\Re(z-z^2) \leq \varepsilon_q (1-|z|^2)$ implies $\Re z-(\Re z)^2 \leq \varepsilon_q(1-(\Re z)^2)$, which implies that $\Re(z)=1$ or $\Re z \leq \frac{\varepsilon_q}{1-\varepsilon_q} = \rho_q$.

The inequality $\Re z \geq -\rho_q$ for every $z \neq e^{\pm \frac{2i\pi}{3}}$ is a bit more involved. More precisely, we shall prove that for every $z$ of modulus $<1$ and non-negative imaginary part, the inequalities $\Re(z^2-z) \leq \varepsilon_q(1-|z|^2)$ and $\Re(z e^{-\frac{2i\pi}{3}} - z^2 e^{\frac{2i\pi}{3}}) \leq \varepsilon_q(1-|z|^2)$ imply that $\Re(z) \geq -\rho_q$. The case of negative imaginary part follows by symmetry. Indeed, the first inequality defines the interior of a hyperbola that is symmetric with respect to the real axis, and which intersects the line $\Re z=-\rho_q$ at the point $-\rho_q+i \frac{\sqrt{2}\varepsilon_q}{(1-\varepsilon_q)^{\frac 3 2}}$ (and its complex conjugate). If $\varepsilon_q \geq \frac 1 3$, this point is outside the unit disc, which shows that for $z$ in the unit disc, the single inequality $\Re(z^2-z) \leq \varepsilon_q(1-|z|^2)$ implies $\Re(z) \geq -\rho_q$. 
If $\varepsilon<\frac 1 3$, this point lies inside the unit disc, but outside the area $\Re(z e^{-\frac{2i\pi}{3}} - z^2 e^{\frac{2i\pi}{3}}) \leq \varepsilon_q(1-|z|^2)$. So this shows that for any complex number with nonnegative real part, the inequalities $\Re(z^2-z) \leq \varepsilon_q(1-|z|^2)$ and $\Re(z e^{-\frac{2i\pi}{3}} - z^2 e^{\frac{2i\pi}{3}}) \leq \varepsilon_q(1-|z|^2)$ imply that $\Re(z) \geq -\rho_q$. This concludes the proof.
%
\end{proof}

\begin{proof}[Proof of Theorem~\ref{thm:propertyT}]
The eigenspaces of $A_{(1,0)}$ for the eigenvalues $ \{1,e^{\frac{2i\pi}{3}}, e^{-\frac{2i\pi}{3}}\}$ span the three-dimensional space $F$ of functions that are constant on each type. In particular, we have that, in restriction to the orthogonal of $F$, $(\Re A_{(1,0)})^n = 2^{-n}(A_{(1,0)}+A_{(0,1)})^n$ has norm $\leq \rho_q^n$, and that $\|(\Re A_{(1,0)})^{3n} - P_F\| \leq \rho_q^{3n}$, where $P_F$ is the orthogonal projection on $F$. As $\rho_q<1$ for $q>1$, this easily implies that $\Gamma$ has property (T), with the same argument as the one showing that Theorem~\ref{thm:strong_T_on_building} implies the main theorem. 
\end{proof}

\begin{proof}[Proof of Lemma~\ref{lem:exponential_integrability}]
Take for $(\pi,E)$ the trivial representation of $\Gamma$ on $\C$, so that $\iE_2$ becomes just $\ell^2(X/\Gamma,\mu)$ for the finite measure $\mu(\Gamma x) = \frac{1}{|\Gamma_x|}$. Let again $P_F$ be the orthogonal projection on the space $F$ of functions which are constant on each type. Denote by $Z_0 = \sum_{x \in \Omega, \typ(x)=\typ(x_0)} \frac{1}{|\Gamma_x|}$. For every $n$, denote by $f_n$ the indicator function of the $\Gamma$-orbit of $\{x \in \Gamma \mid d(x,x_0) \leq 3n, \typ(x)=\typ(x_0)\}$. Then we have 
\[ \langle \Re{A_{(1,0)}}^{3n} f_n,f_0\rangle = \frac{1}{\abs{\Gamma_{x_0}}} (\Re A_{(1,0)})^{3n} f_n)(x_0) = \frac{1}{\abs{\Gamma_{x_0}}},\]
whereas
\[ \langle P_F f_n,f_0\rangle = \frac{1}{Z_0\abs{\Gamma_{x_0}}} 
\langle f_n, 1\rangle.\]
So we deduce from the preceding that
\begin{align*} |\langle f_n,1\rangle - Z_0| & = Z_0\abs{\Gamma_{x_0}} \abs{\langle (\Re A_{(1,0)}^{3n} -P_F)f_n,f_0\rangle}\\
& \leq  Z_0 \abs{\Gamma_{x_0}} \cdot \rho_q^{3n} \cdot \|f_n\| \cdot \|f_0\|\\
& \leq \rho_q^{3n} Z_0 \sqrt{Z_0 |\Gamma_{x_0}|}.
\end{align*}
Or simply
\[ \sum_{x \in \Omega,\typ(x) = \typ(x_0),d(x,x_0)>3n} \frac{1}{|\Gamma_x|} \leq \rho_q^{3n} Z_0 \sqrt{Z_0 |\Gamma_{x_0}|}.\]
We deduce that
\[ \sum_{x \in \Omega,\typ(x) = \typ(x_0)} \frac{c^{d(x,x_0)}}{|\Gamma_x|}<\infty\]
for every $c<1/\rho_q$. Doing the same for the other types yields the lemma.
\end{proof}

It follows from Theorem \ref{thm:propertyT} that $\tilde{A}_2$-lattices are finitely generated. Let us equip $\Gamma$ with the word-length with respect to some finite generating set. We will say that $\Gamma$ is \emph{undistorted} if some orbit map is a quasi-isometric embedding:
\begin{equation}\label{eq:nondistorted}
\exists a,b>0,\forall \gamma \in \Gamma, |\gamma| \leq a d(x_0,\gamma x_0) + b.
\end{equation}
We fix such an $a$ and $b$ for the rest of the section. It follows that all orbit maps are quasi-isometries:
\[ \forall x \in X,\forall \gamma \in \Gamma, |\gamma| \leq a d(x,\gamma x) + (b+ 2 d(x_0,x)).\]
The celebrated Lubotzky-Mozes-Raghunathan theorem \cite{MR1828742} asserts that arithmetic lattices are undistorted. However, we do not know whether all $\tilde{A}_2$-lattices are undistorted.

We record for later use the following easy but important consequence of \eqref{eq:nondistorted}.
\begin{lemma}\label{lemma:undistorted}
If $\Gamma$ is undistorted, then for every $x \in \Omega$ and $\gamma \in \Gamma$, we have
\[ |\gamma| \leq 2a(d(x_0,x) + d(\gamma x,\Omega)) + b.\]
\end{lemma}
\begin{proof}
Let $x' \in \Omega$ such that $d(\gamma x,\Omega) = d(\gamma x,x')$. We can use the assumption \eqref{eq:nondistorted} that $\Gamma$ is undistorted and the triangle inequality to bound
\[    |\gamma| \leq a d(x_0,\gamma x_0) +b \leq a(d(x_0,x') + d(x',\gamma x_0)) + b.\]
By our definition \eqref{eq:Dirichlet_fundamental_domain} of $\Omega$, we have $d(x_0,x') \leq d(x_0,\gamma^{-1}x')=d(x',\gamma x_0)$, so we obtain $|\gamma| \leq 2a d(x',\gamma x_0) +b$, that we can bound using the triangle inequality by
\[
|\gamma| \leq 2a (d(x',\gamma x) + d(\gamma x,\gamma x_0)) + b \leq 2a (d(\gamma x,\Omega) + d(x_0,x))+b.\qedhere
\]

\end{proof}

\subsection{Strategy of the proof}\label{subsection:outline_nonuniform}
The goal of the remainder of the section is to prove the main theorem for undistorted $\tilde{A}_2$-lattices that may not be cocompact.
We now explain the proof which is obtained by combining the ingredients in the proof of the cocompact case with ideas in \cite{dlSActa19}. Let $\Gamma$ be an $\tilde{A}_2$-lattice. As in the cocompact case, restricting to a subgroup of index $\leq 6$, we can and will assume that $\Gamma$ acts on an $\tilde{A}_2$-building $X$ preserving the types.

Given a representation $\pi$ of $\Gamma$ on a Banach space $E$, let $C,\alpha$ be such that 
\[ \norm{\pi(\gamma)} \leq C e^{\alpha |\gamma|}\]
for every $\gamma \in \Gamma$. We define as in the cocompact case
\[
\iE = \{f \colon X \to E \mid \forall \gamma \in \Gamma, x \in X, f(\gamma x) =\pi(\gamma) f(x)\}.
\]
The space $\iE$ carries several natural pseudonorms: for every $r \in [1,\infty]$, we define
\begin{align*}
\norm{f}_r &= \left(\sum_{x \in \Omega} \frac{1}{\abs{\Gamma_x}} \norm{f(x)}^r\right)^{1/r}&\text{and}\\
\norm{f}_\infty &= \sup_{x \in \Omega} \norm{f(x)}
\end{align*}
and denote by $\iE_r$ the subspace of $\iE$  consisting of elements with $\norm{f}_r<\infty$. By Hölder's inequality and the assumption that $\sum_{x \in \Omega} \frac{1}{\abs{\Gamma_x}}<\infty$, we have $\iE_{r_1} \subset \iE_{r_2}$ whenever $r_1 \geq r_2$. 
Of course, when $\Omega$ is finite, $\iE = \iE_r$ for every $r$ and the spaces $\iE_r$ neither depend on $r$ nor on $\Omega$, but otherwise the space $\iE_r$ are all distinct and depend on $\Omega$. As in the cocompact case, $A_\lambda$ maps $\iE$ to itself because the action of $\Gamma$ on $\partial X$ is type-preserving. However, there is no reason to expect that $A_\lambda$ maps $\iE_r$ to itself in general unless  when $\Omega$ is finite or $\pi$ is uniformly bounded.

The basic idea in the proof is to follow the same strategy as for cocompact lattices, but with varying values of $r$ in order to absorb all the small exponentials that arise from the infiniteness of $\Omega$ and the unboundedness of $\norm{\pi(\gamma)}$. This phenomenon which, in presence of a group, was called $2$-step representations in \cite{dlSActa19}, is captured by the following lemma. Recall that we have defined $|\lambda|_1=s+t$ if $\lambda=(s,t)$.
\begin{lemma}\label{lem:2steprep}
Let $r_1 \geq r_2$. Assume that $\alpha \leq \frac{\varepsilon}{2a}(\frac{1}{r_2} -\frac{1}{r_1})$. Then there is an $M \in \R_+$ such that for every $\lambda \in \Lambda$ and $f\in \iE$: 
\[
\norm{A_\lambda f}_{r_2} \le \left(  \sum_{x \in \Omega}\frac{1}{\abs{\Gamma_x}} \frac{1}{|S_\lambda(x)|} \sum_{y \in S_\lambda(x)} \|f(y)\|^{r_2}\right)^{\frac{1}{r_2}} \leq C M e^{2 a \alpha |\lambda|_1} \norm{f}_{r_1}.\label{eq:r2_r1_estimate}
\]
In particular, the operators $A_\lambda$ have norm $\leq C M e^{2 a \alpha |\lambda|_1}$ from $\iE_{r_1}$ to $\iE_{r_2}$.
\end{lemma}
\begin{proof}
For every $y \in X$, there is a unique $x' \in \Omega$ in the $\Gamma$-orbit of $y$. For such $x'$, the number of $\gamma$ such that $\gamma x'=y$ is equal to $\abs{\Gamma_{x'}}$. So we can write
\[ \|f(y)\|^{r_2} = \sum_{x' \in \Omega} \frac{1}{\abs{\Gamma_{x'}}} \sum_{\gamma \in \Gamma, y=\gamma x'} \|\pi(\gamma) f(x')\|^{r_2}.\]
By Lemma \ref{lemma:undistorted}, if $y =\gamma x' \in S_\lambda(x)$ for some $x \in \Omega$, we have $|\gamma| \leq 2a (|\lambda|_1 + d(x_0,x'))+b$.
Therefore we have
\[ \|\pi(\gamma) f(x')\| \leq C e^{\alpha |\gamma|} \|f(x')\| \leq C e^{\alpha b} e^{2a\alpha(|\lambda|_1 + d(o,x'))}\|f(x')\|.\] 
Summing over all $y$ and $x$, we obtain
\begin{multline*}
    \left(\sum_{x \in \Omega}\frac{1}{\abs{\Gamma_x}} \frac{1}{|S_\lambda(x)|} \sum_{y \in S_\lambda(x)} \|f(y)\|^{r_2} \right)^{\frac{1}{r_2}}\\
   \leq C e^{\alpha b+2a \alpha |\lambda|_1} \left(\sum_{x' \in \Omega} \frac{1}{|\Gamma_{x'}|} e^{2a\alpha r_2 d(o,x')} \|f(x')\|^{r_2} \sum_{x \in \Omega} \frac{1}{|\Gamma_x|} \frac{1}{|S_\lambda(x)|} \sum_{\gamma \in \Gamma, \sigma(x,\gamma x') =\lambda} 1 \right)^{\frac{1}{r_2}}
\end{multline*}
The quantity $\frac{1}{|\Gamma_x|} \sum_{\gamma \in \Gamma, \sigma(x,\gamma x') =\lambda} 1$ is the cardinality of the intersection of the $\Gamma$-orbit of $x$ with $S_{\overline{\lambda}}(x')$. By chamber regularity, $|S_\lambda(x)|=|S_{\overline{\lambda}}(x')|$ for every $x,x'$ so we obtain
\begin{multline*}
    \left(\sum_{x \in \Omega}\frac{1}{\abs{\Gamma_x}} \frac{1}{|S_\lambda(x)|} \sum_{y \in S_\lambda(x)} \|f(y)\|^{r_2} \right)^{\frac{1}{r_2}}\\
   \leq C e^{\alpha b+2a \alpha |\lambda|_1} \left(\sum_{x' \in \Omega} \frac{1}{|\Gamma_{x'}|} e^{2a\alpha r_2 d(o,x')} \|f(x')\|^{r_2}  \right)^{\frac{1}{r_2}}.
\end{multline*}
The second inequality now follows using Hölder's inequality and \eqref{eq:exponential_integrability}.

For the first inequality we first apply the triangle inequality and then Hölder's inequality to see that
\[
\|A_\lambda f(x)\|^{r_2} \leq \frac{1}{|S_\lambda(x)|} \sum_{y \in S_\lambda(x)} \|f(y)\|^{r_2}\text{.}\qedhere
\]
\end{proof}

The following proposition is the counterpart of  Theorem~\ref{thm:strong_T_on_building} in the non-cocompact case. It is a combination of Propositions~\ref{prop:harmonic_limit_nonuniform} and~\ref{prop:harmonic_constant_nonuniform} below (the analogs of Propositions~\ref{prop:harmonic_limit} and~\ref{prop:harmonic=constants}).

\begin{proposition}\label{prop:strong_T_on_buildingunbdd}
Let $\Gamma$ be an undistorted lattice \eqref{eq:nondistorted}. Let $1<r_2<r_1<\infty$. If 
\[
\alpha \leq \min\left(\frac{\varepsilon}{8 a},\frac{\theta}{6a} \right) \min\left(\frac{1}{r_2} - \frac{1}{r_1},1 - \frac{1}{r_2}\right)
\]
then, as $\lambda \in \Lambda_0$ tends to infinity, the sequence $(A_\lambda)_{\lambda}$ converges in the norm of $B(\iE_{r_1},\iE_{r_2})$ towards a map $P \in B(\iE_{r_1}, \iE_{r_2})$ whose image is made of functions which are constant on the vertices of each type.
\end{proposition}

As in the cocompact case, this proposition is proved in two distinct steps. We first show in Proposition~\ref{prop:harmonic_limit_nonuniform} that the operators $A_\lambda$ are Cauchy, and therefore converge to an operator whose image is made of $\Lambda_0$-harmonic functions. And then in Proposition~\ref{prop:harmonic_constant_nonuniform}, we show that $\Lambda_0$-harmonic functions in $\iE_r$ are constant on each type if $r>1$ and $\alpha$ is small enough.

The main~theorem follows from the  proposition with almost the same proof as in the cocompact case. As before we define $u(\xi)(y)=\pi(m_y)(\xi)$ for $y \in X$, $v(f) = \frac{1}{Z} \sum_{x \in \Omega} \frac{1}{\abs{\Gamma_x}}f(x)$ for $f \in \iE$  and a probability measure $\mu_\lambda$ on $\Gamma$ by~\eqref{eq:def_mulambda} so that as in \eqref{eq:mu_lambda}
\begin{equation}\label{eq:mu_lambda_unbdd}
v \circ A_\lambda \circ u = \pi(\mu_\lambda)\text{.}
\end{equation}

The first modification is that we regard $v$ as an operator $\iE_1 \to E$, so that it is clearly bounded. The second is that we need to verify that $u \colon E \to \iE_r$ is bounded provided $\alpha$ is small enough:

\begin{lemma}
The map $u \colon E \to \iE_r$ defined by $u(\xi)(y) = \pi(m_y) \xi$ is bounded if $\alpha \leq \frac{\varepsilon}{2ra}$.
\end{lemma}

\begin{proof}
For every $y \in \Omega$, we can bound
\[ \norm{u(\xi)(x)} \leq \max_{\gamma \in \Gamma_x} \norm{\pi(\gamma)} \norm{\xi}.\]
Using Lemma~\ref{lemma:undistorted}, we deduce that
\[ \norm{u(\xi)(x)} \leq  C e^{\alpha b} e^{2a\alpha d(x,o)} \norm{\xi}.\]
By Lemma \ref{lem:exponential_integrability} to obtain that $u(\xi) \in \iE_r$ if $2a\alpha r \leq \varepsilon$.
\end{proof}

\begin{proof}[Proof of the main theorem for general $\Gamma$]
We shall prove that the hypothesis of Lemma~\ref{lem:truncation} holds with the measures $\mu_\lambda$ if $\alpha \leq \min(\frac{\varepsilon}{18a},\frac{2\theta}{27a})$.

Our hypothesis on $\alpha$ is exactly the one in Proposition~\ref{prop:strong_T_on_buildingunbdd} for $r_1=9$ and $r_2 =\frac 95$. So let $P$ be given by this proposition for this choice of $r_1,r_2$. As in the cocompact case one verifies that the operator $v \circ P \circ u \colon E  \to E$ is a projection onto the space $E^\Gamma$ of $\pi(\Gamma)$-invariant vectors.

Now Lemma~\ref{lem:2steprep} implies that $v \circ A_\lambda \circ u$ is well-defined for all $\lambda$ and by Proposition~\ref{prop:strong_T_on_buildingunbdd} these operators converge in norm to $v \circ P \circ u$ as $\lambda \in \Lambda_0$ tends to infinity:
\[
\norm{v \circ (A_\lambda - P) \circ u}_{B(E)} \leq \norm{v}_{B(\iE_{r_2},E)} \norm{A_\lambda -P}_{B(\iE_{r_1},\iE_{r_2})} \norm{u}_{B(E,\iE_{r_1})} \to 0 
\]
In view of \eqref{eq:mu_lambda_unbdd} this completes the proof. 
\end{proof}


The remainder of the section is devoted to the proof of Proposition \ref{prop:strong_T_on_buildingunbdd}.

\subsection{Induction and duals}
Replacing $E$ by $E^*$ and $\pi$ by the dual representation $\gamma \mapsto \pi(\gamma^{-1})^*$ in the definitions of Section~\ref{subsection:outline_nonuniform}, we can define in the same way the space $\diE$ and $\diE[r]$.

Throughout this section, we denote by $r'$ the dual exponent of $r$: $\frac 1 r + \frac 1 {r'}=1$.
For every $r$ we define a duality pairing by 
\begin{equation}\label{eq:duality_pairingv2}
\gen{g,f}= \sum_{x \in \Omega} \frac{1}{\abs{\Gamma_x}} \gen{g(x),f(x)}\text{.}
\end{equation}
whenever $f \in \iE$, $g \in \diE$ and the series converges.

As in the cocompact case, $\iE_r$ is isometric to the $\ell^r$-direct sum $\bigoplus_{x \in \Omega} E^{\Gamma_x}$ via the isomorphism
\begin{align}
\bigoplus_{x \in \Omega} E^{\Gamma_x} &\to \iE\label{eq:etilde_decompositionv2}\\
(\xi_x)_{x \in \Omega} & \mapsto (\gamma.x \mapsto |\Gamma_x|^{\frac 1 r}\pi(\gamma)(\xi))\text{.}\nonumber
\end{align}
And similarly, $\diE[r']$ is isomorphic to the $\ell^{r'}$ direct sum
$\bigoplus_{x \in \Omega} (E^*)^{\Gamma_x}$ . However, this does not yield to an isomorphism between $\iE_r$ and $\diE[r']$ as the spaces $(E^*)^{\Gamma_x}$ and $(E^{\Gamma_x})^*$ are only isomorphic up to a constant $\leq \| \frac{1}{|\Gamma_x|} \sum_{\gamma \in \Gamma_x} \pi(\gamma)\|$, which is a priori unbounded as $x$ varies in $\Omega$. However, the duality pairing \eqref{eq:duality_pairingv2} provides continuous injective maps $\diE[r'] \to  (\iE_r)^* \to \diE[r_2]$ for some $r_2 < r'$, that can be taken arbitrarily close to $r'$ by taking $\alpha$ small enough. This is the content of the next lemma.
\begin{lemma}\label{lem:duality_nonuniform}
\begin{enumerate}
\item For every $g \in \diE[r']$, the map $f \mapsto \langle g,f\rangle$ is a continuous linear form on $\iE_r$ of norm $\leq \|g\|$.
\item Let $r_2 <r'$. If $\alpha \leq \frac{\varepsilon}{2a}(\frac{1}{r_2} - \frac 1 {r'})$, then there is a constant $C'$ such that for every continuous linear form $\varphi \colon \iE_r \to \C$, there is a unique $g \in \diE[r_2]$ such that $\varphi(f) = \langle g,f\rangle$ for every $f \in \iE$. It has norm $\norm{g}_{r_2} \leq C' \norm{\varphi}$.
\end{enumerate}
\end{lemma}
\begin{proof} 
This is essentially the same proof as Lemma~\ref{lemma:dual}. The first point is a direct consequence of Hölder's inequality.

For the second point, let $\varphi \in \iE_r \to \C$ be a continuous linear form. Define $f_{x,\xi} \in \iE_r$ and $g \in \diE$ as in the proof of Lemma~\ref{lemma:dual}. For every $\xi \in E$, using Lemma~\ref{lemma:undistorted} we can bound, for $x \in \Omega$,
\begin{equation}\label{eq:norm_fxxi_v2}
    \norm{f_{x,\xi}(x)} \leq \max_{\gamma \in \Gamma_x} \|\pi(\gamma)\| \|\xi\| \leq C e^{b\alpha + 2a\alpha d(x_0,x)} \|\xi\|.
\end{equation} 
Let us pick, for every $x \in \Omega$, a vector $\xi_x \in E$ such that  $\langle g(x), \xi_x\rangle = \|g(x)\|^{r_2}$ and $\|\xi_x\|\leq (1+\delta)\|g(x)\|^{r_2-1}$ for some $\delta>0$. For every finite subset $A \subset \Omega$, let $g_A$ denote the restriction of $g$ to the $\Gamma$-orbit of $A$. We can bound the norm of $g_A$ in $\diE[r_2]$ as follows.
\begin{multline}
   \norm{g_A}_{r_2}^{r_2}  = \sum_{x \in A} \frac{1}{\abs{\Gamma_x}} \langle g(x), \xi_x\rangle
    = \varphi(\sum_{x \in A} f_{x,\xi_x})\\
    \leq (1+\delta) C e^{b\alpha} \|\varphi\| \left(\sum_{x \in A}  \frac{1}{\abs{\Gamma_x}} e^{2ar\alpha d(x_0,x)} \|g(x)\|^{r(r_2-1)}   \right)^{\frac 1 r}.
\end{multline}
Using Hölder's inequality, writing $\frac{1}{q} = \frac{1}{r_2}-\frac{1}{r'}$, we can bound
\begin{equation*}
       \left(\sum_{x \in A}  \frac{1}{\abs{\Gamma_x}} e^{2ar\alpha d(o,x)} \|g(x)\|^{r(r_2-1)}   \right)^{\frac 1 r} 
   \\\leq \norm{g_A}_{r_2}^{r_2-1} \left(\sum_{x \in \Omega}  \frac{1}{\abs{\Gamma_x}} e^{2aq\alpha d(o,x)} \right)^{\frac 1 q}.
\end{equation*}
By Lemma~\ref{lem:exponential_integrability} this last quantity is finite because $2 a \alpha q \leq \varepsilon$. If we make $\delta\to 0$ and we denote by $C'$ its product with $C e^{b\alpha}$, we obtain
\[ \norm{g_A}_{r_2}^{r_2} \leq C' \norm{\varphi} \norm{g_A}_{r_2}^{r_2-1},\]
that we simplify to $\norm{g_A}_{r_2} \leq C' \norm{\varphi}$. Since $A$ is arbitrary, we obtain that $g$ belongs to $\diE[r_2]$ with norm $\leq C' \norm{\varphi}$. Finally, the fact that $\langle g,f\rangle$ converges for every $f \in \iE_r$ and is equal to $\varphi(f)$ follows as in Lemma \ref{lemma:dual} by decomposing $f = \sum_{x \in \Omega} f_{x,f(x)}$ (convergence in $\iE_r$). The uniqueness of $g$ is clear with the same argument.
\end{proof}

\subsection{Convergence to projection on harmonic functions}\label{subsection:convergence_nonuniform}
We now generalize the content of Section~\ref{sec:convergence_harmonic} to the setting of undistorted $\tilde{A}_2$-lattices.

Recall $\pi$ is a representation of $\Gamma$ on $E$, and $C,\alpha \in [0,\infty)$ are such that
\[
\norm{\pi(\gamma)} \leq C e^{\alpha |\gamma |}\text{.}
\]
\begin{proposition}\label{prop:harmonic_limit_nonuniform} Let $1 \leq r_4 < r_1\leq \infty$. If 
\[\alpha < \min\left(\frac{\varepsilon}{8 a},\frac{\theta}{6a} \right)\left(\frac{1}{r_4} - \frac{1}{r_1}\right),\]
then, as $\lambda \in \Lambda_0$ tends to infinity, $A_\lambda$ converges in $B(\iE_{r_1},\iE_{r_4})$ towards a map $P \in B(\iE_{r_1},\iE_{r_4})$ whose image is made of $\Lambda_0$-harmonic elements.
\end{proposition}
\begin{proof}
With the results proven so far in this section, the proof is essentially the same as Proposition~\ref{prop:harmonic_limit}. 

Define $r_4< r_3<r_2<r_1$ by $\frac{1}{r_2}= \frac{1}{2r_4}+\frac{1}{2r_1}$ and $\frac{1}{r_3} = \frac{3}{4r_4}+ \frac{1}{4r_1}$, so that 
$\alpha \leq \frac{\varepsilon}{2a}(\frac{1}{r_{i+1}} - \frac{1}{r_i})$ for every $i=1,2,3$.

Let $f \in \iE_{r_1}$ and $\varphi \in (\iE_{r_4})^*$ of norm $1$. By Lemma~\ref{lem:duality_nonuniform}, $\varphi$ can be represented by an element of $\diE[r_3']$ of norm $\leq C'$. Let $i,j,s$ such that $i\geq j > s$. Write $\lambda = (s,i+j-2s)$. In a similar way as in Proposition~\ref{prop:harmonic_limit} we define, for every $x \in \Omega$, linear maps $I_{x,i}:\diE \to \LL^0(\P,E^*)$ and $J_{x,j}:\iE\to \LL^0(\L,E)$. Lemma~\ref{lem:2steprep} with $\lambda =(0,j)$ and $(i,0)$ respectively implies that $I_{x,i}$ and $J_{x,j}$ map $\diE[r_3']$ to $\LL^{r_2'}(\P,E^*)$ and $\iE_{r_1}$ to $\LL^{r_2}(\L,E)$ respectively, and moreover we have
\[ \left(\sum_{x \in \Omega} \frac{1}{\abs{\Gamma_x}} \|I_{x,i}(g)\|_{\LL^{r_2}(\P;E^*)}^{r'_2}\right)^{\frac{1}{r_2}} \leq C'' e^{2a\alpha i}\]
and
\[ \left(\sum_{x \in \Omega} \frac{1}{\abs{\Gamma_x}} \|J_{x,j}(f)\|_{\LL^{r_2}(\L;E)}^{r_2}\right)^{\frac{1}{r_2}} \leq C'' e^{2a\alpha j}\]
for some $C''$.
Applying Proposition~\ref{prop:double_counting} to the function $F(x,y) = \langle g(x),f(y)\rangle$ we get
\[ \varphi(A_\lambda f) = \sum_{x \in \Omega} \frac{1}{\abs{\Gamma_x}} \langle I_{x,i}(g),T_{x,s} J_{x,j}(f)\rangle.\]

Finally, by the hypothesis~\eqref{eq:E-valued_Hoelder_bound}, the operator
$T_{x,s} -T_{x,s+1}$ has norm $\leq L e^{-\theta s}$ as an operator from  $\LL^2(\P;E)$ to $\LL^2(\L;E)$, and has clearly norm $\leq 2$ as an operator from $\LL^1(\P;E)$ to $\LL^1(\L;E)$ and from $\LL^\infty(\P;E)$ to $\LL^\infty(\L;E)$. Moreover, by interpolation (see Subsection~\ref{subsection:interpolation}) its norm from $\LL^r(\P;E)$ to $\LL^r(\L;E)$ is a log-convex function of $\frac 1 r$. So we deduce that it has norm $\leq 2 L e^{-\theta (1-|1-\frac {2}{r_2}|)s}$ from $\LL^{r_2}(\P;E)$ to $\LL^{r_2}(\L;E)$. Putting everything together, we obtain
\[\varphi( A_\lambda f - A_{\lambda + (1,-2)}f) \leq 2L C'' e^{2a\alpha(i+j)} e^{-\theta(1-|1-\frac{2}{r_2}|)s}.\]
Taking the supremum over $f$ and $\varphi$, we obtain
\[
\norm{A_\lambda f - A_{\lambda + (1,-2)}f}_{B(\iE_{r_1},\iE_{r_4})} \le 2L C'' e^{-s D + (i+j-2s) E}\]
where $D = \theta(1-|1-\frac{2}{r_2}|) - 4a\alpha$ and $E=2 a \alpha$. Note that $E<D$ by our assumption on $\alpha$.

Exchanging the roles of points and lines, we obtain the same upper bound for the other singular direction. We conclude by Lemma~\ref{lem:lafforgue} that $A_\lambda$ converges as $\lambda \in \Lambda_0$.

That the image of the limit consists of $\Lambda_0$-harmonic elements can be seen as before but a little care is needed formally. First, note that the operators $A_\lambda \colon \tilde{E} \to \tilde{E}$ are well-defined as abstract linear operators, without norms or boundedness involved and the definition of harmonicity makes sense as such. We need to show that
\[
A_\lambda \lim_{\lambda'} A_{\lambda'} = \lim_{\lambda'} A_\lambda A_{\lambda'} = \lim_{\lambda''} A_{\lambda''}
\]
where $A_\lambda \colon \iE_{r_4} \to \iE_{r_4}$ is unbounded and $A_{\lambda'}, A_{\lambda''} \colon \iE_{r_1} \to \iE_{r_4}$ are bounded. The second equality follows as before from the decomposition $A_\lambda A_{\lambda'} = \sum_{\lambda''} p_{\lambda,\lambda'}(\lambda'') A_{\lambda''}$ in \cite[Proposition 2.3]{CartwrightMlotkowski94}. The first is true by linearity because for $f \in \iE$ and $x \in X$ the value $A_\lambda f(x)$ is a finite convex combination.
\end{proof}

\subsection{Harmonic functions are constant} The following proposition generalizes the content of Section~\ref{sec:harmonic_constant}.
\begin{proposition}\label{prop:harmonic_constant_nonuniform}
 If $r>1$ and
 \[\alpha \leq \min\left(\frac{\varepsilon}{8a},\frac{\theta}{6a}\right) \left(1-\frac 1 r\right),\]
then every $\Lambda_0$-harmonic element of $\iE_r$ is constant on each type.
\end{proposition}
\begin{proof}
As in the cocompact case, we can define vector spaces $\iE^{(1)}$ and Banach subspaces $\iE^{(1)}_r$  for the norm
\[
\norm{f} \defeq \bigg(\sum_{(x,x') \in \Omega^{(1)}} \frac{1}{\abs{\Gamma_{x,x'}}} \norm{f(x,x')}^r\bigg)^{\frac{1}{r}}.\]
Here $\Omega^{(1)}$ is a fundamental domain for $\Gamma \backslash X^{(1)}$, that we assume to be contained in $\Omega \times X$.

Again, some care is needed as $\iE^{(1)}_r$ depend on $r$.

As in the proof of Proposition \ref{prop:harmonic_limit_nonuniform}, the assumptions imply that, for every $\truc  \in\{+,\sim,-\}$, the operators $A_\lambda^\truc$ converge in $B(\iE_r,\iE_1^{(1)})$ to some operator $P^\truc$ as $\lambda \in \Lambda_0$. Morever $P^+ = P^\sim=P^-$. We can exploit the same argument as in the cocompact case to deduce that harmonic elements $\iE_r$ are constant.
\end{proof}

\section{Vanishing of $\ell^p$-cohomology}\label{sec:lp_cohomology}
The $\ell^r$-cohomology\footnote{Recall that we avoid the usual term $\ell^p$-cohomology because the letter $p$ is extensively used to denote points in the building at infinity} is defined for every locally finite simplicial complex, see \cite{MR3356973}. In the particular case when $X$ is connected and simply connected (for example an $\tilde A_2$-building) and $k=1$, it can be defined as follows.
Let $V$ and $E$ be the vertices and oriented edges of $X$, respectively, and let $\ell^0(V)$ and $\ell^0(E)$ denote the spaces of functions $V \to \C$ and $E \to \C$ (where orientation change translates into a change of sign).
Define the coboundary map $\delta \colon \ell^0(V) \to \ell^0(E)$ by
\[
\delta f(e) = f(e^+) - f(e^-)\text{.}
\]
Then the first $\ell^r$-cohomology of $X$ is
\[
\ell^r H^1(X) = \delta(\ell^0(V)) \cap \ell^r(E) / \delta(\ell^r(V)).
\]
We shall use the following general lemma, which identifies the $\ell^r$-cohomology with the harmonic cocycles. Recall that the Laplacian of $f \colon V \to \C$ is defined by
\[
\Delta f(x) = \frac{1}{d(x)} \sum_{e \in E, e^- = x} f(x) - f(e^+) .
\]

The following lemma is certainly well-known to experts (see \cite{MR1078111} for manifolds, \cite{MR3497258} for groups, \cite[Lemma~1.18]{MR1926649} for $r = 2$), but we could not locate a proof, so we include one for reference.

\begin{lemma}\label{lemma:harmonic-representatives} If $X$ is a simply connected simplicial complex, and if the underlying graph has bounded degree and positive Cheeger constant, then for every $1<r<\infty$ every class $\ell^r H^1(X)$ is represented by a harmonic function $f\colon V \to \C$ satisfying $\delta f \in \ell^r(E)$, and such a representation is unique up to an additive constant. Consequently,
 \[
 \ell^r H^1(X) \cong \{f \colon V \to \C \mid \delta f \in \ell^r(E), \Delta f=0\} / \C.
 \]
\end{lemma}
\begin{proof} 
By Cheeger's inequality and the bounded degree assumption, the positive Cheeger constant assumption is equivalent to saying that $\Delta$ is invertible on $\ell^2(V)$. By interpolation arguments, this implies that $\Delta$ is invertible on $\ell^r(V)$. Indeed, the invertibility of $\Delta$ on $\ell^2(V)$ is equivalent to $\|1 - \frac{\Delta}{2}\|_{B(\ell^2(V))}<1$. Note also that $\|1 - \frac{\Delta}{2}\|_{B(\ell^r(V))}\leq 1$ for $r=1,\infty$. Thus, the Riesz--Thorin theorem implies that $\|1 - \frac{\Delta}{2}\|_{B(\ell^r(V))}<1$ for every $1<r<\infty$. This in turn implies that $\Delta/2 = 1-(1-\Delta/2)$ is invertible.

Consider now the boundary map $\partial \colon \ell^r(E) \to \ell^r(V)$ mapping $f$ to
\[
\partial f(x) = \frac{1}{d(x)} \sum_{e \in E, e^+=x} f(e)\text{.}
\]
It is a bounded operator satisfying $\partial \circ \delta = \Delta$. By the invertibility of $\Delta$ on $\ell^r(V)$, we deduce that we can decompose $\ell^r(E)$ as $\delta(\ell^r(V)) \oplus \ker \partial$. So we obtain that $\ell^r H^1(X) = (\delta(\ell^0(V)) \cap \ell^r(E)) \cap \ker \partial$, from which the lemma follows.
 \end{proof} 
 
For the rest of the section let $X$ be a locally finite $\tilde{A}_2$-building. The spectrum of $\Delta$ on $\ell^2(V)$ is known to be equal to the interval $[1-\frac{3q}{q^2+q+1},1+\frac{3q}{2(q^2+q+1)}]$ (it is equal to $\{1-\Re(z) \mid z \in \Sigma\}$ for $z$ in the set $\Sigma$ in \cite[Proposition 4.4]{CartwrightMlotkowski94}). In particular, the graph $(V,E)$ has positive Cheeger constant. Therefore, Theorem~\ref{thm:lp_cohomology} is equivalent to the following:

\begin{theorem}\label{thm:lp_cohomology_concrete}
Every harmonic function $f$ on $X$ such that $\delta f \in \ell^r$ for some $r<\infty$ is constant.
\end{theorem}

As before, we identify $X$ with its $0$-skeleton $V$. We equip $\ell^0(X) = \ell^0(V)$, the space of all functions $X \to \C$, with the $\ell^r$ pseudo-distance. This means that in general, two elements of $\ell^0(X)$ are at infinite distance apart, unless their difference lies in $\ell^r(X)$. 

The main technical ingredient in the proof of Theorem~\ref{thm:lp_cohomology_concrete} is Lemma~\ref{lem:cv_of_Alambdaf_ellr} and deals with the convergence of the averaging operators defined in \eqref{eq:def_Alambda} applied to functions with finite $r$-energy. As in the proof of strong property (T), we denote by $X^{(1)}$ the set of pairs $(x,x') \in X\times X$ satisfying $\sigma(x,x') = (1,0)$. 

\begin{lemma}\label{lem:cv_of_Alambdaf_ellr} Let $f \in \ell^0(X)$ be a function such that $\delta f \in \ell^r(E)$. For every $k \in \Z/3\Z$, and $x \in X$, $\lim_{\lambda \in \Lambda_k, |\lambda|\to \infty} A_\lambda f(x)$ exists. Moreover, for every $(x,x') \in X^{(1)}$ and every $k \in \Z/3\Z$ we have
\[
 \lim_{\lambda \in \Lambda_{k}, |\lambda|\to \infty} A_\lambda f(x) = \lim_{\lambda \in \Lambda_{k+2}, |\lambda|\to \infty} A_\lambda f(x').
\]
\end{lemma}



Before we prove this lemma, let us explain how it is used when $f$ is moreover harmonic.
\begin{proof}[Proof of Theorem~\ref{thm:lp_cohomology_concrete}]
Let $f \in \ell^0(X)$ be a harmonic function such that $\delta f \in \ell^r(E)$. 

We use the probabilistic/potential theoretic formulation of harmonicity. For $x \in X$, denote by $\proba_x$ the law of the random walk $(x_n)_{n\geq 0}$ starting at $x$ with transition probabilities
\[ p(y,z) = \begin{cases} \frac 1 3 & \textrm{if }y=z\\
 \frac{2}{3(q^2+q+1)} & \textrm{if }d(y,z)=1.\end{cases}\]
This means that its Markov operator is $\frac{1}{3}(\operatorname{id}+A_{(1,0)} + A_{(0,1)}) = \operatorname{id} + \frac{2}3 \Delta$, so we have $f(x) = \mathbb E_x(f(x_n))$ for every $n$. We denote $p_{n,x}(\lambda) = \mathbb P_x( \sigma(x,x_n)=\lambda)$ (in our case it does not depend on $x$ but we will not use that fact). 
Harmonicity of $f$ then gives 
\[ f(x) = \mathbb E_x (f(x_n)) = \sum_\lambda p_{n,x}(\lambda) A_\lambda f(x).\]
By the transience of the random walk in $X$, we have that $\lim_n p_{n,x}(\lambda)=0$ for every $\lambda$. Moreover, under $\proba_x$, the type of $x_n$ is uniformly distributed in $\Z/3\Z$ for every $n\geq 1$ (this was the reason for taking a lazy random walk rather than the usual random walk $p(y,z) = \frac{1}{2(q^2+q+1)}$ if $d(y,z)=1$), so taking the limit $n \to \infty$ and using Lemma~\ref{lem:cv_of_Alambdaf_ellr} we obtain 
\begin{equation}\label{eq:fxequalfinfty} f(x) = \frac{1}{3} \sum_{k \in \Z/ 3\Z} \lim_{\lambda \in \Lambda_k, |\lambda|\to \infty} A_\lambda f(x).\end{equation}
This is valid for every $x \in X$. Taking $(x,x') \in X^{(1)}$, we see from Lemma~\ref{lem:cv_of_Alambdaf_ellr} that the right-hand side of \eqref{eq:fxequalfinfty} remains unchanged if we replace $x$ by $x'$. As a consequence, $f(x)=f(x')$ for every $(x,x') \in X^{(1)}$. This proves that $f$ is constant and concludes the proof, as the graph obtained by joining two vertices $x$ and $x'$ if $\sigma(x,x') =(1,0)$ is connected.\end{proof}

It remains to prove Lemma~\ref{lem:cv_of_Alambdaf_ellr}.
In the notation of Lemma~\ref{lem:relation_Alambda_Alambdatruc}, for $f \in \ell^0(X)$, and for every $\truc \in \{+,-,\sim\}$ we define $f_{\lambda}^\truc = A_\lambda^\truc f$ that is,
 \[f_{\lambda}^\truc(x,x') = \frac{1}{|S^\truc_\lambda(x,x')|} \sum_{y \in S^\truc_\lambda(x,x')} f(y).\]
 We also equip $\ell^0(X^{(1)})$ with the $\ell^r$ pseudo-distance. 
 \begin{lemma}\label{lem:10.3}
 Let $\theta = \min(\frac 1 r,1 - \frac 1 r)>0$. For every $f \in \ell^0(X)$ such that $\delta f \in \ell^r(E)$, 
 \begin{itemize}
     \item 
 For every $s\geq 1$, $t \geq 2$ and $\truc \in \{+,\sim,-\}$, \begin{equation}\label{eq:ftruc_in_ellr} \|f_{s,t}^\truc - f_{s+1,t-2}^\truc\|_r \leq 4 (s+1) (q^2+q+1)q^{-\theta s} \|\delta f\|_r.
\end{equation}
\item For every $t\geq 1$, $s \geq 2$ and $\truc \in \{+,\sim,-\}$.
\begin{equation}\label{ftruc_dual_direction} \| f_{s,t}^\truc -  f_{s-2,t+1}^\truc\|_r \leq 4 (s+1) (q^2+q+1) t q^{-\theta t} \|\delta f\|_r.
\end{equation}
\end{itemize}
\end{lemma}
\begin{proof} 
We can normalize $f$ so that $\|\delta f\|_r =1$. 
Let ${r'}$ be the dual exponent of $r$: $\frac 1 r + \frac 1 {r'}=1$. Fix $g \in \ell^{r'}(X^{(1)})$ of norm $1$ and finitely supported. For $o \in X$ and integers $i\geq j>s>0$, denote $\alpha_{o,i}^+(p) = g(p_{o,i},p_{o,i+1})$ and $\beta_{o,j}(\ell) = f(\ell_{o,j})-f(o)$. Then \eqref{eq:double_counting_+} in Proposition \ref{prop:double_counting_diamond} for $\Gamma$ the trivial group gives
 \begin{equation}\label{eq:cohom_average}
 \sum_{o \in X} \langle \alpha_{o,i}^+, (T_{o,s} - T_{o,s+1})(\beta_{o,j})\rangle = \frac{1}{q^2+q+1}\langle g, f_{s,i+j-2s}^+ - f^+_{s+1,i+j-2s-2}\rangle.
\end{equation}
The operator $T_{o,s} - T_{o,s+1}$ has norm $\leq 4 q^{-\frac s 2}$ from $\LL^2(\L,\proba_o)$ to $\LL^2(\P,\proba_o)$ by Proposition~\ref{prop:Schatten_norm_of_Ts-Ts-1}, and norm $\leq 2$ from $\LL^1$ to $\LL^1$ and $\LL^\infty$ to $\LL^\infty$. By the Riesz--Thorin theorem it follows that it has norm $\leq 4 q^{-\theta s}$ from $\LL^r(\L,\proba_o)$ to $\LL^r(\P,\proba_o)$. So the left-hand side of \eqref{eq:cohom_average} is less than 
 \[
 \sum_{o\in X} 4 q^{-\theta s} \|\alpha_{o,i}^+\|_{\LL^{r'}(\P)} \|\beta_{o,j}\|_{\LL^{r}(\L)},
 \]
 which by Hölder's inequality is bounded above by 
 \[
 4 q^{-\theta s} (\sum_o \|\alpha_{o,i}^+\|^{r'}_{\LL^{r'}(\P)})^{\frac 1 {r'}}(\sum_o  \|\beta_{o,j}\|_{\LL^{r}(\L)}^{r})^{\frac 1 {r}}.
 \]
 We can bound the left factor by
 \begin{multline*}
   \bigg(\sum_o \|\alpha_{o,i}^+\|_{\LL^{r'}(\P)}^{r'}\bigg)^{\frac{1}{r'}} = \bigg(\sum_{(x,x') \in X^{(1)}} |g(x,x')|^{r'} \sum_o \proba_o(p_{o,i}=x,p_{o,i+1}=x')\bigg)^{\frac{1}{r'}} 
 \\
 \le \bigg(\sum_{(x,x')\in X^{(1)}} |g(x,x')|^{r'} \sum_o \proba_o(p_{o,i}=x) \bigg)^{\frac{1}{r'}} = \|g\|_{r'} = 1
\end{multline*}
and the right factor by
\[
  \sum_o \|\beta_{o,j}\|^r_{L_r(\L)} \le j \|\delta f\|_{r}=j\text{.}
\]
So taking the supremum over $g$ we get
 \[ \|f_{s,i+j-2s}^+ - f_{s+1,i+j-2s-2}^+\|_r \leq 4 j(q^2+q+1) q^{-\theta s} \]
 for every $i\geq j \geq s+1$. 
 
Let us explain how we obtain similar estimates for $+$ replaced by $-$ and $\sim$. To do so, consider $g$ as above, and for every $(o,o') \in X^{(1)}$, denote
\begin{align*}
  \alpha^\simeq_{o,o',i}(p) &= g(p_{o,i},p_{o',i-1})1_{p \sim_{o,1} o'}\\
  \beta^-_{o,o',j}(\ell) &= (f(\ell_{o,j})-f(o))1_{\ell_{o,1} = o'}\\
  \beta^\sim_{o,o',j}(\ell) &= \frac{1}{q} (f(\ell_{o,j})-f(o))1_{\ell_{o,1} \neq o'}\text{.}
\end{align*}
Then \eqref{eq:double_counting_-} and \eqref{eq:double_counting_sim} in Proposition \ref{prop:double_counting_diamond} for $\Gamma$ the trivial group give for $\truc \in \{-,\sim\}$
\[ \sum_{(o,o') \in X^{(1)}} \langle \alpha^\simeq_{o,o',i},(T_{o,s} - T_{o,s+1})\beta^\truc_{o,o',j}\rangle =\frac{1}{q^2+q+1} \langle g, f_{s,i+j-2s}^\truc - f^\truc_{s+1,i+j-2s-2}\rangle.\]
The left-hand side is bounded above by
\[ 
\bigg(\sum_{(o,o') \in X^{(1)}}  \|\alpha^\simeq_{o,o',i}\|_{L^{r'}}^{r'}
\bigg)^{\frac 1 r'} \bigg(\sum_{(o,o') \in X^{(1)}}  \|\beta^\truc_{o,o',i}\|_{L^{r}}^r
\bigg)^{\frac 1 r}.\]
This first factor is equal to $\|g\|_{r'} \leq 1$, whereas the second term is bounded above by $j\|\delta f\|_r=j$, so taking the supremum over $g$ we obtain
\[ \|f_{s,i+j-2s}^\truc - f_{s+1,i+j-2s-2}^\truc\|_r \leq 4 j (q^2+q+1)q^{-\theta s}\]
for every $i\geq j \geq s+1$ and $\truc \in \{\sim,-\}$. Taking $j=s+1$ and denoting $t=i+j-2s$, we obtain \eqref{eq:ftruc_in_ellr}. 

By similar arguments, using Proposition~\ref{prop:double_counting_diamond_left} instead of Proposition~\ref{prop:double_counting_diamond}, denoting $\ltruc f_\lambda=\ltruc A_\lambda f$, we can prove that 
\[
\|\ltruc f_{s,t}-\ltruc f_{s+1,t-2}\|_r \leq 4 (s+1) (q^2+q+1)q^{-\theta s} 
\]
for every $s\geq 1,t\geq 2$ and $\truc \in \{+,\sim,-\}$.

Applying the previous equation in the building $\overline X$, we get that for every $t\geq 1$, $s \geq 2$ and $\truc \in \{+,\sim,-\}$.
\[
\| \ltruc {\overline{f}_{t,s}}-\ltruc{\overline {f}_{t+1,s-2}}\|_r \leq 4 (s+1) (q^2+q+1) q^{-\theta s}.
\]
By Lemma~\ref{lemma:spheres_opposite_building}, this is exactly \eqref{ftruc_dual_direction}.
\end{proof} 

\begin{proof}[Proof of Lemma~\ref{lem:cv_of_Alambdaf_ellr}]

By Lemma \ref{lem:lafforgue} and the inequalities \eqref{eq:ftruc_in_ellr} and \eqref{ftruc_dual_direction} of Lemma~\ref{lem:10.3}, we obtain that, for each $k \in \Z/3\Z$, the nets $f_\lambda^\truc$ for $\lambda \in \Lambda_k$ are Cauchy for the $\ell^r$-distance and therefore converge to some $f_{\infty,k}^\truc$ in $\ell^r$-distance.
 
On the boundary rays $\N \times \{0\}$ and $\{0\} \times \N$ of $\Lambda$, we have the exceptional identities $f_{(s,0)}^\sim = f_{(s,0)}^+$ and $f_{(0,s)}^\sim = f_{(0,s)}^-$ (Lemma~\ref{lem:important_fact}), so taking $s \to \infty$ we obtain $f_{\infty,k}^\sim = f_{\infty,k}^+= f_{\infty,k}^-$. Denote this common function $X^{(1)} \to \C$ by $f_{\infty,k}$.

Let $(x,x') \in X^{(1)}$. By Lemma \ref{lem:relation_Alambda_Alambdatruc} we can write
\[ A_\lambda f(x) = \frac{1}{q^2+q+1}\left(q^2 A_\lambda^+f(x,x') + q A_\lambda^\sim f(x,x')+ A_\lambda^-f(x,x')\right).\]
From this we deduce
\begin{equation}\label{eq:convergence_Alambdafx}
f_{\infty,k}(x,x') = \lim_{\lambda \in \Lambda_k, |\lambda|\to \infty} A_\lambda f(x).
\end{equation}

On the other hand, using Lemma \ref{lem:relationAlambda_ltruc}, we can write $A_\lambda f(x')$ as a convex combination of $A_\mu^\truc f(x,x')$ (for $\mu$ such that $\mu^\truc = \lambda$). In particular since such $\mu$ belongs to $\Lambda_{k+1}$ if $\lambda$ belongs to $\Lambda_k$, we obtain 
\begin{equation}\label{eq:convergence_Alambdafxprime}
f_{\infty,{k+1}}(x,x') = \lim_{\lambda \in \Lambda_k, |\lambda|\to \infty} A_\lambda f(x').
\end{equation}
The lemma follows by comparing \eqref{eq:convergence_Alambdafx} for $k$ and \eqref{eq:convergence_Alambdafxprime} for $k+2$.
\end{proof}

\section{Application to approximation properties of operator algebras}\label{sec:operator_algebras}

This section contains applications of the tools developped in Sections~\ref{sec:hjemlslev_biaffine},~\ref{sec:norm_estimates} and \ref{sec:averaging_basepoint} to operator algebras. We refer to \S~\ref{sec:OSpreliminaries} and the references therein for the terminology.

Throughout the section we let $\Gamma$ be an $\tilde{A}_2$-lattice in the sense of Definition~\ref{def:A2lattice}. We stress that, unlike to the results from the previous sections, we need not assume that $\Gamma$ be undistorted. The following results generalize the results from \cite{LafforguedlS} to $\tilde{A_2}$-buildings that are not Bruhat--Tits.

\begin{theorem}\label{nonAP_for_A2groups}
For $r \in (1,\infty) \setminus [\frac 4 3,4]$ the space $L^r(\VN(\Gamma))$ does not have the operator space approximation property. For $r=\infty$ the space $C^*_{\mathrm{red}}(\Gamma)$ does not have the operator space approximation property, and $\VN(\Gamma)$ does not have the weak-* operator space approximation property.
\end{theorem}

The case $r=\infty$ means that $\Gamma$ does not have the Approximation Property of Haagerup and Kraus \cite{HaagerupKraus}. In particular this proves that $\tilde{A}_2$-lattices are not weakly amenable, answering a question of Ozawa \cite{OzawaSurvey}. Recall that a discrete group is weakly amenable if $C^*_{\mathrm{red}}(\Gamma)$ has the completely bounded approximation property.

\begin{corollary}
  $\tilde{A}_2$-lattices are not weakly amenable.
\end{corollary} 

By passing to a finite index subgroup, we can and do assume throughout that the action of $\Gamma$ on $X$ is type-preserving.

Recall that, if $I$ is a set, a function $\varphi \colon I \times I \to \C$ is called a \emph{completely bounded Schur multiplier} on $S^r$ if there is a real number $C$ such that, for every integer $n$ and every $i_1,\dots,i_n,j_1,\dots,j_n \in I$, the linear map on $S^r_n$ sending $(a_{k,l})_{1 \leq k,l \leq n}$ to  $(\varphi(i_k,j_l)a_{k,l})_{1 \leq k,l \leq n}$ has norm $\leq C$. The smallest such $C$ is called the \emph{completely bounded norm} of the multiplier $\varphi$ on $S^r$ and denoted $\|\varphi\|_{r\text{-cb}}$.

A function $\varphi \colon \Gamma \to \C$ is called a \emph{completely bounded Herz--Schur multiplier} on $S^r$ if the map $(\gamma,\gamma') \mapsto \varphi(\gamma^{-1}\gamma')$ is a completely bounded Schur multiplier on $S^r$. The corresponding completely bounded norm is also denoted $\|\varphi\|_{r\text{-cb}}$.

The main technical ingredient to Theorem~\ref{nonAP_for_A2groups} is
the following. Recall that, for $\lambda\in \Lambda$, we defined $\mu_\lambda$ by  
\[
\mu_\lambda \defeq \frac{1}{Z} \sum_{x \in \Omega} \frac{1}{\abs{\Gamma_x}} \frac{1}{\abs{S_\lambda(x)}} \sum_{y \in S_\lambda(x)} m_{y}\text{,}
\]
where $Z = \sum_{x \in \Omega} \frac{1}{\abs{\Gamma_x}}$ and denote by $m_{y}$ the uniform probability measure on the finite set of all $\gamma \in \Gamma$ such that  $y\in \gamma\Omega$. 

\begin{proposition}\label{prop:limit_of_multipliers} Let $r \in [1,\infty]\setminus[\frac 4 3,4]$. There is a constant $C_r$ such that the following holds. Let $\varphi \colon \Gamma \to \C$ be a completely bounded Herz--Schur multiplier on $S^r$. There exists a complex number $\varphi_\infty$ such that for every $\lambda = (i,j) \in \Lambda_0$, if $N(\lambda)$ denotes $\max(\frac{i+2j}{3},\frac{j+2i}{3})$
\[
\dabs{\int \varphi \ud\mu_\lambda - \varphi_\infty} \leq C_r q^{-N(\lambda)(\frac 1 2 - \frac 2 r)} \|\varphi\|_{r\text{-cb}}.
\]
\end{proposition}
It is not hard to deduce from the proposition that $L^r(\VN(\Gamma))$ and $C^*_{\mathrm{red}}(\Gamma)$ do not have CBAP. With the results of \cite{LafforguedlS} the proposition also implies that they do not have the OAP. 

We cannot resist to give a more direct derivation. This also provides a simpler proof of the results of \cite{LafforguedlS}, in the spirit of \cite{Vergara} (see also \cite{dlS_habil}):

\begin{proof}[Proof of Theorem~\ref{nonAP_for_A2groups}]
Denote by $V$ the space $L^r(\VN(\Gamma))$ (if $1 \leq r<\infty$) or $C^*_{\mathrm{red}}(\Gamma)$ (if $r=\infty$). The case of $\VN(\Gamma)$ is similar and left to the reader (it also follows from the case of $C^*_{\mathrm{red}}(\Gamma)$ by \cite[Theorem 2.1]{HaagerupKraus}). For $\gamma \in \Gamma$ we use the same symbol to denote the corresponding element of $V$ (this is often denoted $\lambda(\gamma)$ but doing so might cause confusion in our context).

For every $\lambda \in \Lambda_0$ and every linear map $T \colon V \to V$, denote by $u_\lambda(T) \in \C$ the complex number
\[
u_\lambda(T) = \int \tau( \gamma^{-1} T(\gamma)) \ud\mu_\lambda
\]
where $\tau \colon V \to \C$ is the usual trace on $C^*_{\mathrm{red}}(\Gamma)$ and $\mu_\lambda$ is as after Theorem~\ref{thm:strong_T_on_building}. There is no problem with the definition, as the integral is in fact a converging sum and $u_\lambda$ corresponds to the element $\sum_\gamma \mu_\lambda(\gamma) \gamma \otimes \gamma^{-1} \in V \otimes V^*$. We claim that $(u_\lambda)_{\lambda \in \Lambda_0}$ is a Cauchy net in the dual space of $\CB(V)$, and that the limit $u_\infty$ is stably point-norm continuous. The claim immediately implies the theorem as $u_\infty$ separates the identity $V \to V$, on which it takes the value $1$, from the finite rank operators, on which it is identically $0$ (as $\lim_{\gamma \to \infty} \tau( \gamma^{-1} T(\gamma)) =0$ for a finite rank map).

To prove the claim, we need the concept of operator space projective tensor product, see \cite{PisierOS,EffrosRuan}. Every element of $V \hat{\otimes} V^*$ defines a stable point-norm continuous linear form on $\CB(V)$.
So it is enough to show that the net $u_\lambda$ is Cauchy in $V \hat{\otimes} V^*$. By the fundamental characterization of the operator space projective tensor product \cite[Theorem 4.1]{PisierOS}, the dual of $V \hat{\otimes} V^*$ is $\CB(V,V^{**})$, so we need to show that $u_\lambda$ is Cauchy in $\CB(V,V^{**})^*$. For that, we extend the definition of $u_\lambda$ to operators $V \to V^{**}$ by defining, for $T \colon V \to V^{**}$
\begin{enumerate}
\item the function $\varphi_{T}\colon \Gamma \to \C$ by $\varphi_T(\gamma) = \langle T(\gamma), v_\gamma\rangle$ where $v_\gamma \in V^*$ is defined by $v_\gamma(x) = \tau(\gamma^{-1}x)$, and
\item $u_\lambda(T) = \int \varphi_T d\mu_\lambda$.
\end{enumerate}
It follows by a very well-known trick of Haagerup that
\begin{equation}\label{HaagerupTrick}\| \varphi_T\|_{r\text{-cb}} \leq \|T\|_{\CB(V,V^{**})}.\end{equation} Indeed, for every integer $n$ and $s_1,\dots,s_n,t_1,\dots,t_n \in \Gamma$, we can consider the map $(a_{k,l})_{1 \leq k,l \leq n} \in M_n(\C) \mapsto (a_{k,l} s_k^{-1}t_l)_{1 \leq k,l \leq n} \in M_n(C^*_{\mathrm{red}}(\Gamma))$. It is a trace-preserving *-homomorphism. So for every $1 \leq r\leq \infty$, it extends to an isometric embedding $\iota_r \colon S^r_n \to S^r_n[V]$ with a norm $1$ conditional expectation $E_r \colon S^r_n[V] \to S^r_n$ sending $(x_{k,l})_{k,l}$ to $(v_{s_k^{-1}t_l}(x_{k,l}))_{k,l}$. Here $S^r_n[V]$ denotes $M_n(C^*_{\mathrm{red}}(\Gamma))$ if $r=\infty$ and $L^r(M_n( \VN(\Gamma)), \mathrm{Tr} \otimes \tau)$ if $1\leq r \leq p$. The notation is compatible with the notion of vector-valued non-commutative $L^r$ spaces \cite{PisierOS}. One computes that $E_r^{**} \circ T \circ \iota_r ((a_{k,l})_{1 \leq k,l \leq n}) = (\varphi(s_k t_l^{-1})a_{k,l})_{1 \leq k,l \leq n}$. So we obtain 
\[ \left\|(\varphi(s_k t_l^{-1})a_{k,l})_{1 \leq k,l \leq n}\right\|_{S^r_n} \leq \|T\|_{\cb} \left\|(a_{k,l})_{1 \leq k,l \leq n}\right\|_{S^r_n},\]
which proves \eqref{HaagerupTrick}. Proposition \ref{prop:limit_of_multipliers} therefore expresses that the maps $u_\lambda$ satisfy
\[|(u_\lambda - u_{\lambda'})(T)| \leq C_r( q^{-N(\lambda)(\frac 1 2 - \frac 2 r)} + q^{-|\lambda'|(\frac 1 2 - \frac 2 r)}) \|T\|_{\CB(V,V^{**})}.\]
 So $u_\lambda$ is Cauchy in $\CB(V,V^{**})^* = (V \hat{\otimes} V^*)^{**}$.
 \end{proof}

Proposition \ref{prop:limit_of_multipliers} will follow rather easily from the analogous statement on the building.

\begin{proposition}\label{prop:limit_of_multipliers_on_X}
Let $r \in [1,\infty]\setminus[\frac 4 3,4]$. There is $C_r>0$ such that the following holds. Let $\psi \colon X \times X \to \C$ be a $\Gamma$-invariant completely bounded Schur multiplier on $S^r$. For every $x \in \Omega$, there exists a complex number $\psi_\infty(x)$ such that for every $\lambda \in \Lambda_0$, 
 \[ \sum_{x \in \Omega} \frac{1}{\abs{\Gamma_x}}\dabs{\frac{1}{S_\lambda(x)} \sum_{y \in S_\lambda(x)} \psi(x,y) - \psi_\infty(x)} \leq C_r q^{-N(\lambda)(\frac 1 2 - \frac 2 r)} \|\psi\|_{r-cb}.\]
\end{proposition}

Let us first explain how Proposition~\ref{prop:limit_of_multipliers_on_X} implies Proposition~\ref{prop:limit_of_multipliers} follows.

\begin{proof}[Proof of Proposition~\ref{prop:limit_of_multipliers}]
We start with the induction procedure. We fix a fundamental domain $\Omega \subset X$ and define a function $\psi \colon X \times X \to \C$ by
\begin{align*}
\psi(x,y) &= \frac{\sum_{\gamma_1\in \Gamma, \gamma_1 x \in \Omega} \sum_{\gamma_2 \in \Gamma,\gamma_2 y \in \Omega} \varphi(\gamma_1 \gamma_2^{-1})}{ |\Gamma_x| |\Gamma_y|} \\& = \int \varphi(\gamma_1^{-1} \gamma_2) \ud m_x(\gamma_1) \ud m_y(\gamma_2)
\end{align*}
Then $\psi$ is $\Gamma$-invariant, and it is a completely bounded Schur multiplier on $S^r$ with norm
\[
\|\psi\|_{r\text{-cb}} \leq \|\varphi\|_{r\text{-cb}}.
\]
If we define $\varphi_\infty := \frac{1}{Z} \sum_{x \in \Omega}\frac{1}{\abs{\Gamma_x}} \psi_\infty(x)$ where $Z = \sum_{x \in \Omega}\frac{1}{\abs{\Gamma_x}}$, then Proposition \ref{prop:limit_of_multipliers_on_X} yields
\[
\dabs{\int \varphi dm_\lambda - \varphi_\infty} \leq C_r q^{-N(\lambda)(\frac 1 2 - \frac 2 r)} \|\varphi\|_{r\text{-cb}}
\]
as desired.
\end{proof}

\begin{proof}[Proof of Proposition~\ref{prop:limit_of_multipliers_on_X}]
By duality we can assume that $r>4$. Fix $s>1$. For $\lambda \in \Lambda$, denote $\psi_{\lambda}(x) = \frac{1}{S_\lambda(x)} \sum_{y \in S_\lambda(x)} \psi(x,y)$. We shall apply Proposition \ref{prop:double_counting} to the function $\psi$. Denote $\lambda = (s,i+j-2s)$ and $\lambda'=(s-1,i+j-2s+2)$. Consider the quantity $\Delta_{i,j}(s)$ defined as the difference of \eqref{eq:double_counting} between the values at $s$ and $s-1$. We first evaluate $\Delta_{i,j}(s)$ by looking at the right-hand side of \eqref{eq:double_counting}. We obtain
 \begin{align*} 
 \Delta_{i,j}(s) &= \sum_{x \in \Omega} \frac{1}{|\Gamma_x|} \left(\frac{1}{|S_\lambda(x)|} \sum_{y \in S_\lambda(x)} \psi(x,y) - \frac{1}{|S_{\lambda'}(x)|} \sum_{y \in S_{\lambda'(x)}} \psi(x,y)\right) \\
  & = \sum_{x \in \Omega} \frac{1}{|\Gamma_x|} (\psi_\lambda(x) - \psi_{\lambda'}(x)).
  \end{align*}
We then evaluate $\Delta_{i,j}(s)$ by looking at the left-hand side of \eqref{eq:double_counting}. Let $o \in \Omega$. The operator $T_{o,s}$ is defined $L^2(\P) \to L^2(\L)$ but its image is contained in $L^2(\L_s)$. So by (co)restriction we can view it as an operator $L^2(\P_i) \to L^2(\L_j)$, or as a matrix indexed by $\L_j \times \P_i$. Then we can write $\E_o[\psi(p_{o,i},\ell_{o,j})\mid (p,\ell)_o=s]$ as the trace of $M_o(T_{o,s}) E$ where:
 \begin{enumerate}
 \item $E \colon L^2(\L_j) \to L^2(\P_i)$ is the rank one and norm one operator sending a function $f$ on $\L_j$ to the constant function on $\P_i$ equal to the average of $f$, and
 \item $M_o$ is the Schur multiplier with symbol $\psi(p_{o,i},\ell_{o,j})$.
 \end{enumerate}
 In particular, we obtain
 \[ \Delta_{i,j}(s) = \sum_{o \in \Omega}\frac{1}{|\Gamma_o|} \mathrm{Tr}(M_o(T_{o,s} - T_{o,s-1}) E),\]
which together with Hölder's gives
 \[|\Delta_{i,j}(s)| \leq \sum_{o \in \Omega}\frac{1}{|\Gamma_o|} \|M_o(T_{o,s} - T_{o,s-1})\|_{S^r} \|E\|_{S^{r'}}.\]
 But $\|E\|_{S^{r'}} = 1$, and using Proposition \ref{prop:Schatten_norm_of_Ts-Ts-1}
 \[
 \|M_o(T_{o,s} - T_{o,s-1})\|_{S^r} \leq \| \psi\|_{r-cb} \|T_{o,s} - T_{o,s-1}\|_{S^r} \leq 4q^{\frac 3 r} q^{s(\frac 1 2 - \frac 2 r)} \|\psi\|_{r\text{-cb}} .
 \]
So all in all, if we denote $C_r = 4q^{\frac 3 r} \sum_{o \in \Omega}\frac{1}{|\Omega_o|}$, we obtain
 \[ |\sum_{x \in \Omega} \frac{1}{\abs{\Gamma_x}}(\psi_\lambda(x)-\psi_{\lambda'}(x) | \leq C_r |\Gamma_o| q^{s(\frac 2 r-\frac 1 2)} \|\psi\|_{r\text{-cb}}.\]
 
Now consider, for any $\Gamma$-invariant map $\omega \colon X \to \{z \in \C \mid \abs{z}=1\}$, the function $\psi^{\omega}(x,y) = \omega(x) \psi(x,y)$. It is still equivariant and satisfies $\|\psi^\omega\|_{r-cb} = \|\psi\|_{r-cb}$. Applying the preceding inequality to $\psi^\omega$, we obtain 
\[
|\sum_{x \in \Omega} \frac{\omega_x}{\abs{\Gamma_x}}(\psi_\lambda(x)-\psi_{\lambda'}(x)) | \leq C_r q^{s(\frac 2 r-\frac 1 2)} \|\psi\|_{r\text{-cb}}.
\]
And taking the supremum over $\omega$ we get
\[
\sum_{x \in \Omega} \frac{1}{\abs{\Gamma_x}}|\psi_\lambda(x)-\psi_{\lambda'}(x)| \leq C_r q^{s(\frac 2 r-\frac 1 2)} \|\psi\|_{r\text{-cb}}.
\]
Exchanging the roles of points and lines, we obtain the same inequalities for $\lambda = (i+j-2s,s)$ and $\lambda'=(i+j-2s+2,s-1)$. The proposition now follows from Lemma~\ref{lem:lafforgue} with $D=\frac 1 2 - \frac 2 r$ and $E=0$.
\end{proof}

\bibliographystyle{smfplain}
\bibliography{refs.bib}

\end{document}